\documentclass[11pt,a4paper]{amsart}
\usepackage[toc,page]{appendix}

\renewcommand{\appendixpagename}{Appendix}
\usepackage{a4}
\usepackage[active]{srcltx}
\usepackage{amsmath,bbm,esint}
\usepackage{amssymb}
\usepackage{accents,calc}
\usepackage{subfig}
\usepackage{caption}
\usepackage{dsfont}
\usepackage{mwe}
\usepackage{xcolor}
\usepackage{comment}
\usepackage{tikz}
\usepackage{amsfonts}
\usepackage{dsfont}
\usepackage{enumitem}\usepackage{mathtools} 
\mathtoolsset{showonlyrefs}

\usepackage{xcolor}
\usepackage[toc,page]{appendix}
\makeatletter
\renewcommand{\@chap@pppage}{%
  \clear@ppage
  \thispagestyle{plain}%
  \if@twocolumn\onecolumn\@tempswatrue\else\@tempswafalse\fi
  \null\vfil
  \markboth{}{}%
  {\centering
    \interlinepenalty \@M
    \normalfont
    \Large\bfseries \appendixpagename\par}
  \if@dotoc@pp
    \addappheadtotoc
  \fi
  \vfil\newpage
  \if@twoside
    \if@openright
      \null
      \thispagestyle{empty}%
      \newpage
    \fi
  \fi
  \if@tempswa
    \twocolumn
  \fi
}

\renewcommand{\@sec@pppage}{%
  \par
  \addvspace{4ex}%
  \@afterindentfalse
  {\parindent \z@
    \interlinepenalty \@M
    \normalfont
    \centering
    \LARGE\bfseries \appendixpagename
    \markboth{}{}\par}%
  \if@dotoc@pp
    \addappheadtotoc
  \fi
  \nobreak
  \vskip 3ex
  \@afterheading
}
\makeatother

\renewcommand{\appendixpagename}{Appendix}
\newenvironment{myappendices}{%
  \begin{appendices}%
    \addtocontents{toc}{\protect\setcounter{tocdepth}{0}}%
}{%
    \addtocontents{toc}{\protect\setcounter{tocdepth}{2}}%
  \end{appendices}%
}

\usepackage[normalem]{ulem}

\usepackage{mathrsfs} 
\usepackage[colorlinks=true,urlcolor=blue,
citecolor=red,linkcolor=blue,linktocpage,pdfpagelabels,bookmarksnumbered,bookmarksopen]{hyperref}

\usepackage{xcolor}

\usepackage{collectbox}



\def\e{\varepsilon}

\def \p{\partial}

\def \sh{\textnormal{SH}}

\def \0{\mathbf{0}}

\newcommand{\cP}{{\mathcal P}}
\newcommand{\cS}{{\mathcal{S}}}
\newcommand{\cN}{{\mathcal N}}

\newcommand{\R}{{\mathbb R}}

\usepackage{subfig}

\newcommand{\be}{\begin{equation}}
\newcommand{\ee}{\end{equation}}

\newcommand{\ba}{\begin{array}}
\newcommand{\ea}{\end{array}}

\usepackage{subfiles}

\newtheorem{theorem}{\textbf{Theorem}}[section]
\newtheorem{remark}[theorem]{\textbf{Remark}}
\newtheorem{lemma}[theorem]{\textbf{Lemma}}

\newtheorem{corollary}[theorem]{\textbf{Corollary}}

\newtheorem{proposition}[theorem]{\textbf{Proposition}}

\newtheorem{definition}[theorem]{\textbf{Definition}}

\newtheorem{innercustomgeneric}{\customgenericname}
\providecommand{\customgenericname}{}
\newcommand{\newcustomtheorem}[2]{%
  \newenvironment{#1}[1]
  {%
   \renewcommand\customgenericname{#2}%
   \renewcommand\theinnercustomgeneric{##1}%
   \innercustomgeneric
  }
  {\endinnercustomgeneric}
}

\newcustomtheorem{customthm}{Framework}

\numberwithin{equation}{section}

\begin{document}

\title[Fractal freezing and fine structure of maximal solutions]{The supercooled Stefan problem: fractal freezing and the fine structure of maximal solutions}

\author{Raymond Chu, Inwon Kim, Sebastian Munoz}

\begin{abstract}  We study the supercooled Stefan problem in arbitrary dimensions. First, we study general solutions and their irregularities, showing generic fractal freezing and nucleation, based on a novel Markovian gluing principle. In contrast, we then establish regularity properties of \emph{maximal} solutions, which are obtained by maximizing a suitable notion of ``average" freezing time. Unexpectedly, we show that maximal solutions have a transition zone that is open modulo a low-dimensional set: this allows us to apply obstacle problem theory for a finer regularity analysis. We further show that  maximal solutions are in general non-universal, and we obtain sharp stability results under perturbation of each maximal solution. Lastly, we study maximal solutions in both the radial and the one-dimensional setting. We show that in these cases the maximal solution is universal and minimizes nucleation, in agreement with phenomena observed in the physics literature.
\end{abstract}

\maketitle

\tableofcontents

\section{Introduction}
We study the supercooled Stefan problem
\begin{equation} \label{eq:stefan intro}
    \begin{cases} \eta_t-\frac12\Delta \eta=(\chi_{\{\eta>0\}})_t & \text{in }\;(0,\infty)\times \R^d,\\
    \eta \geq 0,\,\, \eta(0,\cdot)=\mu\quad &\text{on} \;\;\R^d,\\    
     
    \end{cases}
\end{equation}
where $0\leq \mu \in L^{\infty}(\R^d)$ has compact support.  The system \eqref{eq:stefan intro} models the freezing of supercooled water into ice, where $\eta$ denotes the negative temperature of the water \cite{rubins1971}. Initially the water is at temperature $-\mu$, the ice is kept at temperature $0$ throughout the evolution, and the transition from water to ice occurs at zero temperature. Thus $\{\eta>0\}$ is the region occupied by the water, and the {\it free boundary} $\partial\{\eta>0\}$ denotes the interface between ice and water. Letting $U:=\{\mu>0\}$ be the space occupied by water at the initial time, we require $\{\eta(t,\cdot)>0\} \uparrow U$ as $t\downarrow 0$, so that $(\mu,U)$ serves as the initial data for the free boundary problem; this will be made precise in Definitions \ref{def:weak solution}.  

Unlike the classical Stefan (melting) problem, the supercooled Stefan problem is ill-posed and unstable \cite{ivantsov1988}, as is typical for solidification from a supercooled state. Even in one dimension, classical solutions can develop discontinuities where the positive set jumps in finite time \cite{sherman1970,herrero96,dembo2019}. When the free boundary is constrained to be a single point, the solution can be continued past the jump \cite{fasano1989singularities}, and there is a unique continuation that minimizes the jump size \cite{DNS22,Mun25}. 

Beyond the single-free-boundary-point case, mostly formal/numerical work and well-prepared initial data have been treated \cite{howison1985,herrero2000,chen1997,dib1984,kim2024stefan}. 

The main analytical challenge is the rich geometry of potential singularities that can form. For general initial data, even global-in-time existence was open until the recent work \cite{choi2024existence},  which also introduced the notion of {\it maximal} solutions.

Our aim is to analyze \eqref{eq:stefan intro} for general 
$\mu$, beginning with weak solutions and then focusing on the class of maximal solutions. Akin to \cite{choi2024existence}, our analysis is motivated by the underlying particle dynamics and combines methods from probability and PDE: we present some of the main ideas in Section \ref{subsection:main ideas}.

To state our results, we first define  the largest class of initial data for which global-in-time weak solutions can exist, formulated in terms of the \emph{subharmonic order} $\leq_{\operatorname{SH}}$ between measures (see Definitions \ref{def:subharmonic ordering}--\ref{def:subharmonic ordering relative}). Roughly speaking, $\mu\leq_{\operatorname{SH}} \nu$ with respect to $U$ if there exists a stopped Brownian motion in $U$ with initial distribution $\mu$ and final distribution $\nu$.

\begin{definition} \label{def:Sdelta} Let $U$ be a bounded open subset of $\R^d$, and let $\delta \in [0,1)$. We say that $\mu\in \cS_{\delta}(U)$ if $\mu\in L^{\infty}(\R^d)$ is non-negative, $\{\mu>0\}=U$, and 
\begin{equation}
    \mu \leq_{\sh}\nu\;\;\textnormal{with respect to} \;U,\textnormal{ for some } \nu \;\textnormal{with }\nu\leq (1-\delta)\chi_{U}.
\end{equation}
\end{definition}
It is known that a necessary condition for \eqref{eq:stefan intro} to have a solution is $\mu \in \cS_{0}(U)$ (see \cite[Thm. 7.2]{choi2024existence}).  
Conversely, \cite{choi2024existence} shows existence whenever $\mu \in \cS_{\delta}(U)$ for some $\delta>0$, $U$ is Lipschitz,  and $|\{\mu=1\}|=0$; we discuss how the Lipschitz assumption can be removed in Proposition \ref{prop: existence} and Remark \ref{rem:existence general}. 

Clearly $\mu \leq 1-\delta$ implies $\mu \in \cS_{\delta}(U)$, but the converse need not hold.
Heuristically, elements of $\cS_{\delta}(U)$ may take large values at interior points far from the boundary of $U$, where there is more room to diffuse the mass by Brownian motion. The importance of the critical supercooling threshold $\mu=1$ has been long known in the physics literature \cite{CaCaRo92,Pel88}, and will be observed here in several contexts.

Given any weak solution $\eta$ to \eqref{eq:stefan intro}, the sets $\{\eta(t, \cdot)>0\}$ are monotone nonincreasing in $t$, and $\eta$ has a unique lower semicontinuous representative that ensures the set $\{\eta>0\} \subset (0,\infty)\times U $ is open. One may then unambiguously decompose 
\begin{equation} \label{eq:U decomp intro}  
U=F_0\dot\cup\bigcup_{t>0}\{\eta(t, \cdot)>0\},
\end{equation}
where $|F_0|=0$, and $F_0$ is the set of points that freeze at the initial time. More generally,  if one defines the \emph{freezing time} $s \in \operatorname{LSC}(\R^d; [0,\infty])$ by
\begin{equation} \label{eq:s defi intro}
    s(x):=\inf\{t>0:\eta(t,x)=0\}, 
\end{equation}
then $F_0=\{x\in U:s(x)=0\}$ and, for $t_0>0$, $\{s=t_0\}$  is the set of points that freeze at time $t_0$, which equals $\partial\{\eta>0\}\cap\{t=t_0\}$ if $s$ is continuous (see, however, Proposition \ref{prop:waiting time}).  We also define the \emph{transition zone} of $\eta$ as the set of points that freeze in positive time,
\begin{equation} \label{def:transition zone intro}
    \Sigma=\{0<s(x)<\infty\} \subset U.
\end{equation}
Our first result shows that, for general initial data, the freezing set may contain any prescribed fractal. 

\begin{theorem}[Programmed fractal freezing] \label{thm: fractal freezing} Let $U$ be a bounded open set, and let $\mu \in \cS_{\delta}(U)$ for some $\delta \in(0,1)$, with $|\{\mu=1\}|=0$.  Then the following hold:
\begin{enumerate}
    \item[(i)] Suppose $\mu\leq 1-\delta$. Then for any closed null set $F\subset U$,  there exists a weak solution $\eta$ to \eqref{eq:stefan intro} such that 
    \begin{equation*}
        F\subset \{x:s(x)=0\}.
    \end{equation*}
    \item[(ii)] Let $\eta_0$ be a weak solution to \eqref{eq:stefan intro}. Then there exists $0\leq T_0=T_0(\|\mu\|_{\infty})<\infty$ such that the following holds: for any $T> T_0$ and  for any closed null set $F\subset \{\eta_0(T,\cdot)>0\}$, there is a weak solution $\eta$ to \eqref{eq:stefan intro} with corresponding freezing time $s$, such that 
    \begin{equation}\label{eq:fractal freeze t0 intro}
        \eta(t,\cdot)=\eta_0(t,\cdot), \quad 0\leq t<T, \quad \text{and }\;\;     F \subset\{x: s(x) =T\}.
    \end{equation} 
Moreover, if $\mu\leq 1$, then $T_0=0$. In particular, for any $T>0$, there exists a weak solution of \eqref{eq:stefan intro} with the property
\begin{equation} \label{eq:freeze dim d t0 intro}
    \operatorname{dim}_{\mathcal{H}}(\{x\in U: s(x)=T\})=d.
\end{equation} 
\end{enumerate}   
\end{theorem}

In physical terms, Theorem \ref{thm: fractal freezing} can be viewed as a timed seeding principle: a brief intervention that releases or activates nucleants along a prescribed thin set $F$, forcing a flash-freeze at $F$. The set $F$ may be space-filling in the fractal sense, which corresponds to the formation of crystals of arbitrary complexity. This also highlights the extreme non-uniqueness of the model in the absence of regularizing physics such as surface tension.  In a similar vein, we also show (see Proposition \ref{prop:t=0 nucleation}) that, even with smooth initial data, solutions may exhibit \emph{nucleation} (spontaneous appearance of ice) at the initial time, as well as an interior waiting time phenomenon.

 Given this ill-posedness, it is natural to ask whether one may select a canonical, physically meaningful notion of solution that admits a well-posedness theory and improved regularity properties. While this has been explored extensively in the one-dimensional case of a single free boundary point \cite{DNS22,Mun25}, very little is known in a more general setting.

To explore these questions, we focus for the remainder of the paper on \emph{maximal} solutions of \eqref{eq:stefan intro}, first introduced in \cite{choi2024existence} (see Definition~\ref{def:maximal}). Informally, these solutions  maximize their time to freeze, in a sense analogous to the one-dimensional physical solutions of \cite{DNS22}. Maximal solutions arise from the stochastic optimization problem 
\begin{equation} \label{eq:primal intro}
\mathcal{P}(\mu,u):= \inf_{\nu} \left\{ \int_{\mathbb{R}^d} u(x)\,d\nu(x) :\; W_0 \sim \mu,\; W_\tau \sim \nu \leq \chi_U,\; \tau \leq \tau^U \right\},  
\end{equation}
where $\mu\in \cS_{\delta}(U)$ (and thus $U=\{\mu>0\}$), $W_t$ is Brownian motion,  $\tau^U$ denotes its first exit time from $U$, and $u>0$ is a smooth, strictly superharmonic function. Denoting by $\nu$ the unique optimizer, $\nu=\chi_{\Sigma}$ a.e. for some measurable $\Sigma \subset U$, and there exists a unique solution $\eta$ to \eqref{eq:stefan intro} with transition zone $\Sigma$ \cite{choi2024existence}. In fact, if $\tau$ is the Skorokhod--optimal stopping time (see Section \ref{sec:prelim}) such that $W_{\tau}\sim \nu$, then $\eta$ is precisely the active particle distribution of $W_{t\wedge \tau}$ (see Definition \ref{def:eulerian_variables}), and $\eta$ is the unique solution to \eqref{eq:stefan intro} that postpones solidification as much as possible, in the sense that it maximizes
\begin{equation} \label{eq:E maximized intro}
\mathbb{E}\left(\int_{0}^{\tau}-\Delta u(W_t)dt\right)=\int_{0}^{\infty}\int_{\R^d}-\Delta u(x)\eta(t,x)dtdx .
\end{equation} As we will see in Proposition \ref{prop: maximal_sol_characterization} this maximization property uniquely characterizes maximal solutions.
In particular, in the simplest case of $\Delta u \equiv -1$, the maximal solution associated to $u$ maximizes the quantity $\mathbb{E}\left[\tau\right]$, or equivalently $\|\eta\|_{L^1( (0,\infty) \times \R^d)}$. Moreover, these solutions have a maximal transition zone, in the sense of the subharmonic order (see \cite[Thm. 1.2]{choi2024existence}). 

As we will see next, maximal solutions also enjoy substantial regularity and stability properties. We first discuss our regularity results for the transition zone \eqref{def:transition zone intro}. It was conjectured in \cite[Rem. 6.2]{choi2024existence} that the transition zone of maximal solutions is, in general, not an a.e. open set. Our next result proves this conjecture to be false, showing that the transition zone may only contain a low-dimensional subset of its boundary. In what follows, given $S \subset \R^d$, $S^{\circ}$ will denote the topological interior of $S$.

\begin{theorem}[Fine structure of the transition zone] \label{thm: regularity intro} Let $U$ and $\mu$ be as given in Theorem \ref{thm: fractal freezing}, and assume that $U$ has a Lipschitz boundary.  Let $u$ be a smooth, positive, and strictly superharmonic function on $\overline{U}$.  Let $\eta$ be the maximal solution to \eqref{eq:stefan intro} associated to $\mathcal{P}(\mu,u)$, and let the freezing time $s:\R^d \to [0,\infty]$ be defined by \eqref{eq:s defi intro}. Then the transition zone $\Sigma=\{x:0<s(x)<\infty\} $ satisfies
\begin{equation} 
    \Sigma=\Sigma^{\circ}\cup A,
\end{equation}
where $\dim_{\mathcal{H}}(A)\leq d-1$. If $u$ is real analytic, then 
\begin{equation} \label{eq:dim d-2 intro}
    \dim_{\mathcal{H}}(A)\leq d-2.
\end{equation}
One also has, for general $u$,
\begin{equation} \label{eq:dim bd sigma intro}
    \dim_{\mathcal{H}}(\partial\Sigma \cap \{s(x)>0\})\leq d-1, \;\;\quad\quad |\partial \Sigma|=0,
\end{equation}
and the set $\Sigma$ is measure-saturated. 

\end{theorem}

 We refer to Lemma \ref{lem:sigma saturation} for the definition of measure-saturation. Heuristically, this means that: (i) there cannot exist cracks that freeze in finite time surrounded by water that never freezes, and (ii) there cannot exist cracks that never freeze surrounded by water that freezes in finite time. 

Regarding \eqref{eq:dim d-2 intro}, the gain in regularity when $u$ is analytic rather than merely smooth reflects a broader pattern that highlights the inherent instability of the problem (see also the discussion after Theorem \ref{thm:universality} regarding the converse phenomenon).

In combination with the results of the companion paper \cite{EKM25}, Theorem \ref{thm: regularity intro} yields strong regularity results for the freezing time, and precise dimension bounds for the set of singular points that freeze at positive times, outside of which the free boundary is smooth (see Corollary \ref{cor:regu} and Remark \ref{rem:dim par sharp}). In particular, in the subcritical case when $\mu \leq 1$, Corollary \ref{cor:regu} implies that fractal freezing never occurs for maximal solutions at any positive time, in contrast with Theorem \ref{thm: fractal freezing}. Moreover, for general $\mu \in \cS_{\delta}(U),$ maximal solutions do not exhibit fractal freezing for most times, namely
$$
\dim_{\mathcal{H}}(\{t: \dim_{\mathcal{H}}(\{x:s(x)=t\})>d-1\})=0.$$ While Theorem~\ref{thm: fractal freezing} suggests that this estimate is close to sharp, the optimality of the estimate and whether maximal solutions can develop fractal freezing when $\sup\mu>1$ remain open. 

For general maximal solutions, the Hausdorff dimension estimate \eqref{eq:dim bd sigma intro} for $\partial \Sigma$ cannot be improved to include the set of initial freezing  $\{s=0\}$. In fact, for any closed null set $F\subset U$ (in particular fractals with $\dim_{\mathcal{H}}(F)=d$), there exists $\mu \in \cS_{0}(U)\cap C(\overline{U})$ such that every weak solution (including maximal solutions) of \eqref{eq:stefan intro} satisfies $F \subset \{s=0\} \subset \partial \Sigma $ (see Proposition \ref{fractal_freezing_maximal}, compare with Theorem \ref{thm: fractal freezing} (i)). It remains an interesting open question whether for maximal solutions the regularity of $\{s=0\}$ can be improved with the slack condition $\mu \in \cS_{\delta}(U)$, $\delta>0$.

We also note that, even in the more regular setting of maximal solutions with smooth initial data, it may be difficult to understand the geometry of $\Sigma$ beyond  Hausdorff dimension bounds. For instance, while it is easy to show that $\partial U \subset \overline{\Sigma}$, {\it waiting time} phenomena may develop. That is, there may be points near the initial free boundary $\partial U$ that never freeze, and where the freezing time $s$ is discontinuous (see Figure \ref{fig:sigma} for a numerical example). The following construction illustrates this fact. 

\begin{proposition}[Waiting time, Proposition \ref{prop:fourier example}]\label{prop:waiting time}
Let $d = 2$, let $U = B_1 \setminus \overline{B_{\frac{1}{2}}(0)}$, and let $\delta \in (0,1)$. 
There exists a smooth density $\mu \in \cS_{\delta}(U)$ such that, for any choice of $u$, if $\eta$ denotes the maximal solution associated to $\mathcal{P}(\mu, u)$, 
then the corresponding transition zone $\Sigma$ satisfies 
\begin{equation}\label{eq:waiting time intro}|U_\varepsilon \setminus \Sigma| > 0 \quad \text{for every}\quad  \varepsilon>0,\end{equation}
where $U_\varepsilon := \{ x \in U : d(x, \partial U) < \varepsilon \}$. In particular, the freezing time of any such $\eta$ is discontinuous. Moreover, $\mu$ may be chosen to be arbitrarily close, in any $C^k$ norm, to a smooth density $0<\mu_{*}<1$ such that the associated transition zone for $\mu_*$ is the open union of two outer annuli of $U$.
\end{proposition} 
This result also shows that the absence of a waiting time is an unstable phenomenon, since the limiting density $\mu_*$ does not satisfy the pathology \eqref{eq:waiting time intro}. The construction is robust, in the sense that $\mu_*\in (0,1)$ can be taken to be any continuous \emph{radial} density.  

In hopes of identifying a canonical notion of solution,  \cite{choi2024existence} conjectured that maximal solutions are \emph{universal}. Namely, while $\eta$ is uniquely determined by its maximization of \eqref{eq:E maximized intro}, it is natural to ask whether $\eta$ is independent of the choice of the superharmonic weight $u$. The following result shows by counterexample that this conjecture is false in dimensions $d\geq 2$, yet {universality holds in the radial and one-dimensional cases. We refer to Propositions \ref{prop:radial transition zone} and \ref{target_measure_universial} for a precise description of the universal solutions in the latter cases. 

\begin{theorem}[Non-universality of maximal solutions]  \label{thm:universality}In dimension $d\ge2$, and for any $\delta>0$, there exist smooth $\mu \in \cS_{\delta}(B_1)$ and smooth  weights $u_1$ and $u_2$ such that the  maximal solutions $\eta_1$ and $\eta_2$ to \eqref{eq:stefan intro} associated to $\cP(\mu,u_1)$ and $\cP(\mu,u_2)$ are not equal. 

If $\mu \in \cS_0(U)$, $|\{\mu=1\}|=0$, and either $d=1$ or $\mu$ and $u$ are both radial functions, then there exists a unique maximal solution, independent of $u$. In addition, if $\mu \in \mathcal{S}_{\delta}(U)$ for some $\delta>0$, then the  set of initial freezing is empty.
\end{theorem}

The counterexample constructed in 
Theorem \ref{thm:universality} is fairly general, and highlights the converse phenomenon to Theorem \ref{thm: regularity intro}: non-analyticity of $u$ on its own can induce irregularities of the transition zone (see Proposition \ref{no_universality_prop}). In the same spirit, the explicit construction in the proof of Proposition \ref{prop:waiting time} (see Proposition \ref{prop:fourier example}) shows that the waiting time pathology arises due to the insufficiently fast decay of the tangential Fourier coefficients of the smooth density $\mu$, that is, its quantitative non-analyticity (see also \cite[Sec. 7.1]{choi2024existence}).

 In light of Theorem \ref{thm:universality}, a natural follow-up question  is to ask which weight 
$u$ should be regarded as physically canonical; this remains open and likely requires a further understanding of the thermodynamics behind the problem.

Given the lack of universality, the most one can reasonably expect is stability of maximal solutions with respect to perturbations of $\mu$ and $u$; this is the content of our next result. We now specify the notion of convergence for our domains and transition zones.

\begin{definition} \label{def:set convergence}Let $\{S_n\}_{n=1}^{\infty}$ and $S$ be measurable subsets of $\R^d$. We say that $S_n \to S$ as $n\to \infty$ if $\lim_{n\to \infty}|S_n \Delta S|=0$ and, for every open set $W$ such that $\overline{S} \subset W$, there exists $N>0$ such that $\overline{S_n} \subset W$ for $n\geq N$.
\end{definition} 

Now for $\delta>0$, $\mu \in\cS_{\delta}(U)\cap C(\overline{U})$ with a Lipschitz domain $U$,
 consider $\mu_n\in \cS_{\delta}(U_n)\cap C(\overline{U_n})$ such that $U_n$ is Lipschitz. We assume that $(\mu_n, U_n)$ converges to $(\mu, U)$ in the following sense.
\begin{equation}\label{conv:mode}
(\R^d \setminus U_n) \to (\R^d \setminus U), \quad \|\mu_n-\mu\|_{L^{\infty}(U_n \cap U)}\to0\;\text{ as }\; n\to \infty.
\end{equation} 
We further assume the following on the convergence of the weight function, with an open set $D$ containing $\overline{U}$ and $\overline{U_n}$: 
\begin{equation}\label{conv:mode2}
\{u_n\} \subset C^2(D) ,\;\; \Delta u_n < 0 \;\hbox{ in }D, \hbox{ and } u_n \hbox{ converges to some }u \hbox{ in } C^2(D), \;\;\hbox{as } n\to\infty.
\end{equation}
In this setting we have the following stability result.

\begin{theorem}[Stability of maximal solutions, Propositions~\ref{general_stability} and \ref{stable_Stefan_strong}]\label{thm:stability intro}

For $(\mu_n, u_n)$ and $(\mu,u)$ as given above, let $\eta_n$ and $\eta$ be the corresponding maximal solutions to \eqref{eq:stefan intro} associated to $\cP(\mu_n,u_n)$ and $\cP(\mu,u)$, respectively.  Then, as $n\to \infty$,
\begin{equation}
    \Sigma_n \to \Sigma\;\; \textnormal{ and } \;\;  \|\eta_n-\eta\|_{L^p((0,\infty)\times \R^d)} \rightarrow 0 \textnormal{ for all } p \in [1,\infty),
\end{equation}
where $\Sigma_n$ and $\Sigma$ are the corresponding transition zones for $\eta_n$ and $\eta$.

\end{theorem}
The assumptions of Theorem \ref{thm:stability intro} are sharp, in the sense that even convergence in measure of $\Sigma_n$ to $\Sigma$ is false without the outer stability assumption of Definition \ref{def:set convergence} (see Remark \ref{rem:stability counterexample} for a counterexample).

Finally, we discuss the free boundary behavior of maximal solutions in dimension $1$ and in the radial case, which are universal by Theorem~\ref{thm:universality}. We are interested in the topology of the positive set. Given a radial or one-dimensional solution $\eta$ to \eqref{eq:stefan intro} for which $U:=\{\mu>0\}$ is connected, we say that $\eta$ exhibits nucleation if for some $t_0>0$, the set $\{\eta(t_0,\cdot)>0\}$ is disconnected.
\begin{theorem}[Proposition \ref{smooth_fb}, Proposition \ref{nuc_example}, Proposition \ref{nucleation_most_unlikely}]  \label{thm:1-d}

Assume that $d=1$ with $U=(0,1)$, and let $\mu \in \mathcal{S}_0(U)$ with $|\{\mu=1\}|=0$. Then the following holds for the maximal solution $\eta$  with initial data $\mu$:
\begin{itemize}
\item[(a)] If $0 \leq \mu \leq 1$, then nucleation does not occur, and the free boundary consists of two smoothly evolving boundary points. 

\item[(b)] There exists $\mu \in C^{\infty}((0,1))$ with $\max \mu >1$ such that  $\eta$ nucleates.
\item[(c)] If $\eta$ nucleates, then every weak solution to \eqref{eq:stefan intro} with initial data $\mu$ also nucleates.  
\end{itemize}
For  $d\geq2$,  parallel statements hold if $\mu$ and $u$ are radial functions in a ball or an annulus.
\end{theorem}

Although the nucleation example above is specific, Section \ref{maximal_solution_d_1_section} offers heuristics indicating that nucleation is generic when $\sup\mu>1$. In higher dimensions, it remains an open question whether the occurrence of nucleation depends sharply on the critical supercooling value (see \cite[§17.3.2]{glicksman2010principles}, where nucleation occurs experimentally beyond the threshold $1$ for the kinetic undercooled Stefan problem).

In the two-phase setting, Gurtin~\cite{gurtin1994thermodynamics} interpreted nucleation as an entropy-increasing mechanism that stabilizes the system, providing a thermodynamic rationale for continuing solutions past blow-up. In this context, Theorem~\ref{thm:1-d}(c) asserts that maximal solutions are \emph{nucleation-minimal}; they nucleate only when continuation otherwise fails. This mirrors the notion of one-dimensional ``physical'' solutions in \cite{DNS22}, which minimize the free boundary jump in the single-free-boundary-point regime. 

\subsection{Main ideas}\label{subsection:main ideas}
 After discussing some preliminaries in Section \ref{sec:prelim}, we obtain, in Section \ref{sec:subharmonic rough}, a useful characterization for subharmonic order in rough domains, and a low-regularity Green-type formula for subharmonically ordered measures. These will be very important throughout the paper, particularly for the fractal freezing results, and to handle the low regularity of the dual optimizers of \eqref{eq:primal intro}, as is further discussed below.

In Section \ref{sec:fractalfreezing}, we prove Theorem \ref{thm: fractal freezing} with a two-step strategy. First, by exploiting the connection between solutions to \eqref{eq:stefan intro} and so-called subharmonically generated sets \cite[Sec. 6]{kim2024stefan}, we extend the existence theory of \cite{choi2024existence} to irregular domains $U$ (see Proposition \ref{prop: existence}). Then, by solving \eqref{eq:stefan intro} in a suitably rough subdomain $\tilde{U}\subset U$ such that $|U\setminus \tilde{U}|=0$, we produce a solution to \eqref{eq:stefan intro} such that the set of initial freezing $\{s=0\}$ contains an arbitrary fractal. 

The main ingredient in the second step, which is also of independent interest, is a {\it Markovian gluing principle} for the optimal Skorokhod problem; see Proposition~\ref{gluing_lemma}. By \cite[Rem. 6.19]{beiglbock2017optimal}, any admissible stopping time that is the first hitting time to a closed, nondecreasing in time set, must be optimal. This characterization of optimal times allows us to glue them: the optimal stopping times associated with two Skorokhod problems can be concatenated at any prescribed deterministic time, and the resulting stopping time is optimal for the suitably glued Skorokhod problem. We stress that this gluing result is stronger than a standard dynamic programming principle, due to the freedom in choosing the continuation after the gluing time. 

Due to the connection between \eqref{eq:stefan intro} and the optimal Skorokhod problem \eqref{Sko}, it follows then that one can patch together two weak Stefan solutions to obtain a new weak solution (see Proposition~\ref{cor:programmed_gluing_Stefan}).  The combination of these two steps generates ``programmed fractal freezing" at positive times. 
 
In Section \ref{sec:regularity}, we prove Theorem \ref{thm: regularity intro} by making use of the dual formulation of \eqref{eq:primal intro}, namely
\begin{equation}
    \label{eq:dual intro} \mathcal{D}(\mu,u)=-\inf_{\psi}\left\{\int (\psi-u)^+-\int\psi \mu :\psi \in L^1(U),\; \Delta\psi \geq 0\right\}. 
\end{equation}
The minimizer $\nu$ for \eqref{eq:primal intro} may then be a.e. characterized as \begin{equation} \label{eq:nu defi dual}\nu=\chi_{\{\psi>u\}},\end{equation}
where $\psi$ is any optimizer for \eqref{eq:dual intro}. While one must be cautious about the fact that the sets $\{\psi>u\}$ and $\Sigma$ only agree up to a null set and cannot be identified, this suggests the natural strategy of obtaining regularity of $\Sigma$ by studying the regularity of the optimizers $\psi$. 

Indeed, we show that any optimizer $\psi$ is harmonic away from the set of initial freezing. This  allows us to take advantage of the parabolic obstacle problem formulation of \eqref{eq:stefan intro} and precisely describe, through PDE methods, the fine structure of $\Sigma$. The dimension estimate \eqref{eq:dim d-2 intro} is obtained through a Weierstrass preparation argument that uses in a critical way the strict superharmonicity of the weight $u$ and the measure saturation of the transition zone.

Let us mention that the dual problem \eqref{eq:dual intro} that we study in Section \ref{sec:regularity} is a generalization of the classical problem of approximating, in the $L^1$ norm, an arbitrary superharmonic function $u$ by a subharmonic function $\psi$ (see \cite{armitage96,armitage99}, and see also \cite{armitage92quadrature,shapiro00,GoHaRo} for the problem with \emph{harmonic} $\psi$). Indeed, upon adding a constant, a computation shows that \eqref{eq:dual intro} is equivalent to
\begin{equation} \label{eq:dual L1 approx}
    \inf_{\psi}\left\{\int (\psi-u)^+(1-\mu)+\int(\psi-u)^- \mu :\psi \in L^1(U), \Delta\psi \geq 0\right\},
\end{equation}
and, by taking $\mu \equiv \frac12\chi_{U}$, one recovers the classical case
\begin{equation} \label{eq:L1 approx classic}
   \inf_{\psi\in L^1(U), \;\Delta \psi\geq 0}\|\psi-u\|_{L^1(U)}. 
\end{equation}
Even in this special setting, the problem is delicate, and global regularity results for the optimizers only exist in special domains, under strong conditional assumptions. For instance, it was shown in \cite[Thm. 2]{armitage96} that if $U=B_1$ and \eqref{eq:L1 approx classic} has a continuous optimizer $\psi \in C(\overline{B_1})$, then $\psi$ is a harmonic function which coincides with the best harmonic $L^1$ approximant of $u$. Counterexamples show, however, that even smooth superharmonic functions in a ball can have discontinuous best subharmonic $L^1$ approximants \cite[Cor. 7.2]{shapiro00}.

In Section \ref{sec:non-universality}, we prove Theorem \ref{thm:universality} and Proposition \ref{prop:waiting time}. First, we obtain a complete description of radial and one-dimensional maximal solutions: their transition zone is an outer shell of the domain and there is no initial freezing (Proposition~\ref{prop:radial transition zone}); these well-behaved solutions are then used as building blocks for the non-universality and waiting time constructions.

Section \ref{sec:stability} contains the proof of Theorem \ref{thm:stability intro}. Using the regularity of the dual optimizers obtained earlier, we first show that the transition zone is stable under perturbations of $\mu$ and $u$, which yields weak convergence of $\eta$ via the use of potential variables. This is then upgraded to strong convergence by using some of the free boundary regularity implied by Theorem \ref{thm: regularity intro}.

Finally, we prove Theorem \ref{thm:1-d}  in Section \ref{maximal_solution_d_1_section}. In the radial and one-dimensional cases, the spatial (or radial) derivative of the potential satisfies a monotonicity property when $0 \le \mu \le 1$ (see the arguments of Proposition \ref{prop: no_nuc_jumps}). 
This monotonicity rules out nucleation and jumps of the free boundary for the maximal solution in this regime. With nucleation and jumps absent for $0 \le \mu \le 1$, regularity theory for the obstacle problem yields that the free boundary is a smooth hypersurface. In contrast, when $\sup\mu > 1$, we present a generic example of a maximal solution in $d=1$ that nucleates at a positive time, even when $\mu$ is smooth.

\section{Preliminaries} \label{sec:prelim}
Throughout the paper, we will use the same symbol to refer to an absolutely continuous measure on $\R^d$ as to refer to its density function.
We begin by giving the definition of a weak solution to the supercooled Stefan problem.
\begin{definition} \label{def:weak solution}
  Let $0\leq\mu\in L^{\infty}(\R^d)$ be of compact support and denote $U = \{\mu>0\}$. We say that $\eta \in L^2((0,\infty) ;H^1(\R^d))$ 
is a weak solution of \eqref{eq:stefan intro}   if 
\begin{enumerate}[label=(\alph*)]
  \item the set $\{\eta(t,\cdot)>0\}$ is nonincreasing in $t$,
  \item $\eta$ vanishes outside $U\times(0,\infty)$, and $\limsup_{t \downarrow 0}\{\eta(t,\cdot)>0\}=U$ up to a null set,
  \item for any test function $\varphi \in C_c^\infty([0,\infty) \times \mathbb{R}^d)$,
  \begin{equation}
  \label{eq:weakform}
  \int_0^\infty\!\!\int_{\mathbb{R}^d}
  \Big[(\eta-\chi_{\{\eta>0\}})\varphi_t - \tfrac{1}{2}\nabla \eta\cdot \nabla\varphi\Big]\,dt\,dx
  = 
  \int_{\mathbb{R}^d}\!\big[(-\mu+\chi_U)\varphi\big](0,x)\,dx.
  \end{equation}
\end{enumerate}
\end{definition}

An important notion of ordering for establishing existence of solutions to the Stefan problem is the subharmonic order.
\begin{definition}[Subharmonic ordering]\label{def:subharmonic ordering}For $0\leq \mu,\nu \in L^{\infty}(\R^d)$ with compact support and $\int_{\R^d} \mu(x)dx =\int_{\R^d} \nu$, we say that $\mu$ and $\nu$ are {\it subharmonically ordered} if for some open set $\mathcal{O}$ containing the support of $\mu+\nu$, one has
\begin{equation}\label{eq:subharmonic defi}
    \int \varphi(x)d\mu(x)\leq \int \varphi(x)d\nu(x) \;\;\text{ for every smooth subharmonic function}\;\; \varphi \text{ in } \;\mathcal{O},
\end{equation}  
and we write $\mu \leq_{\operatorname{SH}}\nu$.
\end{definition}
Subharmonic ordering admits a natural probabilistic interpretation in terms of Brownian motion \cite{rost1971stopping,Roo69}. For our filtered probability space, we adopt the same standard assumptions as in \cite[Sec. 2.2.1]{kim2024stefan}. By \cite[Thm. 1.5]{ghoussoub2020optimal}, it is known that $\mu\leq_{\operatorname{SH}}\nu$ if and only if there exists a stopping time $\tau$ for Brownian motion such that 
\begin{equation}\label{eq:subharmonic defi stop}
       W_0 \sim \mu, \quad W_{\tau}\sim \nu, \quad \mathbb{E}[\tau]<\infty. 
\end{equation}
We use this probabilistic interpretation to define subharmonic order with respect to a rough domain. 
\begin{definition}\label{def:subharmonic ordering relative} Let $\mu$, $\nu$ be as in Definition \ref{def:subharmonic ordering}, let $U$ be a bounded open set such that $\mu(\R^d \setminus U)=\nu(\R^d \setminus U)=0$, and let $\tau^U :=\inf\{t \geq 0: W_t \notin U\}$ be the exit time from $U$. We say that $\mu\leq_{\operatorname{SH}}\nu$ with respect to $U$ if there exists a stopping time $\tau\leq \tau^U$ for Brownian motion such that \eqref{eq:subharmonic defi stop} holds.    
\end{definition}
We now recall the {\it optimal Skorokhod problem} $\mathcal{Q}(\mu, \nu)$, which given two measures $\mu \leq_{\textnormal{SH}} \nu$, is given by
\begin{equation}\label{Sko} \mathcal{Q}(\mu,\nu) := \inf_{\tau} \left\{ \mathbb{E}[\tau^2]: W_0 \sim \mu \text{ and } W_{\tau} \sim \nu \right \}, \end{equation} where $\tau$ are stopping times (for further background on the optimal Skorokhod problem, see the survey \cite{obloj2004skorokhod} and the references therein). By the results of \cite{beiglbock2017optimal}, if $\mu \leq_{\textnormal{SH}} \nu$, then there exists a unique optimal stopping time for $\mathcal{Q}(\mu,\nu)$, which we denote as $\tau^*$. \smallskip

By \cite[Thm. 4.15]{kim2024stefan} (or \cite{gassiat2021free}), one can characterize $\tau^*$ in terms of an obstacle problem. Indeed, if one considers the unique solution to 
\begin{equation} \label{obstacle_problem_123}
\begin{cases}
    \p_t w - \frac{1}{2} \Delta w = -\frac{1}{2} \nu \chi_{ \{w>0\} } \text{ on } (0,\infty) \times \R^d, \\ 
    w(0,x) = \Delta^{-1}(\nu-\mu)(x) \text{ on } \R^d,
\end{cases}
\end{equation} then
\begin{equation} \tau^* = \inf\{t \geq 0:(t,W_t) \in \{w=0\} \} \text{ almost surely, where } w \text{ solves } \eqref{obstacle_problem_123}.  \label{optimal_stopping_time} \end{equation}

Here  $\Delta^{-1}\mu(x) := \mathcal{N} * \mu$, where $\mathcal{N}$ is the standard Newtonian potential. \smallskip

Next, we review the connection between subharmonic ordering, \eqref{Sko}, and \eqref{eq:stefan intro} established in \cite{kim2024stefan}. To make this link precise, we employ the Eulerian formulation introduced in \cite{ghoussoub2019pde}.

\begin{definition}[Eulerian variables] \cite[Prop. 2.2]{ghoussoub2019pde} \label{def:eulerian_variables} Let $\mu \leq_{\operatorname{SH}} \nu$ and let $\tau$ be a stopping time such that $W_0 \sim \mu$, $W_{\tau} \sim \nu$, and $\mathbb{E}[\tau]<\infty$. We say that $(\eta,\rho)$ are, respectively the {\it active and stopped} particle distributions associated to $(\mu,\nu,\tau)$ if  
\[  \mathbb{E} \left[ \int_0^{\tau} \varphi(s,W_s) ds \right] = \iint_{(0,\infty) \times \R^d} \varphi(t,x) \eta(t,x)dtdx,  \quad \mathbb{E}[\varphi(\tau,W_{\tau})] = \iint_{[0,\infty) \times \R^d} \varphi(t,x) d\rho(t,x),  \]
for all test functions $\varphi \in C^{\infty}_c( [0,\infty) \times \R^d)$. When $\tau^*=\tau^*_{\mu,\nu}$ optimizes $\mathcal{Q}(\mu,\nu)$, we suppress the dependence on the third argument of $(\mu,\nu,\tau^*)$ and simply write $(\eta,\rho)$ for the active and stopped distributions associated with $(\mu,\nu)$.
\end{definition} 
\begin{proposition}[\cite{kim2024stefan},\cite{choi2024existence}] \label{prop:eulerian} 
Let $\mu, \nu$ and $\tau^*$ be as in Definition~\ref{def:eulerian_variables}. Then the following holds:
\begin{itemize}

    \item[(a)] \cite[Thm. 4.15, Sec. 4.3]{kim2024stefan}   For $\eta$ associated to $(\mu, \nu)$, the potential variable
$w(t,x) :=\frac{1}{2}\int_t^{\infty} \eta(s,x) ds$ is the unique solution to the obstacle problem \eqref{obstacle_problem_123}. In addition,  $\{w>0\}=\{\eta>0\}$ a.e., and if $\mu_t$ denotes the law of $W_{t \wedge \tau^*},$ then
\[ w(t,x) = \Delta^{-1}(\nu-\mu_t). \]

\item[(b)] \cite[Thm. 6.5, Cor. 6.6]{choi2024existence}  Suppose further that $U$ has a Lipschitz boundary and that $\mu \in \mathcal{S}_{\delta}(U)$ for some $0<\delta<1$. Then for any weight function $u \in C^2(\overline{U})$ with $\Delta u < 0$ in $U$, there exists a unique optimizer $\nu^*$ of \eqref{eq:primal intro}, and it satisfies $\nu^*\in\{0,1\}$ and $\mu \leq_{\textnormal{SH}} \nu^*$ over $U$. 

\item[(c)] \cite[Thm. 7.4]{choi2024existence} If $\nu\in \{0,1\}$, $|\{\mu=1\}|=0$, and $\mu \leq_{\textnormal{SH}} \nu$ with respect to the Lipschitz domain $\{\mu>0\}$, then $\eta$ is a weak solution of \eqref{eq:stefan intro}.

\item[(d)] \cite[Thm. 7.2, Rem. 7.3]{choi2024existence} For any weak solution $\eta$ to \eqref{eq:stefan intro}, one has $\mu \leq_{\textnormal{SH}} \chi_{\Sigma}, $ where $\Sigma$ is the transition zone of $\eta$.
 
\end{itemize}
\smallskip
\end{proposition}
\begin{remark}\label{rem:weak sol non lipschitz}While Proposition \ref{prop:eulerian} (c) is stated for Lipschitz domains, the proof in \cite[Thm. 7.4, Rem. 7.3]{choi2024existence} remains valid as long as $\Delta^{-1}(\nu-\mu)\equiv 0$ outside $U=\{\mu>0\}$. While this is automatically true if $U$ is Lipschitz, it may fail for topologically irregular sets such as punctured balls (see Proposition \ref{prop:subharmonic equivalent} and Remark \ref{rem:counterexample subharmonic L1}).   
\end{remark}
We will often refer to $w$ as the \emph{potential variable} associated with $\eta$, or simply as the potential variable when $\eta$ is clear from context. It is also important to note that, by standard parabolic regularity, the function $w$ enjoys $C^{1,\alpha}_x\cap C^{0,1}_t$ regularity for any $0<\alpha<1$, and $\nabla w$ is space-time H\"older continuous (see \cite[Prop. 7.18]{lieberman96}, \cite[p. 80, Lem. 3.3]{ladyzhenskaya}). 
\medskip

We now show that $\eta$ has a well-defined LSC representative, which in particular makes the freezing set at each time slice well-defined.
\begin{lemma} \label{LSC_eta} Under the notation and assumptions of Proposition \ref{prop:eulerian} (a), there exists $\tilde{\eta} \in \operatorname{LSC}( (0,\infty) \times \R^d)$, which satisfies $\tilde{\eta}=\eta$ a.e., and $\{\tilde{\eta}>0\} \cap \{t>0\}=\{w>0\} \cap \{t>0\}$.
\end{lemma}

\begin{proof} By \cite{ghoussoub2019pde} (or see \cite[Eq. (2.3)]{kim2024stefan}), $(\eta,\rho)$ is a weak solution to
\[ \p_t \eta + \rho = \frac{1}{2} \Delta \eta.  \] Because $\rho \sim (\tau^*,W_{\tau^*})$, \eqref{optimal_stopping_time}, and the closedness of $\{w=0\}$ imply that $(\tau^*,W_{\tau^*}) \in \{w=0\}$ almost surely. Hence $\rho(\{w>0\})=0$, so that $\eta$ is caloric in the open set $\{w>0\}$. Therefore, one can modify $\eta$ on a set of space-time measure zero, so that $\eta$ is smooth on $\{w>0\}$. As $\eta \geq 0$, the strong minimum principle implies that this modification satisfies $\{w>0\} \cap \{t>0\} \subset \{\eta>0\} \cap \{t>0\}$. By Proposition \ref{prop:eulerian} (a), 
\[ \tilde{\eta}(t,x) := \begin{cases} \eta(t,x) &(t,x) \in \{w>0\} \\ 0 & \text{otherwise} \end{cases} \] is a valid modification of $\eta$, which is LSC (because $\tilde{\eta} \geq 0$ and $\tilde{\eta}$ is smooth on $\{w>0\} \cap \{t>0\}$), and satisfies $\{\tilde{\eta}>0\} \cap \{t>0\}=\{w>0\} \cap \{t>0\}$.
\end{proof}

From now on, we always identify $\eta$ with this precise LSC representative that satisfies $\{\eta>0\} \cap \{t>0\} = \{w>0\} \cap \{t>0\}.$

 We may now give the definition of maximal solutions.
\begin{definition}[Maximal solutions associated to $u$]\label{def:maximal}
Let  $\mu$ be as given in Proposition~\ref{prop:eulerian} (b)--(c), let $0<u\in C^2(\overline{U})$, $U:=\{\mu>0\}$  with $\Delta u<0$ in $U$, and let $\nu^*$ be the unique optimal target measure of $\mathcal{P}(\mu,u)$. We call $\eta$, the active particle distribution associated to $(\mu, \nu^*)$, the maximal solution of \eqref{eq:stefan intro} corresponding to the weight $u$. 
\end{definition}

Because of Proposition~\ref{prop:eulerian}, the maximal solution $\eta$ associated to the weight $u$ is a well-defined weak solution to \eqref{eq:stefan intro}.  The conditions of Propositions~\ref{prop:eulerian} (c) are meant to ensure that $\eta$ is a weak solution. Indeed, without the assumption $|\{\mu=1\}|=0$, $\eta$ may exhibit instantaneous freezing on a set of positive measure, thereby altering the initial data. For example, if $\mu=\chi_{U}$, then for any weight $u$ the optimal target measure of $\mathcal{P}(\mu,u)$ is $\mu$, so the active particle distribution is $\eta\equiv 0$, which does not solve~\eqref{eq:stefan intro} with initial data $\mu$. We now show that among all weak solutions of \eqref{eq:stefan intro}, $\eta$ uniquely maximizes \eqref{eq:E maximized intro}.
    \begin{proposition}[Delayed solidification characterization for maximal solutions] \label{prop: maximal_sol_characterization} Suppose $\eta^{*}$ is the maximal  solution of \eqref{eq:stefan intro} associated with the weight $u$, with initial data $\mu$  satisfying the assumptions of Definition~\ref{def:maximal}.  Then $\eta^{*}$ is the 
\emph{unique weak solution} of \eqref{eq:stefan intro} such that, for every weak solution $\eta$ of \eqref{eq:stefan intro},
\begin{equation} \label{L1_maximal_property}
    \int_{\R^d} \int_{0}^{\infty} -\Delta u(x)\, \eta^{*}(t,x)\,dt\,dx 
    \;\ge\;
    \int_{\R^d} \int_{0}^{\infty} -\Delta u(x)\, \eta(t,x)\,dt\,dx.
\end{equation}
\end{proposition}

\begin{proof} 
Let  $\eta$ be a weak solution. Then $\{s(x)>0\}\subset U$ due to Definition \ref{def:weak solution}. Moreover, Proposition \ref{prop:eulerian} (d)
 implies that $\mu \le_{\mathrm{SH}} \nu:=\chi_{\{0<s<\infty\}}$. Hence,  $\mu \le_{\mathrm{SH}} \chi_{\{0<s<\infty\}}$ with respect to $U$ because of \eqref{optimal_stopping_time}. In particular $\nu:=\chi_{\{0<s<\infty\}}$ is admissible for $\mathcal{P}(\mu,u)$, and thus  
$\int u\, d\nu^{*} \le \int u\, d\nu$.  
Let $\tau$ and $\tau^{*}$ be the optimal stopping times for 
$\mathcal{Q}(\mu,\nu)$ and $\mathcal{Q}(\mu,\nu^{*})$. Itô's formula then yields
$$\mathbb{E}\left[\int_{0}^{\tau^{*}} \Delta u(W_s)\,ds\right]
 \leq 
 \mathbb{E}\left[\int_{0}^{\tau} \Delta u(W_s)\,ds\right],$$
 so we infer from Definition \ref{def:eulerian_variables} that \eqref{L1_maximal_property} holds. Uniqueness now follows from Proposition~\ref{prop:eulerian} (b), which guarantees the uniqueness of an optimal target measure of $\mathcal{P}(\mu,u)$.
\end{proof}

\begin{remark}[Maximal solutions for $\mu \in \mathcal{S}_0(U)$] \label{rem:maximal S0}
Definition~\ref{def:maximal} may be naturally extended in the obvious way to $\mu \in \mathcal{S}_0(U)$ whenever $\mathcal{P}(\mu,u)$ admits a unique optimal target measure, and the associated particle distribution $\eta$ is a weak solution to \eqref{eq:stefan intro}.  However, we emphasize that neither of these assumptions is known to hold in general without the slack condition $\delta>0$ (except when $d=1$ or in the radial case, as we show in Proposition \ref{target_measure_universial}).  
\end{remark}

\section{Properties of relative subharmonic order in rough domains}
\label{sec:subharmonic rough}
In this section we obtain, under minimal regularity assumptions on $U$, a characterization of the subharmonic order with respect to $U$. This will be useful in later sections, especially for the fractal freezing results of Section \ref{sec:fractalfreezing}. Along the way we also derive a low-regularity Green-type formula for subharmonically ordered measures that will be important in  Section \ref{sec:regularity}. 

We begin with a basic elliptic estimate.
\begin{lemma} \label{lem: w quadratic growth} Let $h \in L^\infty(\R^d)\cap L^1(\R^d)$, and assume that $v:= \Delta^{-1}h\geq0$. Then there exists $C=C(d)>0$ such that for any $x\in\{v=0\}$  and $r>0$,
\begin{equation*}
    \|v\|_{L^{\infty}(B_r(x))}+r\|\nabla v\|_{L^{\infty}(B_r(x))}\leq C\|h\|_{L^{\infty}(\R^d)}r^2.
\end{equation*}    
\end{lemma}
\begin{proof}
The Harnack inequality \cite[Thms. 8.17, 8.18]{GilTru} yields
\begin{equation*}    
\|v\|_{L^{\infty}(B_r(x))}\leq C(\inf_{B_r(x)} v+r^2\|\Delta v\|_{L^{\infty}(B_{2r}(x))})= C\|h\|_{L^{\infty}(\R^d)}r^2.\end{equation*}
On the other hand, by the Calder\'on-Zygmund estimates \cite[Thm. 9.11]{GilTru}, we have
\begin{equation*}
    r\|\nabla v\|_{L^{\infty}(B_r(x))}\leq C(\|v\|_{L^{\infty}(B_{2r}(x))}+r^2\|\Delta v\|_{L^{\infty}(B_{2r}(x))})\leq C\|h\|_{L^{\infty}(\R^d)}r^2.
\end{equation*} \end{proof}
For the next lemma, we recall that
\begin{equation}
    U_{\varepsilon}=\{x\in U: \operatorname{dist}(x,\partial U)<\e\}.
\end{equation}
Besides its important role in this section, the other main purpose of the following Green-type formula will be to eventually test it against the optimizer $\psi$ of the dual problem \eqref{eq:dual intro}, which is merely in $L^1(U)$. We stress that both the $L^1$ regularity of $\psi$ and the quadratic decay of the Green potential $v$ near $\partial U$ are critical in the proof below.

\begin{lemma}[Green identity for subharmonic order]\label{lem:integration by parts} Let $U$ be a bounded open set, let $0\leq \mu, \nu \in L^{\infty}(\R^d)$ vanish outside $U$, and let $v=\Delta^{-1}(\nu-\mu)$. Then the following holds:
\begin{itemize}
    \item[(i)] If $v\geq 0$ and $v$ vanishes outside $U$, then, for every subharmonic function $\psi\in L^1(U)$,
\begin{equation} \label{eq:integration by parts}
    \int \psi(x)\nu(x)dx =\int\psi (x)\mu(x)dx +\int v(x)\Delta \psi(dx),
\end{equation}
where $\Delta \psi\geq 0$ is the distributional Laplacian of $\psi$.
\item[(ii)] If $\mu \leq_{\operatorname{SH}}\nu$ with respect to $U$ and $U=(\overline{U})^{\circ}$, then $v\geq 0$, $v$ vanishes outside $U$, and the conclusion of (i) holds.
\end{itemize}
\end{lemma}
\begin{proof} Assume first that $v\geq 0$ and $v \equiv 0$ outside $U$. Let $\zeta_{\e}\in C^{\infty}_c(U)$ be such that $0\leq \zeta_{\e} \leq 1$, $\zeta_{\e} \equiv 1$ in $U\setminus U_{\varepsilon}$, $|D\zeta_{\e} |\leq \frac{C}{\varepsilon}$, and $|D^2\zeta_{\e} |\leq \frac{C}{\varepsilon^2}$.
We then have
\begin{equation} \label{ipa13deq1}\left|\int_{U}\Delta(v\zeta_{\e})\psi-\int_{U}(\Delta v\zeta_{\e}\psi)\right|=\left| \int_{U}( 2\nabla v\cdot \nabla \zeta_{\e}\psi+\Delta \zeta_{\e} v\psi)\right|\\\leq C\int_{U_{\varepsilon}} (|\nabla v\|\nabla \zeta_{\e}\|\psi|+|\Delta \zeta_{\e}\|v\|\psi|).\end{equation}
By Lemma \ref{lem: w quadratic growth}, since $v$ vanishes on $\partial U$, we have $|v|+\varepsilon|\nabla v|\leq C\varepsilon ^2$ on $U_{\varepsilon}$, and therefore
\begin{equation}\label{ipa13deq2}
\int_{U_{\varepsilon}} (|\nabla v\|\nabla \zeta_{\e}\|\psi|+|\Delta \zeta_{\e}\|v\|\psi|)\leq C\int_{U_{\varepsilon}}|\psi|.
\end{equation}
Since $\psi\in L^1(U)$ is a subharmonic function, $\sigma :=\Delta\psi $ is a non-negative Radon measure on $U$. From \eqref{ipa13deq1} and \eqref{ipa13deq2}, we then have the estimate
\begin{equation*} \left|\int_{U}v\zeta_{\e} d\sigma -\int_{U}(\Delta v\zeta_{\e}\psi)\right|\leq C\int_{U_{\varepsilon}}|\psi|.    
\end{equation*}
Letting $\varepsilon \to 0$ and using the fact that $\psi \in L^{1}(U)$, we conclude from the monotone convergence theorem and the dominated convergence theorem that
\begin{equation*}
 \int_{U}vd\sigma=\int_{U}\Delta v \psi,
\end{equation*}
which is precisely \eqref{eq:integration by parts}. This proves (i).

Assume now that $\mu \leq_{\operatorname{SH}} \nu $ with respect to $U$ and $U=(\overline{U})^{\circ}$. 
For each $x\in \R^d$, by approximation, one can test \eqref{eq:subharmonic defi} against $\cN(\cdot-x)$, which readily yields $v(x)\geq 0$. By It\^o's formula, for any $\xi\in C^{\infty}_c(\R^d \setminus \overline{U})$, and any $\tau\leq \tau^U$ such that \eqref{eq:subharmonic defi stop} holds, 
\begin{equation*}
    \int_{\R^d}v\xi =\int_{\R^d} (\nu-\mu)\Delta^{-1} \xi =\mathbb{E}[\Delta^{-1} \xi(W_{\tau})-\Delta^{-1} \xi (W_{0})]=\mathbb{E}\left[\int_{0}^{\tau}\frac12\xi (W_s)ds\right]=0. 
\end{equation*}
This implies that $v\equiv 0$ a.e. outside $\overline{U}$. Since $U=(\overline{U})^{\circ}$, for any $x\notin U$ and any ball $B_r(x)$, we must have $B_r(x)\setminus\overline{U}\neq \emptyset$. Since $\R^d \setminus \overline{U}$ is open, we have $|B_r(x)\setminus \overline{U}|>0 $, and since $v=0$ a.e. outside $\overline{U}$, we infer that $x\in \overline{\{v=0\}}$. But $v$ is continuous, so it follows that $v(x)=0$, which shows that $v \equiv 0$ outside $U$. We may then apply part (i) to deduce \eqref{eq:integration by parts}.
\end{proof}
We finish this section by characterizing the  subharmonic order in $U$.

See Remark \ref{rem:counterexample subharmonic L1} for a counterexample showing that the equivalence can be false if the mild condition $U=(\overline{U})^{\circ}$ is removed.
\begin{proposition}[Relative subharmonic order characterization]\label{prop:subharmonic equivalent} Let $\mu, \nu$ be as in Definition \ref{def:subharmonic ordering}, and assume in addition that $\mu,\nu\equiv 0$ outside of a bounded open set $U$ such that $U=\overline{U}^{\circ}$.  Then the following are equivalent:
    \begin{equation} \label{eq:subharmonic defi equiv L1}
    \int \varphi(x)d\mu(x)\leq \int \varphi(x)d\nu(x) \;\;\text{ for every  subharmonic function} \;\;\varphi\in L^1(U),
\end{equation}
\begin{equation} \label{eq:pot vanish outside U equiv}
    \mu \leq_{\operatorname{SH}}\nu \quad \text{ and }\quad  v:=\Delta^{-1}(\nu -\mu)\equiv 0\quad \text{outside}\quad  U,
\end{equation}
and
\begin{equation}\label{eq:rel subharmonic equiv}
    \mu \leq_{\operatorname{SH}} \nu\;\;\; \text{with respect to } U.
\end{equation}
 Moreover, if one removes the assumption that $U=(\overline{U})^{\circ}$, then \eqref{eq:subharmonic defi equiv L1} and \eqref{eq:pot vanish outside U equiv} are still equivalent, and they both imply \eqref{eq:rel subharmonic equiv}.
\end{proposition}
\begin{proof} Assume first that \eqref{eq:subharmonic defi equiv L1} holds. Since $U$ is bounded, it follows that $\mathcal{N}(\cdot - x_0) \in L^1(U)$ for every $x_0 \in \mathbb{R}^d$. 
If moreover $x_0 \in U^c$, then $\mathcal{N}(x - x_0)$ is harmonic in $U$. Thus, testing \eqref{eq:subharmonic defi equiv L1} against $\pm \mathcal{N}(x - x_0)$, we get
\[
    \int_{\mathbb{R}^d} \mathcal{N}(x - x_0)\, d\mu(x)
    = \int_{\mathbb{R}^d} \mathcal{N}(x - x_0)\, d\nu(x),
\]
that is, $v(x_0) = 0$. This proves that $v$ vanishes outside $U$, so \eqref{eq:pot vanish outside U equiv} holds. 

Assume now \eqref{eq:pot vanish outside U equiv}. As in the proof of Lemma \ref{lem:integration by parts}, we see from \eqref{eq:subharmonic defi} that $v\geq 0$. Then, by Lemma \ref{lem:integration by parts} (i), \eqref{eq:subharmonic defi equiv L1}
holds. Letting $\tau^*$ be the minimizer of \eqref{Sko} we see by \eqref{optimal_stopping_time} that $\tau^*\leq \tau^U$. This proves that $\mu \leq_{\operatorname{SH}}\nu$ with respect to $U$, and thus \eqref{eq:rel subharmonic equiv} holds. We have thus far shown that \eqref{eq:subharmonic defi equiv L1} and \eqref{eq:pot vanish outside U equiv} are equivalent, and they both imply \eqref{eq:rel subharmonic equiv}.

Finally, assume that $U=(\overline{U})^{\circ}$ and \eqref{eq:rel subharmonic equiv} holds. Then, \eqref{eq:subharmonic defi equiv L1} follows from Lemma \ref{lem:integration by parts} (ii).

\end{proof}

\section{Programmed fractal freezing}
\label{sec:fractalfreezing}

Whenever $\eta$ is a weak solution to \eqref{eq:stefan intro}, we recall that the freezing time $s:\R^d\to [0,\infty]$ is defined by \eqref{eq:s defi intro}. The freezing set at each time can then be defined as
\begin{equation}\label{initial_freezing}
    F_t=\{x\in U: s(x)=t\} \quad t\in [0,\infty).
\end{equation}
While $F_0$ can only be a measure zero set (see Definition \ref{def:weak solution}), it can act as a barrier for Brownian particles that affects the pointwise dynamics of ice in significant ways, and is thus worth our attention. We will first see in Section \ref{subsection:existence} that, for \emph{arbitrary initial data}, $F_0$ can contain any prescribed closed null set. In Section \ref{subsection:gluing}, we take advantage of this observation, together with a programmed gluing result (Proposition~\ref{cor:programmed_gluing_Stefan}) to obtain solutions for which $F_t$ includes any prescribed null set at any positive time $t>0$, and prove Theorem \ref{thm: fractal freezing}.

\subsection{Existence result and the geometry of initial freezing}

\label{subsection:existence}

In this section, we study the generic nature of the set $F_0$ given by \eqref{initial_freezing}. We begin by obtaining an existence result for rough domains with initial data below the critical supercooling threshold, which removes the Lipschitz assumption on $U$ from \cite[Thm. 1.2]{choi2024existence}.  Roughly speaking, the main observation is that, by Proposition \ref{prop:eulerian}, there is a correspondence between solutions to \eqref{eq:stefan intro} and measurable sets $\Sigma \subset U$ such that $\mu\leq_{\operatorname{SH}}\chi_{\Sigma}$ with respect to $U$.
\begin{proposition}[Existence for the supercooled Stefan problem in bounded domains]\label{prop: existence} Let $U\subset \R^d$ be a bounded open set. Assume that  $0\leq \mu \leq 1-\delta$ for some $\delta\in(0,1)$ and $U=\{\mu>0\}$.

Then there exists a weak solution $\eta$ to \eqref{eq:stefan intro}. 
\end{proposition}
\begin{proof}By Remark \ref{rem:weak sol non lipschitz} and Proposition \ref{prop:subharmonic equivalent}, it is enough to show that there exists a measurable set $\Sigma \subset U$ such that the pair $(\mu,\chi_{\Sigma})$ satisfies \eqref{eq:subharmonic defi equiv L1}. 
Given that $U$ is open, we may write
\begin{equation*}U=\bigcup_{i=1}^\infty \overline{Q_i}, \end{equation*}
where the $Q_i$ are open, disjoint dyadic cubes. Since, for each $i$, $Q_i$ is a bounded Lipschitz domain and $\mu \chi_{Q_i}\in \cS_{\delta}(Q_i)$, Proposition \ref{prop:eulerian} (b) implies that there exist measurable sets $\Sigma_i \subset Q_i$ such that
\begin{equation} \label{eq: Qi SH ord}
    \mu |_{Q_i}\leq_{\text{SH}}\chi_{\Sigma_i}
\end{equation}
in the subharmonic order with respect to $Q_i$. Let $\Sigma =\bigcup_{i=1}^{\infty}\Sigma_i$, and let $\nu =\chi_{\Sigma}$. Then, for any subharmonic function $\phi\in L^1(U)$, we have, by \eqref{eq: Qi SH ord} and Proposition \ref{prop:subharmonic equivalent}, that
\begin{equation*} 
 \int_{U} \phi d\mu =\sum_{i=1}^{\infty}\int_{Q_i}\phi d\mu \leq \sum_{i=1}^{\infty}\int_{\Sigma_i}\phi=\int_{U} \phi d\nu,  
\end{equation*}
which proves that $(\mu,\chi_{\Sigma})$ satisfy \eqref{eq:subharmonic defi equiv L1}. \end{proof}
\begin{remark}\label{rem:existence general} While this is not needed for our purposes, the conclusion of Proposition \ref{prop: existence} holds for general $\mu \in \cS_{\delta}(U)$ such that $|\{\mu=1\}|=0$ (without assuming that $\mu \leq 1-\delta$), with a refined version of the same proof, if one either: (i) assumes that $U=\overline{U}^{\circ}$ or (ii) replaces the assumption $\mu \in \cS_{\delta}(U)$ by the slightly stronger condition that the pair $(\mu,\nu)$ in Definition \ref{def:Sdelta} satisfy \eqref{eq:subharmonic defi equiv L1}. This technicality arises due to the non-equivalence of \eqref{eq:subharmonic defi equiv L1}--\eqref{eq:rel subharmonic equiv} for ``punctured'' domains.    
\end{remark}

Next, we show that the closed null set $F_0$ of points that freeze at $t=0$ can be arbitrarily irregular, in the sense that, for any $\mu$, it can contain any given closed null set (and in particular, any closed fractal, including those of dimension $s\in (d-1,d]$). 

\begin{proposition}[Initial freezing can contain any closed null set] \label{prop: initial freezing} Let $U, \mu$ be as given in Proposition~\ref{prop: existence}, and let $F\subset U$ be any closed set of measure zero. Then there exists a weak solution $\eta$ to \eqref{eq:stefan intro} such that $F \subset F_0$, where $F_0$ is the set of initial freezing of $\eta$. \end{proposition}
\begin{proof}
 Let $\tilde{U}=U\backslash F$. Then, by Proposition \ref{prop: existence}, there exists a weak solution $\eta$ to \eqref{eq:stefan intro} with initial data $\mu \chi_{\tilde{U}}$. By Proposition \ref{prop:eulerian} and Definition \ref{def:weak solution}, we have, for the associated freezing time $s$,
 \begin{equation}\label{eq: spt eta fasf1}
     \{s>0\}=\limsup_{t \downarrow0}\{\eta(t,\cdot)>0\}\subset \tilde{U},\;\; |\tilde{U} \cap \{s=0\}|=0.
 \end{equation}
 
 Since $|U\backslash \tilde{U}|=0$, $\eta$ is also a weak solution to \eqref{eq:stefan intro} with initial data $\mu$, and, by \eqref{eq: spt eta fasf1}, $F\subset F_0$.
\end{proof}

The next example shows that the freezing time $s(x)$ may be discontinuous in the interior of $U$, and nucleation (i.e. spontaneous freezing away from the ice boundary) can occur at the initial time. Indeed, in the following construction, there exists $x_0\in U$ that freezes at $t=0$, and $x_0$ may be approximated by entire open annuli of points that \emph{never} freeze. The proof will be postponed to Section \ref{subsec:waiting}, since the construction involves radial maximal solutions, which have not yet been discussed. This proposition will not be used in the proof of any result.  

\begin{proposition}[Initial nucleation and waiting time]\label{prop:t=0 nucleation} Let $U=B_1$ and assume that $\mu\in \cS_{\delta}(U)$ is radial, with $\mu\leq1-\delta$ for some $\delta>0$. Then there exists a solution to \eqref{eq:stefan intro} with freezing time $s$ such that
\begin{equation*}
s(0)=0,  \quad B_{r^{(2)}_n} \backslash \overline{B}_{r_n^{(1)}} \subset \{s=\infty\},
\end{equation*}
where
\begin{equation*}
    0<r^{(1)}_n<r_n^{(2)}<1, \quad \lim_{n\to\infty}r^{(1)}_n=\lim_{n\to\infty}r^{(2)}_n=0.
\end{equation*}

\end{proposition}

\subsection{Markovian gluing and programmed freezing at positive times} \label{subsection:gluing}
 
We now show how to construct weak solutions to \eqref{eq:stefan intro} with a prescribed $F_t$ at any positive time $t> 0$ by using the Markov property of Brownian motion. Let us first recall the definition of the hitting time.\smallskip

\begin{definition}[Hitting time]
Let $R \subset [0,\infty) \times \mathbb{R}^d$ be a measurable set and let $(X_t)_{t \ge 0}$ be a stochastic process.  
The \emph{first hitting time} of $(t,X_t)$ to the set $R$ is defined by $\tau := \inf\{ t \ge 0 : (t, X_t) \in R \}$. For brevity, we also say that $\tau$ is the first hitting time of $X_t$ to $R$.
\end{definition}

We require the following result from \cite[Rem. 6.19]{beiglbock2017optimal}, which implies that the optimizer of $\mathcal{Q}(\mu,\nu)$ is characterized as the unique admissible stopping time that is also the first hitting time to a closed, nondecreasing set.

\begin{proposition}
    [\cite{beiglbock2017optimal}] \label{converse_OSP}  Let $0\leq \mu,\nu\in L^{\infty}(\R^d)$ be compactly supported,  with $\mu \leq_{\textnormal{SH}} \nu$. Suppose that there exists $s^*\in \operatorname{LSC}(\R^d;[0,\infty])$ such that the first hitting time
$\tau^*$ of
$(t,W_t)$ to $R^*:= \{(t,x): t\geq s^*(x)\}$ 
satisfies the admissibility condition  $W_0\sim \mu$, $W_{\tau^*}\sim \nu$, $\mathbb{E}[(\tau^*)^2]<\infty$.
Then  $\tau^*$ is the unique  optimizer of $\mathcal{Q}(\mu,\nu)$.
\end{proposition}

We begin with a useful lemma, which is an immediate consequence of \eqref{optimal_stopping_time} combined with the closedness and monotonicity of the barrier set $\{(t,x):s(x) \leq t\}$.

\begin{lemma} \label{hitting_time_sets} Let $\mu,\nu$ be as given in Proposition~\ref{converse_OSP}, and let $\tau^*$ be the optimizer of $\mathcal{Q}(\mu,\nu)$, and let $s(x)$ be the associated freezing time. Define the {\it active region at time $t$} as 
\begin{equation}\label{active_region} U_{\mu,\nu}(t)  := \{x:s(x)>t\}.  
\end{equation}
Then $\{\tau^* > t\} = \{ W_{t \wedge \tau^*} \in U_{\mu,\nu}(t) \}$ a.s.

 \end{lemma}

Note from \eqref{eq:s defi intro} and Lemma \ref{LSC_eta} one has that $U_{\mu,\nu}(t) = \{x:w(t,x)>0\}$ where $w$ is the associated potential variable. \smallskip

We are now ready to show the main technical result of this section, a \emph{Markovian gluing principle} that will later yield the generic non-uniqueness result of Theorem \ref{thm: fractal freezing}. It states that one can concatenate two optimal stopping times associated with two different Skorokhod problems, and the resulting stopping time is optimal for the corresponding glued problem. The key step below will be to show that, if $R_i$ denote the barriers associated with these Skorokhod problems, then as we will see below, the concatenated stopping time is the hitting time of the glued barrier \eqref{glued_barrier}. By Proposition \ref{converse_OSP}, it therefore suffices to verify that the glued barrier is closed and nondecreasing in time to conclude that this hitting time is optimal for the glued problem. In the context of the Stefan problem, $\nu_0$ below will correspond to the characteristic function of the transition zone of the first solution. Meanwhile, $\nu_1$ will be the characteristic function of the transition zone of the second solution combined with the region that has already frozen, at the gluing time, under the first solution
(see Proposition~\ref{cor:programmed_gluing_Stefan}).  
\begin{proposition}[Markovian gluing principle] \label{gluing_lemma} Let $0 \leq \mu_0,\nu_0,\nu_1 \in L^{\infty}(\R^d)$ be compactly supported, with $\mu_0 \leq_{\textnormal{SH}} \nu_0$, let $\tau_0$ be the optimizer of $\mathcal{Q}(\mu_0,\nu_0)$ with respect to the Brownian motion $W_t$.

For $t \geq 0$ let $\mu_t$ be the law of $W_{t \wedge \tau_0}$, and assume that there exists $t_1 > 0$ such that
 \begin{equation} \mu_{t_1} \leq_{\textnormal{SH}} \nu_1 \label{SH_ordering_glue_1_2}, \end{equation}
  with the property
 \begin{equation} U_{\mu_{t_1}, \nu_1}(0) \subset U_{\mu_0, \nu_0}(t_1).  \label{SH_ordering_glue_2_2} \end{equation}
\noindent Let $\tau_1$ be the optimizer of $\mathcal{Q}(\mu_{t_1},\nu_1)$ with respect to the Brownian motion $\tilde{W}_t := W_{t + (t_1 \wedge \tau_0)}$. Then the stopping time
 \begin{equation} \tau :=  \tau_0 \chi_{ \{ \tau_0 \leq t_1 \} } + ( t_1 +   \tau_1) \chi_{ \{\tau_0 > t_1\} }   \label{optimal_glued_stopping_time} \end{equation} optimizes $\mathcal{Q}(\mu_0,\nu_1)$. 
\end{proposition}

\begin{proof} Note first that, since $\mu_0$ and $\nu_0$ have bounded densities, $\mu_{t_1}$ does as well. Moreover, we trivially have $\mu_{t_1}\leq_{\operatorname{SH}}\nu_0$, and thus testing \eqref{def:subharmonic ordering} against convex functions implies that $\mu_t$ is supported in the convex hull of $\operatorname{supp}(\nu_0)$. Additionally, the strong Markov property implies that $\tilde{W}_t$ is a Brownian motion with initial distribution $\mu_{t_1}$, so $\tau_1$ is well-defined.

Since $W_0 \sim \mu_0$ and $\mathbb{E}[\tau^2]\leq \mathbb{E}[\tau_0^2]+\mathbb{E}[(t_1+\tau_1)^2]<\infty,$ Proposition \ref{converse_OSP} implies that, to show that $\tau$ is optimal, it suffices to show that $W_{\tau} \sim \nu_1$ and that $\tau$ is the first hitting time of $(t,W_t)$ to a closed and nondecreasing in time set. 

 First we will show that $W_{\tau} \sim \nu_1$. For any continuous and bounded function $f(x)$, we have
\begin{multline} \mathbb{E}[f(W_{\tau})] = \mathbb{E}[f(W_{\tau_0}) \chi_{ \{\tau_0 \leq t_1 \} }] + \mathbb{E}[f(W_{t_1 + \tau_1}) \chi_{ \{\tau_0 > t_1\} }] 
  \\= \mathbb{E}[f(W_{\tau_0}) \chi_{\{\tau_0 \leq t_1\}}] + \mathbb{E}[f(\tilde{W}_{\tau_1}) \chi_{\{\tau_0 > t_1\}}]. \label{decompose_expectation_123}   \end{multline}

The second term in the right hand side may be written as
 \begin{equation}
      \mathbb{E}[f(\tilde{W}_{\tau_1}) \chi_{\{\tau_0 > t_1\}}]  = \mathbb{E}[f(\tilde{W}_{\tau_1})] - \mathbb{E}[f(\tilde{W}_{\tau_1})\chi_{ \{ \tau_0 \leq t_1 \}}]\label{expectation_decompose_234}.
 \end{equation}  Now from \eqref{SH_ordering_glue_1_2}, \eqref{SH_ordering_glue_2_2}, and \eqref{optimal_stopping_time},
 one has $\tau_1 \leq \tau^{U_{\mu_0,\nu_0}(t_1)}$ (the exit time from $U_{\mu_0,\nu_0}(t_1)$). Hence, Lemma \ref{hitting_time_sets} implies that $\{\tau_0 \leq t_1\} = \{ \tilde{W}_0 \notin U_{\mu_0,\nu_0}(t_1)\}  \subset \{\tau_1=0\}$ a.s. In particular, in view of \eqref{expectation_decompose_234},
\begin{equation} \mathbb{E}[f(\tilde{W}_{\tau_1}) \chi_{\{\tau_0 > t_1\}}]   = \mathbb{E}[f(\tilde{W}_{\tau_1}) ]- \mathbb{E}[f(\tilde{W}_{0}) \chi_{ \{\tau_0 \leq t_1\} } ]. \label{decompose345_expectation}    \end{equation}
 Therefore, we deduce from \eqref{decompose345_expectation}, \eqref{decompose_expectation_123}, and $\tilde{W}_{\tau_1} \sim \nu_1$ that
\[ \mathbb{E}[f(W_{\tau})] = \mathbb{E}[f(\tilde{W}_{\tau_1})] = \int_{\R^d} f(x) d\nu_1(x), \] that is, $W_{\tau} \sim \nu_1$. Applying Proposition \ref{prop:subharmonic equivalent} to a large ball containing $\operatorname{supp}(\mu_0+\nu_1)$, we infer that $\mu_0 \leq_{\operatorname{SH}}\nu_1$.

It remains to show that $\tau$ is the first hitting time of $(t,W_t)$ to a nondecreasing in time closed barrier. For this purpose, let $R_i$ be optimal barriers for $\tau_i$ given by \eqref{optimal_stopping_time}, so that $\tau_0$ is the first hitting time of $(t,W_t)$ to $R_0$, and $\tau_1$ is the first hitting time of $(t,\tilde{W}_t)$ to $R_1$. We define the shifted barrier
\[ \tilde{R}_1 := \{ (t+t_1,x):(t,x) \in R_1 \}, \] and the glued barrier
\begin{equation} R := (R_0 \cap ([0,t_1] \times \R^d)) \cup (\tilde{R}_1 \cap ([t_1,\infty) \times \R^d)). \label{glued_barrier} \end{equation} 

\begin{figure}[h]
    \centering
    \includegraphics[width=0.60\textwidth]{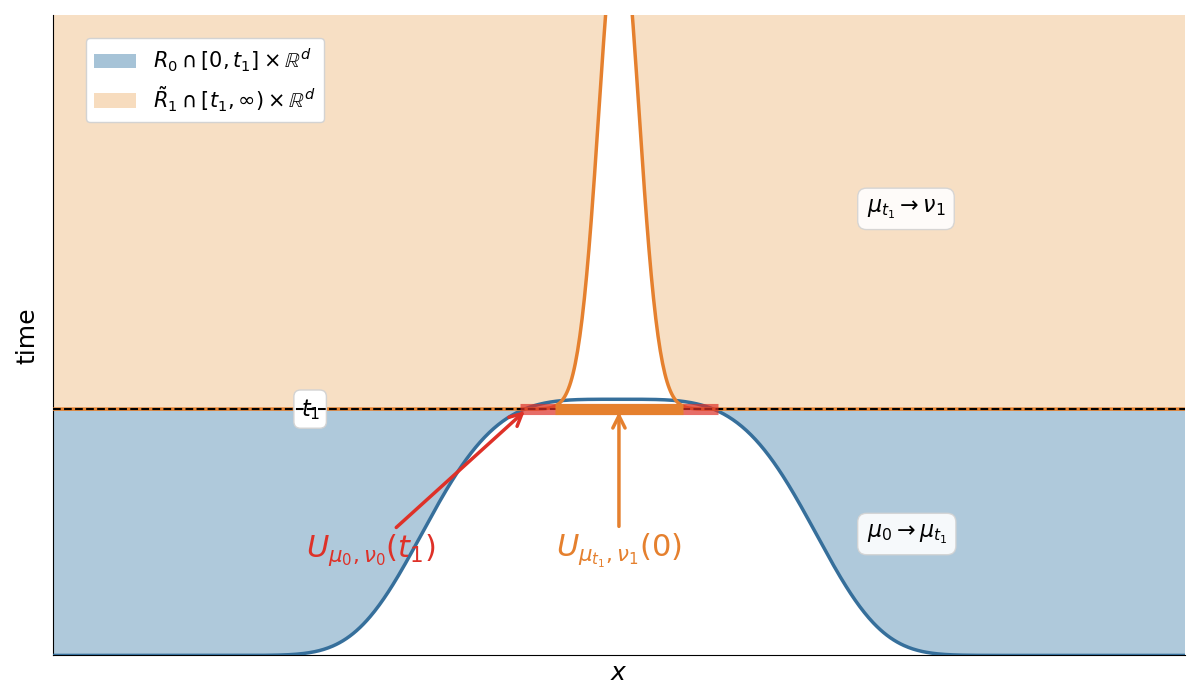}
    \caption{Illustration of the gluing procedure for stopping times at $t_1$. 
    The ordering of the barriers at time $t_1$ implies that the glued barrier is closed and is nondecreasing in time.}
\end{figure}
We first claim that $\tau$ is the first hitting time of $W_t$ to $R$. For a sample $\omega$ such that $\tau_0(\omega)\leq t_1$, $\tau(\omega)=\tau_0(\omega)\leq t_1$ is the first hitting time of $W_t$ to $R_0$, and thus to $R$.

If instead, a sample path satisfies $\tau_0(\omega)  >t_1$, then $\tau(\omega) = t_1 + \tau_1(\omega)$. From $\tau_0(\omega)>t_1$, we have $(t,W_t(\omega)) \notin R_0$ for $t \leq t_1$. In addition, along this sample path, $\tilde{W}_t(\omega) = W_{t + t_1}(\omega)$, so, by definition of $\tau_1$, $\tau_1(\omega)$ is the first hitting time of $(t,W_{t+t_1})$ to $R_1$. Hence, $\tau(\omega)=t_1+\tau_1(\omega)$ is the first hitting time of $W_t$ to $\tilde{R}_1\cap ([t_1,\infty) \times \R^d)$, and thus to $R$.

In summary, for any sample path, we have verified that $\tau$ is the first hitting time of $W_t$ to  $R$. 

Now, recalling the characterization \eqref{optimal_stopping_time} for the barriers $R_i$, \eqref{SH_ordering_glue_2_2} implies that
$R_0 \cap \{t=t_1\}   \subset \tilde{R}_1 \cap \{t=t_1\}$. This inclusion, along with the nondecreasing in time property of both $R_0$ and $R_1$ implies that $R$ is nondecreasing in time. We conclude by observing that $R$ is closed, since it is the union of two closed sets.  \end{proof}

\begin{corollary}[Programmed gluing of active particle distributions]\label{lem:gluing} 

Under the assumptions and notation of Proposition \ref{gluing_lemma},
\begin{equation} \label{eq:glued sol}
    \eta(t,x) :=
    \begin{cases}
        \eta_0(t,x) & t \in (0,t_1 ), \\[0.3em]
        \eta_1(t-t_1,x) & t \geq t_1,
    \end{cases}
\end{equation} is the active particle distribution associated to $(\mu_0,\nu_1)$, where $\eta_0$ and $\eta_1$ are the active particle distributions associated to $(\mu_0,\nu_0)$ and $(\mu_{t_1},\nu_1)$, respectively.
\end{corollary}

\begin{proof} 

By Proposition \ref{gluing_lemma}, the optimizer of $\mathcal{Q}(\mu_0,\nu_1)$ is given by $\tau$ from \eqref{optimal_glued_stopping_time}.
 Now for any $g \in C_b([0,\infty) \times \R^d)$, we have 
\begin{multline*} \mathbb{E} \left[ \int_0^{\tau} g(s,W_s) ds \right] = \mathbb{E} \left[ \chi_{ \{\tau_0 \leq t_1\} } \int_0^{\tau_0} g(s,W_s) ds \right] + \mathbb{E} \left[ \chi_{ \{\tau_0 > t_1\} } \int_0^{t_1+\tau_1} g(s,W_s) ds \right] \\ = \mathbb{E} \left[ \chi_{ \{\tau_0 \leq t_1\} } \int_0^{\tau_0} g(s,W_s) ds \right] + \mathbb{E} \left[ \chi_{ \{\tau_0 > t_1\} } \int_0^{t_1} g(s,W_s) ds \right] + \mathbb{E} \left[ \chi_{ \{\tau_0 > t_1\} } \int_{t_1}^{\tau_1+t_1} g(s,W_s) ds \right]\\
 = \mathbb{E} \left[ \int_0^{\tau_0 \wedge t_1} g(s,W_s) ds \right] + \mathbb{E} \left[ \chi_{ \{ \tau_0 > t_1\} } \int_{t_1}^{\tau_1+t_1} g(s,W_s) ds \right].
\end{multline*}
 Recalling that $\tilde{W}_t = W_{t + (t_1 \wedge \tau_0)}$, this yields
\begin{multline} \label{eq:dispeqactprt1pf} \mathbb{E} \left[ \int_0^{\tau} g(s,W_s) ds\right]= \mathbb{E} \left[ \int_0^{\tau_0 \wedge t_1} g(s,W_s) ds \right] + \mathbb{E} \left[ \chi_{ \{\tau_0 > t_1 \} } \int_{0}^{\tau_1} g(s+t_1,\tilde{W}_s) ds \right]   \\
= \mathbb{E} \left[ \int_0^{\tau_0 \wedge t_1} g(s,W_s) ds \right] + \mathbb{E} \left[  \int_{0}^{\tau_1} g(s+t_1,\tilde{W}_s) ds \right] - \mathbb{E} \left[ \chi_{\{ \tau_0 \leq t_1 \}} \int_0^{\tau_1} g(s+t_1,\tilde{W}_s)ds \right].   \end{multline} By Lemma \ref{hitting_time_sets}, we have \begin{equation} \label{eq:acpt pf 123}
    \{\tau_0 \leq t_1\} = \{W_{t_1 \wedge \tau_0} \notin U_{\mu_0,\nu_0}(t_1)\} = \{ \tilde{W}_0 \notin U_{\mu_0,\nu_0}(t_1)\} \quad \text{a.s.}
\end{equation}
On the other hand, from \eqref{SH_ordering_glue_1_2}, \eqref{SH_ordering_glue_2_2}, and \eqref{optimal_stopping_time}, 
we get $\tau_1 \leq \tau^{U_{\mu_0,\nu_0}(t_1)}$ (the exit time from $U_{\mu_0,\nu_0}(t_1)$). In view of \eqref{eq:acpt pf 123}, this implies that $\tau_1= 0$ a.s. on  $\{\tau_0 \leq t_1\}$.
Hence, recalling Definition \ref{def:eulerian_variables}, we deduce from \eqref{eq:dispeqactprt1pf} that
\begin{multline*}
     \mathbb{E} \left[ \int_0^{\tau} g(s,W_s)  ds \right] = \mathbb{E} \left[ \int_0^{\tau_0 \wedge t_1} g(s,W_s) ds \right] + \mathbb{E} \left[  \int_{0}^{\tau_1} g(s+t_1,\tilde{W}_s) ds \right] \\
     = \mathbb{E} \left[ \int_0^{\tau_0} g(s,W_s) \chi_{\{s \leq t_1\}} ds \right] + \mathbb{E} \left[  \int_{0}^{\tau_1} g(s+t_1,\tilde{W}_s) ds \right] = \int_0^{t_1} \int_{\R^d} g(s,x) \eta_0(s,x) dsdx\\
     + \int_{t_1}^{\infty} \int_{\R^d} g(s,x) \eta_1(s-t_1,x) dsdx   = \int_{\R^d} \int_0^{\infty}  g(s,x) \eta(s,x) dsdx.
\end{multline*}
\end{proof}
We see next that this programmed gluing of the active particle distributions implies a gluing of weak solutions of \eqref{eq:stefan intro}. Recall that the freezing time $s$ for a solution $\eta$ of \eqref{eq:stefan intro} has been given in \eqref{eq:s defi intro}, with $\{s\leq t\}$ describing the frozen set by time $t$. \smallskip

For the proof, we introduce the following notation. If $(\eta,\rho)$ denotes the Eulerian variables associated to $(\mu,\nu,\tau)$ (see Definition \ref{def:eulerian_variables}), we disintegrate $\rho$ 
with respect to its spatial marginal $\nu$, interpreting $\rho(ds,x)$ as $\nu(x)$ times the 
conditional law of $\tau$ given $W_\tau=x$:
\begin{equation}
\int_{0}^{t}\rho(ds,x):=\nu(x)\,\mathbb{P}[\tau\le t\,|\,W_\tau=x].
 \label{def_integral_rho}
\end{equation}
Also, by decomposing $\mu_t$ from Proposition \ref{gluing_lemma} into active particles plus frozen particles  (see the arguments of \cite[Thm. 4.15]{kim2024stefan} or \cite[Lem. 4.11]{chu2025nonlocal}), we may write
\begin{equation}
    \mu_t(x)  = \eta(t,x)  + \int_0^t \rho(ds,x). \label{mu_t_formula}
\end{equation}

\begin{proposition}[Programmed gluing for the Stefan problem]  \label{cor:programmed_gluing_Stefan} Let $U$ be a bounded open set, and let $\mu \in \cS_{\delta}(U)$ for some $\delta \in(0,1)$ with $|\{\mu=1\}|=0$. Let $\eta_0$ solve \eqref{eq:stefan intro} with initial data $\mu$. For any $t_1 > 0$, let $\eta_1$ solve  \eqref{eq:stefan intro} with initial data $\eta_0(t_1,\cdot)$. If the freezing times $s_i$ of $\eta_i$  satisfy  
\begin{equation}\label{freezing_assumption}
    \{x:s_1(x) >0 \} \subset \{x:s_0(x)>t_1\},
\end{equation} 
then  
\begin{equation} \label{glue_stefan}
    \eta(t,x) :=
    \begin{cases}
        \eta_0(t,x) & t \in (0,t_1) \\[0.3em]
        \eta_1(t-t_1,x) & t \geq t_1
    \end{cases}
\end{equation} is a weak solution of \eqref{eq:stefan intro}.
\end{proposition}

\begin{proof} For simplicity, we denote $\mu_0 := \mu$ and $\tilde{\mu}_1(\cdot) := \eta_0(t_1,\cdot)$. Let $\nu_0 := \chi_{\Sigma_0}$ and $\tilde{\nu}_1 := \chi_{\Sigma_1}$, where $\Sigma_i:= \{ 0<s_i(x)<\infty\}$. By Proposition \ref{prop:eulerian} (d),
\begin{equation} \mu_0 \leq_{\textnormal{SH}} \nu_0 \text{ and } \tilde{\mu}_1 \leq_{\textnormal{SH}} \tilde{\nu}_1. \label{the_SH_order123} \end{equation} Then $\eta_0$ and $\eta_1$ are, respectively, the active particle distributions associated to $(\mu_0,\nu_0)$ and $(\tilde{\mu}_1, \tilde{\nu}_1)$ \cite[Thm. 5.6]{kim2024stefan}. Let $\rho_0$ denote the stopped particle distribution corresponding to $(\mu_0,\chi_{\Sigma_0})$.
Let $\tau_0$ and $\mu_{t_1}$ be as in Proposition \ref{gluing_lemma}, and define 
\begin{equation}\label{eq:nu1 defi pf4.2sc}\nu_1 := \tilde{\nu}_1 + \int_0^{t_1} \rho_0(ds,x).\end{equation} 
By \eqref{mu_t_formula}, we have $\mu_{t_1} = \tilde{\mu}_1 + \int_0^{t_1} \rho_0(ds,x).$ Thus, in view of  \eqref{the_SH_order123},  $\mu_{t_1}$ satisfies $\mu_{t_1}\leq_{\textnormal{SH}}\nu_1$. Additionally, \eqref{SH_ordering_glue_2_2} is satisfied due to \eqref{freezing_assumption}. Thus, by Corollary \ref{lem:gluing}, $\mu_0 \leq_{\textnormal{SH}} \nu_1$, and  $\eta$ is the active particle distribution associated to $(\mu_0, \nu_1).$ 

First, we check that $\nu_1 \in \{0,1\}$. Indeed, as $\rho_0 \sim (\tau_0, W_{\tau_0})$ and $W_{\tau_0} \sim \chi_{\Sigma_0}$, we have from \eqref{def_integral_rho} that $\int_0^{t_1} \rho_0(ds,x) = \chi_{\Sigma_0}(x) \cdot \mathbb{P}[\tau_0 \leq t_1 | W_{\tau_0}=x  ].$ Now by \eqref{optimal_stopping_time} and using that $\{s_0\leq t\}$ is closed and nondecreasing in time, we have that $\mathbb{P}[\tau_0 \leq t_1 | W_{\tau_0}=x  ] = \chi_{\{s_0(x)\leq t_1\}},$ and thus
\begin{equation} \label{eq:intrho0t1sec42pf}    \int_{0}^{t_1}\rho_0(ds,x)=\chi_{\Sigma_0 \cap \{s_0(x)\leq t_1\}}.  
\end{equation}
From \eqref{freezing_assumption}, we have $\Sigma_1 \subset \{s_0(x)>t_1\}$, so \eqref{eq:nu1 defi pf4.2sc} and \eqref{eq:intrho0t1sec42pf} yield $\nu_1 \in  \{0,1\}.$

Next, we check that $\Delta^{-1}(\nu_1-\mu_0)(x)=0$ for $x \notin U$. By Definition \ref{def:weak solution}, $\eta_0$ vanishes outside $(0,\infty)\times U$, so from \eqref{freezing_assumption}, $\eta_1$ also vanishes outside $(0,\infty)\times U$. Therefore, from \eqref{glue_stefan}, $\eta$ vanishes outside $(0,\infty)\times U$. We then infer from Lemma \ref{LSC_eta} that if $w_i$ is the associated potential variable for $\eta_i$, then $w_i(t,x) = 0$ for $x \notin U$. Hence, for any $x \notin U$
\[ \Delta^{-1}(\nu_1-\mu_0)(x) = \underbrace{\Delta^{-1}(\nu_1-\mu_{t_1})(x)}_{w_1(0,x)} + \underbrace{\Delta^{-1}(\nu_0-\mu_0)(x)}_{w_0(0,x)} + \underbrace{\Delta^{-1}(\mu_{t_1}-\nu_0)(x)}_{-w_0(t_1,x)} = 0.   \] We thus conclude from Proposition \ref{prop:eulerian} and  Remark \ref{rem:weak sol non lipschitz} that $\eta$ is a weak solution of \eqref{eq:stefan intro}. \end{proof}
Finally, we use the above gluing property to extend Proposition \ref{prop: initial freezing} for positive times, thereby proving the main result of this section.

\begin{proof}[Proof of Theorem \ref{thm: fractal freezing}]    
 Part (i) of the statement is  Proposition \ref{prop: initial freezing}.  For (ii), note that since  $(\chi_{\{\eta_0>0\}})_t\leq 0$, $\eta_0$ is subcaloric. Thus, if $t_1>0$ is sufficiently large, then $\eta_0(t_1) \leq 1-\delta$ for some $\delta>0$. Set $U_1=\{x:\eta_0(t_1,\cdot)>0\}=\{x:s_0(x)>t_1\}$, where $s_0$ is the freezing time of $\eta_0$, and let $F \subset U_1$.
  
Proposition \ref{prop: initial freezing} then yields a solution $\eta_1$ to \eqref{eq:stefan intro} on $(0,\infty)\times \R^d$ with initial data $\eta_0(t_1,\cdot)$ and transition zone $\Sigma_1$, such that its freezing time $s_1$ satisfies
\begin{equation} \label{eq: lF0+y pf}
    F \subset \{x:s_1(x)=0\}
\end{equation}

Since $ \{x:s_1(x)>0\}\subset \{x:\eta_0(t_1,x)>0\}=U_1=  \{x: s_0(x)> t_1\}$, Proposition \ref{cor:programmed_gluing_Stefan} implies that the glued function $\eta$ defined by \eqref{glue_stefan} solves \eqref{eq:stefan intro}, and by \eqref{eq: lF0+y pf}, its  freezing time $s$ satisfies
\begin{equation*}
   F \subset \{x:s(x)=t_1\}. 
\end{equation*}
To see that \eqref{eq:freeze dim d t0 intro} may occur, it suffices to choose $F$ with $\dim_{\mathcal{H}} (F)=d$. 
\end{proof}

\section{Transition zone and regularity of maximal solutions} \label{sec:regularity}

In this section, we switch gears to show that  maximal solutions  enjoy substantial regularity properties. Our main goal will be to analyze the fine structure of the transition zone to prove Theorem \ref{thm: regularity intro}. We then explain (Corollary \ref{cor:regu}), in light of the local results of \cite{EKM25}, how our study of the transition zone yields a thorough regularity theory for maximal solutions, and precise dimension estimates on the singular set. In particular, we show that maximal solutions do not generate fractal freezing, except at a Hausdorff zero-dimensional set of times, and only if $|\{\mu>1\}|>0$.   Roughly speaking, this says that, while certain maximal solutions could potentially exhibit the irregularities featured in Theorem \ref{thm: fractal freezing}, this cannot happen in the subcritical case $\mu\leq 1$, and would be a worst case scenario in general. We then show by example that, even for maximal solutions, the set $F_0$ of initial freezing can be arbitrarily irregular.

 We recall that the transition zone is the set
\begin{equation}
    \Sigma=\{x\in U: 0<s(x)<\infty\}.
\end{equation}
Since this section is dedicated to analyzing the fine structure of $\Sigma$, we will be careful  to avoid identifying $\Sigma$ with sets that agree on a set of measure zero. Throughout, $U$ will be a Lipschitz domain, and we will work with a fixed choice for the superharmonic weight $u$. \smallskip

The dual formulation of $\mathcal{P}(\mu, u)$ plays a central role in the present section. Let $\nu^*$ be the optimal target measure for $\mathcal{P}(\mu,u)$. By \cite[Thm. 6.5]{choi2024existence},  $\nu^*=\chi_{\Sigma}$, and
\begin{equation} \label{eq:sigma a.e.}
   |\Sigma \Delta \{\phi<0\}|=0, 
\end{equation}
with $(\phi,\varphi)$ being any maximizer of the dual problem
\begin{equation} \label{dualp}
    \mathcal{D}(\mu,u)= \sup_{(\phi,\varphi)\in \mathcal{F}}\left[\int \phi(y)dy-\int \varphi (x)\mu(dx)\right],
\end{equation}
\[\mathcal{F}:=\{(\phi,\varphi)\in L^1(U): \phi \leq 0, \varphi \geq \phi -u, \;\Delta \varphi \leq0 \}\]
where the inequality $\Delta \varphi \leq 0$ is interpreted in the distributional sense. Part of our analysis will consist of refining the rough statement \eqref{eq:sigma a.e.}, to give a more precise characterization of $\Sigma$ in terms of the optimizer of the dual problem. We begin by reducing the dual variables.

\begin{lemma}\label{lem:dual_simple} Under the assumptions of Theorem \ref{thm: regularity intro}, we have

\begin{multline} \label{eq:dualp simple}
       \mathcal{D}(\mu,u)= -\inf_{\psi\in L^1(U), \Delta \psi \geq 0}\left\{\int (\psi-u)^+ -\int\psi \mu\right\} \\=-\inf_{\psi\in  L^1(U),\; \psi\geq0, \Delta \psi \geq 0}\left\{\int (\psi-u)^+ -\int\psi \mu \right\},
\end{multline}    
and $\psi \geq 0$ is an optimizer of \eqref{eq:dualp simple} if and only if the pair $(\min(-\psi+u,0),-\psi)$ is an optimizer of \eqref{dualp}. Moreover, if $\Sigma$ is the transition zone, then
\begin{equation} \label{eq:sigma=psi u}
    |\Sigma \Delta \{\psi>u\}|=0.
\end{equation}
\end{lemma}
\begin{proof}Let $(\phi,\varphi)$ be any optimizer of \eqref{dualp}. Replacing $\phi$ with $\min(\varphi+u,0)$ preserves admissibility, and only increases the value of the objective. Moreover, replacing, in two steps, first $\varphi$ with $\min(\varphi,0)$, and then $\phi$ with $\min(0,\varphi+u)$, also preserves admissibility at each step and only increases the value of the objective. This proves that
   \begin{multline}
       \mathcal{D}(\mu,u)= \sup_{\varphi\in L^1(U),-\Delta\varphi\geq0}\left\{\int \min(\varphi+u,0) -\int\varphi\mu \right\} \\=\sup_{\varphi\in L^1(U),-\Delta\varphi\geq0, \; \varphi\leq 0}\left\{\int \min(\varphi+u,0) -\int\varphi\mu \right\}.  \label{eq:new_dual_formula}
\end{multline}     
The claim now follows by setting $\psi = -\varphi$, noting that \eqref{eq:sigma=psi u} is a restatement of \eqref{eq:sigma a.e.}.
\end{proof}
In order to characterize the precise relationship between $\Sigma$ and the optimizer, we will need the key property that $\Sigma$ has no hidden ``cracks'' of measure zero inside $\{s > 0\}$, in the precise measure-saturation sense below. We also observe that some concentration of $\Sigma$ occurs at every boundary point of $U$.
\begin{lemma}[Measure-saturation of the transition zone]\label{lem:sigma saturation} Let $\eta$ be a weak solution to \eqref{eq:stefan intro}. The transition zone $\Sigma$ is measure-saturated  in the following sense. For any open set $V \subset U$ the following holds:
\begin{itemize}
    \item [(i)] If $|V \backslash \Sigma|=0$ and $V \subset \{s>0\}$, then $V \subset \Sigma$. 
    \item [(ii)] If $|V \cap \Sigma|=0$, then $V \cap \Sigma = \emptyset$ and $V \cap \{s=0\}=\emptyset$.
\end{itemize}
Moreover, we have $|B_r(x)\cap \Sigma|>0$ for any $x\in \partial U$ and any $r>0$.
\end{lemma}
\begin{proof}
    Let $w(t,x)=\frac12\int_{t}^{\infty}\eta(s,x)ds$ be the potential variable. We note that, since $\eta$ is subcaloric in $(0,\infty)\times \mathbb{R}^d$ and compactly supported with bounded initial data, the Poincar\'e inequality and a standard Gronwall argument show that 
    $\eta$ decays exponentially, and, in particular,
 \begin{equation} \label{eq:w limit t infty}
\lim_{t\to \infty}\|w(t,\cdot)\|_{L^{\infty}(\mathbb{R}^d)}= \lim_{t\to \infty}\|\eta(t,\cdot)\|_{L^{\infty}(\mathbb{R}^d)}=0.
 \end{equation} 
 Assume first that $| V \setminus \Sigma |=0$ and $V \subset \{s>0\}$, and let $x_0 \in V$. Since $|V\setminus\Sigma|=0$, for sufficiently small $r>0$ such that $B_{2r}(x_0)\subset V$, $w$ satisfies
 \begin{equation} \label{eq:obst local pf}
     w_t-\frac12\Delta w =-\frac12\chi_{\{w>0\}} \quad \text{in } \; (0,\infty)\times B_{2r}(x_0)
 \end{equation}
 in the distributional sense. By the nondegeneracy property of the parabolic obstacle problem \cite[Lem. 5.1]{caffarelli2004regularity}, if $w(t,x_0)>0$, then, for some dimensional constant $c_d>0$,
 \begin{equation*}
     \sup_{B_r(x_0)\times [t-r^2,t]}w\geq c_dr^2. 
 \end{equation*}
 By \eqref{eq:w limit t infty}, we infer that there exists $t_0>0$ such that $w(t_0,x_0)=0$, which proves that $s(x_0)<\infty$, and thus $x_0 \in \Sigma$ and $V \subset \Sigma$.

Suppose now that $|V \cap \Sigma |=0$ and let $x_0 \in V$. Then, since $V$ is open, we deduce as before that there exists $r>0$ such that $w$ solves the heat equation distributionally (and thus classically) in $(0,\infty)\times B_{2r}(x_0)$. Since $w\geq 0$ and $w(0,\cdot)>0$ a.e. in $B_{2r}(x_0)$ (recall that $\limsup_{t\to 0^+} \{\eta(x,t)>0\} = U$ a.e.), the strong maximum principle implies that $w(t,x_0)>0$ for all $t$, namely $x_0\in \{s=\infty\}$. This shows that $V \cap \Sigma =\emptyset$ and $V\cap \{s=0\}=\emptyset$.

Finally, if $x\in \partial U$, then assuming by contradiction that $|B_r(x)\cap \Sigma|=0$ for some $r>0$, we infer that $\Delta w(0,\cdot)=-\frac12\mu \leq 0$ a.e. in $B_r(x)$, which  contradicts the strong minimum principle  \cite[Thm. 8.19]{GilTru}.
\end{proof}

We can now return to the dual problem and analyze more precisely how the dual minimizer encodes the geometry of the transition zone. In particular, the following result already shows that $\Sigma$ is a.e. equal to an open set, providing a negative answer to the conjecture in \cite[Rem. 6.2]{choi2024existence}.

\begin{proposition}[Regularity of dual minimizers and structure of $\Sigma$] \label{prop:sigma structure} Under the assumptions of Theorem \ref{thm: regularity intro}, for a minimizer $\psi$  of \eqref{eq:dualp simple}, its unique USC representative  is harmonic in the open set $E=\{x:s(x)>0\}$. Letting $f:=(\psi-u)|_{E}$, the transition zone $\Sigma$ satisfies
\begin{equation} \label{eq:sigma structure}
    \{f>0\}\subset \Sigma \subset \{f>0\} \cup \{f=0, \nabla f=0\}.
\end{equation}
    \end{proposition}
Moreover, if $x_0\in \partial \Sigma$ satisfies an exterior ball condition, then $x_0\notin \Sigma$.
 \begin{proof}   By duality \cite[Thm. 4.5]{choi2024existence},  we have $\mathcal{D}(\mu,u)=\mathcal{P}(\mu,u)$, that is,
\[-\int (\psi-u)^+dy + \int \psi d\mu = \int u(y)d\nu(y),\]
where $\nu=\chi_{\Sigma}$.  In view of \eqref{eq:sigma a.e.}, this may be rewritten as 
\[\int_{\Sigma} (-\psi(y)+u(y))dy + \int \psi(x)d\mu(x)=\int u(y)d\nu(y),\]
which means 
\[\int \psi d\nu=\int \psi d\mu.\]
Letting $v=\Delta^{-1}(\nu-\mu)=w(0,\cdot)$, since $U$ is a Lipschitz domain, Lemma \ref{lem:integration by parts} (ii) yields
\begin{equation*}
 \int_{U}(v\Delta \psi) dx=0.
\end{equation*}
Since $v\geq 0$ and $\Delta \psi\geq 0$, this proves that $\psi$ is harmonic in the open set $E=\{v>0\}=\{s>0\}$. 

 Recall that  $|\{f>0\} \Delta \Sigma|=0$ due to Lemma~\ref{lem:dual_simple}. Note that, by definition, $\{f>0\} \subset \{s>0\}$, and the sets $\{f>0\}$ and $\{f<0\}$ are  open since $f$ is continuous in $E$. Hence we infer from Lemma \ref{lem:sigma saturation} that 
 \begin{equation}
    \label{eq:f>0 subs sig pf} \{f>0\} \subset \Sigma \subset \{f \geq 0\}.
 \end{equation}

Suppose now that there exists $x_0\in\Sigma$ such that  $f(x_0)=0$ and $\nabla f(x_0)\neq 0$. Then the set $\{f=0\}$ is a smooth hypersurface in a neighborhood of $x_0$, and, in particular, the set $\{f<0\}$ has an interior ball $B$ at $x_0$. Since $x_0\in \Sigma\subset\{f\geq 0\}$, this ball $B$ is exterior to $\Sigma$ at $x_0$, namely $B\subset U\backslash \Sigma$ and $x_0\in \partial B \cap \partial\Sigma$. We will show now that the exterior ball condition at $x_0\in\partial\Sigma$ implies $x_0\notin \Sigma$, from which we can conclude.  Note that the potential variable $w$ is positive and caloric in $B \times (0,\infty)$. In particular, for each $t_0>0$, 
 \begin{equation} \label{eq:w(t0) elliptic pf}
     w(t_0,\cdot)>0, \quad -\Delta w(t_0,\cdot)=-2w_t(t_0,\cdot)\geq0\;\; \;\;\text{ in } B.
 \end{equation}
On the other hand, if we had $w(t_0,x_0)=0$, then $\nabla w(t_0,x_0)=0$ due to its $C^{1,\alpha}_x$ regularity, which by \eqref{eq:w(t0) elliptic pf} would contradict Hopf's lemma. Thus $w(t,x_0)>0$ for all $t>0$, that is, $x_0\not \in \Sigma$.
\end{proof}
We may now prove the main result of this section.
\begin{proof}[Proof of Theorem \ref{thm: regularity intro}]
The measure-saturation property is simply Lemma \ref{lem:sigma saturation}. Letting $f$ be defined on $E=\{s>0\}$ as in Proposition \ref{prop:sigma structure}, we have $\Delta f=-\Delta u>0$ in $E$, and thus
\begin{equation} \label{eq: f=0gf=0 cont pf}
\{f=0,\nabla f=0\}\subset \bigcup_{i=1}^d\{f_{x_i}=0, f_{x_ix_i}>0 \},
\end{equation}
with the right hand side being the union of $d$ smooth hypersurfaces. Since $\{f=0, \nabla f\neq 0\}$ is a smooth hypersurface, we infer that
\begin{equation*}
    \dim_{\mathcal{H}}(\{f=0\})\leq d-1.
\end{equation*}
By \eqref{eq:sigma structure}, we have $\Sigma \backslash \Sigma ^{\circ}\subset \partial \Sigma \cap\{s>0\} \subset \{f=0\}$, and $\dim_{\mathcal{H}}(A)\leq \dim_{\mathcal{H}}(\partial\Sigma \cap \{s>0\})\leq d-1$ is proven. A posteriori, we infer that $|\partial \Sigma|=0$ from the fact that, by Definition \ref{def:weak solution}, $|U \backslash \{s>0\}|=0$.

Assume now that $u$ is real analytic, and let $x_0 \in \Sigma \backslash \Sigma^\circ$. We note first that, by the saturation property of Lemma \ref{lem:sigma saturation}, we must have $|B_r(x_0) \cap \{f<0\}| >0$ for every $r>0$, which implies that $x_0 \in  \overline{ \{f<0\} }$. Therefore, by \eqref{eq:sigma structure} and \eqref{eq: f=0gf=0 cont pf}, it is enough to show that, $\operatorname{dim}_{\mathcal{H}}(S)\leq d-2$, where
\begin{equation*}
    S=\{f=0, \nabla f=0,\; f_{x_d x_d}>0\}\cap \overline{\{f <0 \}}.
\end{equation*}
We fix $x_0 \in S$, and write $x':=(x_1,\ldots,x_{d-1})$. With no loss of generality, we may assume that $x_0=0$. Since $u$ is analytic and $\psi$ is harmonic in $\{s>0\}$, $f=\psi-u$ is analytic in $\{s>0\}$. 

By the Weierstrass preparation theorem \cite[Thm. 6.1.3]{krantz2002}, we may write, for some open rectangle $R\times (-\varepsilon,\varepsilon) \subset \R^{d-1}\times \R$ centered at $0$,
\begin{equation*}
    f(x',x_d)=h(x',x_d)(x_d^2+b(x')x_d+c(x')),\;\; (x',x_d)\in R\times (-\varepsilon,\varepsilon),
\end{equation*}
where $h$, $b$, and $c$ are real analytic functions, with $h > 0$ (because $f_{x_d x_d}(x_0)>0$). Let $D(x'):=(b^2-4c)(x')$ on $R$, and suppose that $(y',y_d)\in S \cap (R\times (-\varepsilon,\varepsilon))$. Since $h>0$ and $f=|\nabla f| =0$ at $(y',y_d)$, $y_d$ is a double root  of the one-variable polynomial $x \mapsto x^2+b(y')x+c(y')$. That is,
\begin{equation*}
    D(y')=0, \quad y_d=-b(y')/2.
\end{equation*}
We therefore have
\begin{equation*}
    S \cap R \subset \{(x',x_d)\in R\times(-\varepsilon,\varepsilon) : D(x')=0, \;x_d=-b(x')/2\}.
\end{equation*}
It is then enough to show that
\begin{equation*}
    \text{dim}_{\mathcal{H}}(\{D=0\})\leq d-2.
\end{equation*}
Since $D$ is a real analytic function on $R \subset \R^{d-1}$, it is enough to prove that $D \not \equiv 0$. But if we had $D \equiv 0$, then
\begin{equation*}
    f(x',x_d)=h(x',x_d) (x_d+b(x')/2)^2 \geq 0 \;\text{ for every }\; (x',x_d)\in R\times (-\varepsilon,\varepsilon),
\end{equation*}
contradicting the fact that $x_0\in S \subset \overline{\{f<0\}}.$
\end{proof}

Next, we apply the regularity results of \cite{EKM25} for decreasing solutions to the parabolic obstacle problem, to derive some important consequences of Theorem \ref{thm: regularity intro} for the free boundary regularity of maximal solutions, the size of their singular set, and the time frequency of high-dimensional fractal freezing. 

\begin{corollary}[Free boundary regularity and size of the singular set] \label{cor:regu}Under the assumptions of Theorem \ref{thm: regularity intro}, the freezing time $s(x)$ is $C^1$ in $\Sigma^{\circ}$, and is $C^{\infty}$ in the set $\Sigma^{\circ}\cap\{\nabla s\neq0\}$. Moreover, we may write
\begin{equation*}
  \partial\{(t,x):\eta(t,x)>0, \;s(x)>0\}=R \dot \cup S, 
\end{equation*}
where:
\begin{itemize}
    \item[(i)] $R \subset (0,\infty)\times \R^d$, the set of regular points, is relatively open in $\partial\{\eta>0\}$. The free boundary is a $C^{\infty}$ hypersurface advancing with finite speed near every point of $R$.
    \item[(ii)]  $S \subset (0,\infty)\times \R^d$ is the set of singular points where $\partial\{\eta>0\}$ either attains infinite speed or has a discontinuous freezing time. One has $\dim_{\operatorname{par}}S\leq d+1$ and, if the generating weight $u$ is analytic, $S$ satisfies the sharp estimate
\begin{equation} \label{eq:dim par sing}
    \dim_{\operatorname{par}}(S)\leq d.
\end{equation}
\end{itemize}
Moreover, one has
\begin{equation} \label{eq:fractals are rare in time}
    \dim_{\mathcal{H}}(\{t: \dim_{\mathcal{H}}(\{x:s(x)=t\})>d-1\})=0,
\end{equation}
and, if $\mu \leq 1$, then the above set of times is empty.
\end{corollary}
\begin{proof}
If $x_0\in \Sigma^{\circ}$, then, recalling \eqref{obstacle_problem_123}, there exists a neighborhood $V= (t_0-r^2,t_0+r^2) \times B_r(x_0)$ such that the associated potential $w$ solves the parabolic obstacle problem
\begin{equation}\label{eq:obst pf th}\begin{cases}
     w_t -\frac12 \Delta w=-\frac12\chi_{\{w>0\}}\\
     w\geq 0, \;w_t \leq 0,\; \{w>0\}=\{w_t<0\}
\end{cases}   
\end{equation}
in $V$. The regularity result for $s$ then follows from \cite[Thm. 1.2]{EKM25}, and, by \cite[Thm. 1.1]{EKM25}, one may write
\begin{equation*}
    \partial\{(t,x): \eta(t,x)>0, x\in \Sigma^{\circ}\}=R \dot \cup S_1, 
\end{equation*}
where $R$ satisfies (i), and $S_1=\{(s(x),x)\in \partial\{\eta>0\}: x\in \Sigma^\circ, \nabla s(x)=0 \}$, the set of points where the boundary attains infinite speed, satisfies
\begin{equation*}
    \dim_{\operatorname{par}}(S_1)\leq d.
\end{equation*}
On the other hand, let $S_2:=\{(t,x)\in \partial\{\eta>0\}: x\in \Sigma \setminus \Sigma^{\circ}\},$ be the set of free boundary points where the freezing time is discontinuous. By the definition of parabolic Hausdorff dimension, we have \begin{equation*}
    \dim_{\operatorname{par}}(S_2)\leq  \dim_{\mathcal{H}}(\Sigma \backslash \Sigma^{\circ})+2.
\end{equation*} Letting $S:=S_1 \dot \cup S_2$, \eqref{eq:dim par sing} follows from Theorem \ref{thm: regularity intro}.

Similarly, \cite[Thm. 1.1, Thm. 1.3]{EKM25} implies that
\begin{equation} \label{eq:fractalsrare1pf}
    \dim_{\mathcal{H}}(\{t:\dim_{\mathcal{H}}(\{x:t=s(x), (t,x)\in R\cup S_1\})>d-1\})=0,
\end{equation}
with the above set being empty if $\mu\leq 1$. Since $\dim_{\mathcal{H}}(\Sigma \setminus \Sigma^{\circ})\leq d-1$, we trivially have
\begin{equation} \label{eq:fractalsrare2pf}
\{t:\dim_{\mathcal{H}}(\{x:t=s(x), (t,x)\in S_2\})>d-1\}=\emptyset.
\end{equation}
Hence, \eqref{eq:fractals are rare in time} follows from \eqref{eq:fractalsrare1pf} and \eqref{eq:fractalsrare2pf}, as does the upgraded statement when $\mu\leq 1$.
\end{proof}
\begin{remark}\label{rem:dim par sharp} We note that, by Theorem \ref{thm: fractal freezing}, \eqref{eq:dim par sing} is sharp for weak solutions to \eqref{eq:stefan intro}. In fact \eqref{eq:dim par sing} is also sharp for local solutions to \eqref{eq:obst pf th}, since it is attained, for any $t_0>0$, by the simple example
\begin{equation*}
    w(t,x)=(t_0-t)^+.
\end{equation*}
See, however, \cite[Thm. 1.3]{EKM25}.    

Let us also mention that, from the point of view of optimal Skorokhod embeddings (see Section \ref{sec:prelim}), the freezing time $s$ is precisely the Root‑type barrier associated with the optimal stopping time realizing $\mu \rightsquigarrow \chi_{\Sigma}$. The $C^1$ regularity result of Corollary \ref{cor:regu} may therefore be understood as a partial high-dimensional analogue of the one-dimensional continuity theorem for Root's barrier obtained in \cite{bayraktar2022continuity}.

\end{remark}

Note that the above theorem excludes the discussion of the initial freezing set $\{x\in U:s(x)=0\}$. In fact, the initial freezing of a maximal solution can be as pathological as wanted, as long as it is a measure zero set, as we show next (see Remark \ref{rem:maximal S0}).
\begin{proposition}[Fractal freezing at $t=0$ in the class $\mathcal{S}_0(U)$] \label{fractal_freezing_maximal} Let $U$ be an open bounded domain with Lipschitz boundary and let $F \subset U$ be a closed set with $|F|=0$.  Then there exists an initial datum $\mu \in C(\overline{U}) \cap \mathcal{S}_0(U)$  with $|\{\mu=1\}|=0$ such that there exists a unique weak solution $\eta$ of \eqref{eq:stefan intro}. The unique solution $\eta$ is also a maximal solution for any weight $u$, and if $s$ denotes its  freezing time, then 
\[ F \subset \{x \in U:s(x)=0\}. \] 

\end{proposition}

\begin{proof} We first construct $\mu$. For $\tilde{U} := U \setminus F$, there exist open disjoint cubes $Q_i=\prod_{j=1}^{d} (a^{(i)}_j,b^{(i)}_j)$ such that
$\tilde{U} = \bigcup_{i=1}^{\infty} \overline{Q_i}$. Letting
\[ \psi_i(x) := \prod_{j=1}^d (x_j-a^{(i)}_j)^4(b^{(i)}_j-x_j)^4 ,   \] we have $\psi_i \geq 0$ with $\psi_i(x)=|\nabla \psi_i(x)|=\Delta \psi_i =0$ on $\partial Q_i$. Choose $\e_i > 0$ sufficiently small such that $\e_i \Delta \psi_i < 1$ on $Q_i$, and define
\[ \mu_i(x) := \begin{cases} 1 - \e_i \Delta \psi_i(x) &\text{ if } x \in Q_i, \\ 0 &\text{ otherwise. } \end{cases} \] Observe that $|\{\mu_i=1\}|=0$ by analyticity of $\psi_i$. If $\varphi\in L^1(Q_i)$ is a subharmonic function over $Q_i$, then
\[ \int_{\R^d} \varphi\mu_i  = \int_{\R^d} \varphi - \int_{Q_i} \e_i \Delta\psi_i \cdot \varphi \leq \int_{\R^d} \varphi \chi_{Q_i}.  \] The last inequality is justified as in the proof of Lemma \ref{lem:integration by parts},  by approximating $\psi_i$ with smooth, compactly supported functions in $Q_i$, and using the definition of the distributional Laplacian (together with $\psi_i=|\nabla \psi_i|=0$ on $\partial Q_i$). By Proposition \ref{prop:subharmonic equivalent}, we infer that $\mu_i \leq_{\textnormal{SH}} \chi_{Q_i}$ over $Q_i$.  We then define $\mu \in C(\overline{U})$ by
\[ \mu(x) := \begin{cases} \mu_i(x) &\text{ if } x \in Q_i \\ 1 &\text{ if } x \in \overline{U} \setminus \tilde{U}.  \end{cases} \]
Since $\mu_i \leq_{\textnormal{SH}} \chi_{Q_i}$ over $Q_i$, Proposition \ref{prop:subharmonic equivalent} implies that $\Delta^{-1}(\chi_{Q_i}-\mu_i) \equiv 0$ outside $Q_i$. Therefore,
\begin{equation} \Delta^{-1}(\chi_U-\mu) = \sum_{j=1}^{\infty} \Delta^{-1}(\chi_{Q_i} - \mu_i)  \equiv 0 \text{ outside } \tilde{U}. \label{subharmonic_sum_measures} \end{equation} Proposition \ref{prop:eulerian} and Remark \ref{rem:weak sol non lipschitz} thus imply the existence of a weak solution $\eta$ of \eqref{eq:stefan intro} with initial data $\mu$ and transition zone a.e. equal to $U$. By \eqref{subharmonic_sum_measures}, the associated freezing time $s$ satisfies
\[ F \subset \{x\in  U:s(x)=0\}. \]

We now show uniqueness. Suppose that $\tilde{\eta}$ is a weak solution of \eqref{eq:stefan intro} with initial data $\mu$. If $\tilde{\Sigma}$ denotes the transition zone of $\tilde{\eta}$, then Proposition \ref{prop:eulerian} (d) implies that $\mu \leq_{\textnormal{SH}} \chi_{\tilde{\Sigma}}$.  In particular, as $\mu(\R^d)=|U|$, we see from $\mu \leq_{\textnormal{SH}} \chi_{\tilde{\Sigma}}$ that $|U|= |\tilde{\Sigma}|$. But as $\tilde{\Sigma} \subset U$ due to Definition \ref{def:weak solution}, we conclude that $\tilde{\Sigma}= U$ a.e. That is $\tilde{\eta}$ has the same initial data and transition zone a.e. as $\eta$, so \cite[Cor. 9.10]{kim2024stefan} implies that $\tilde{\eta}=\eta$. 

Finally, we show that $\eta$ is a maximal solution for any weight $u$. Since $\mu(\R^d)=|U|$ the only admissible target measure for $\mathcal{P}(\mu,u)$ is $\nu(x) = \chi_U(x)$. Thus, $\nu$ is the optimizer for $\mathcal{P}(\mu,u)$, so $\eta$ is the maximal solution associated to $u$.

\end{proof}

\section{Non-universality and waiting time}
\label{sec:non-universality}
Our main goal in this section will be to prove Theorem \ref{thm:universality}. We first show that, in the radial and one-dimensional cases, maximal solutions enjoy a universal transition zone that we characterize explicitly. We also show that with the slack condition $\mu \in\cS_{\delta}(U)$ for $\delta>0$, the solutions exhibit no initial freezing. 

Next, we prove that in dimensions $d\geq 2$, maximal solutions with non-radial data are not universal, as they may depend on the weight $u$. We also exhibit an additional instability phenomenon for maximal solutions, by constructing, in Proposition \ref{prop:fourier example}, solutions that display waiting time near the initial free boundary.
\subsection{Universality of radial and one-dimensional solutions} \label{subsec:radial}
We begin by proving universality of the primal problem $\mathcal{P}(\mu,u)$ when $d=1$, or when both $\mu$ and $u$ are radial, by identifying the optimal target measure $\nu$. We note that the argument below covers the case $\delta=0$, which falls outside the scope of the results of \cite{choi2024existence}, where well-posedness of the primal problem was studied. However, the result is new even in the case $\delta>0$, since uniqueness was only known for a \emph{fixed choice} of $u$.

\begin{proposition} [Universality for one-dimensional or radial solutions] \label{target_measure_universial} Assume that 
\begin{equation*}U=B_1 \quad \text{or} \quad U=A_{\rho}:=\{\rho<|x|<1\}, \; \;\;\rho\in (0,1) \quad \text{or} \quad U=(0,1), \;\;\;d=1.    
\end{equation*}

Then for any $\mu \in \mathcal{S}_0(U)$ and any weight $u\in C^2(\overline{U})$, both assumed radial if $d\geq 2$,
the unique optimal target measure for $\mathcal{P}(\mu,u)$ is

\begin{equation} \label{eq:nu* radial}
  \nu^*(x) =
  \begin{cases}
    \chi_{(0,a)}(x) + \chi_{(b,1)}(x), 
      & \text{if } U = (0,1), \\[6pt]
    \chi_{\{\tilde r < |x| <1\}}(x), 
      & \text{if } U = B_1, \\[6pt]
    \chi_{\{\rho < |x| < r_1\}}(x) 
      + \chi_{\{r_2 < |x| < 1\}}(x), 
      & \text{if } U = A_\rho,
  \end{cases}
\end{equation} where $(a,b)$ are from \eqref{one_dimension_mass_moments}, 
$\tilde r$ is from \eqref{same_mass_ball}, 
and $(r_1,r_2)$ are from \eqref{eq:annuli_constraints}.
\end{proposition}

\begin{proof} We prove the case $U = A_{\rho}$; the remaining two cases follow from a similar argument. 

Since $\mu \in \mathcal{S}_0(A_{\rho})$, the admissible set for $\mathcal{P}(\mu,u)$ is nonempty. We first show that there exists a radial competitor. Given any competitor $\nu$, we consider its symmetrization
\begin{equation} \label{eq:T(nu) defi}
    T(\nu)(x) := \int_{O(d)} \nu(Ax)\, dA,
\end{equation}
where $O(d)$ is the orthogonal group equipped with the Haar probability measure. Then $T(\nu)$ is radial, and since $\nu \le \chi_{A_{\rho}},$ we immediately obtain $T(\nu)(x) \le \chi_{A_{\rho}}(x).$

Next, we verify that $\mu \le_{\mathrm{SH}} T(\nu)$ on $A_{\rho}$.  
Let $\varphi \in L^{1}(A_{\rho})$ be subharmonic on $A_{\rho}$, and similarly  consider the symmetrization $T(\varphi)$. Then, by Fubini's theorem and the radiality of $\mu$, $T(\varphi)  \in L^1(A_{\rho})$ is subharmonic on $A_{\rho}$, and 
\begin{equation}
   \int_{\mathbb{R}^d} T(\varphi)(x)\, d\mu(x)
= \int_{\mathbb{R}^d} \varphi(x)\, d\mu(x) \quad \text{ and } \quad  \int_{\R^d} T(\varphi)(x) d\nu(x) = \int_{\R^d} \varphi(x) T(\nu)(x) dx.
\end{equation}
Hence, we obtain from $\mu \leq_{\textnormal{SH}} \nu$ over $A_{\rho}$ and Proposition \ref{prop:subharmonic equivalent} that
\[
\int_{\mathbb{R}^d} \varphi(x)\, d\mu(x)
= \int_{\mathbb{R}^d} T(\varphi)(x) \, d\mu(x)
\le \int_{\mathbb{R}^d} T(\varphi)(x) \, d\nu(x)
= \int_{\mathbb{R}^d} \varphi(x) T(\nu)(x) dx,
\]
and thus $\mu \le_{\mathrm{SH}} T(\nu)$ on $A_{\rho}$.  That is, $T(\nu)$ is a radial competitor for $\mathcal{P}(\mu,u)$. \smallskip

 As equality in \eqref{eq:subharmonic defi equiv L1} holds for harmonic functions over $A_{\rho}$, we may apply it to $\varphi(x)=1$ and $\varphi(x)=\mathcal{N}(x)$ to obtain that
\[  \int_{\rho}^1 r^{d-1} \mu(r) dr = \int_{\rho}^1 r^{d-1} T(\nu)(r) dr \quad \text{and} \quad\int_{\rho}^1 \phi(r)\mu(r) dr = \int_{\rho}^1 \phi(r) T(\nu)(r)dr,  \] where $\phi$ is as in Proposition \ref{prop:annuli_target}. 
 Thus, Proposition \ref{prop:annuli_target} implies 
 \begin{equation} T(\nu) \leq_{\textnormal{SH}} \nu^* \text{ on } A_{\rho},  \label{raidal_target_ordering} \end{equation} and hence $\mu \leq_{\textnormal{SH}}  \nu^* \text{ on } A_{\rho},$ which means $\nu^*$ is admissible for $\mathcal{P}(\mu,u)$.
 Now if $\nu_r$ is an admissible radial target measure, then by \eqref{raidal_target_ordering}, $T(\nu_r)=\nu_r \leq_{\textnormal{SH}} \nu^*$ on $A_{\rho}$, so that
\begin{equation} \int_{\R^d} u(x) d\nu^*(x) \leq \int_{\R^d} u(x) d\nu_r(x), \label{best_radial} \end{equation} because $u$ is superharmonic on $A_{\rho}$. To upgrade this to any admissible target measure, we observe from Fubini's theorem and the radiality of $u$ that
\[ \int_{\R^d} u(x) d\nu(x)  = \int_{\R^d} u(x) T(\nu)(x) dx  \geq \int_{\R^d} u(x) d\nu^*(x), \] where we used \eqref{best_radial}. Hence, $\nu^*$ is an optimizer of $\mathcal{P}(\mu,u)$. \smallskip

To obtain uniqueness, let $\tilde{\nu}$ be another optimizer of $\mathcal{P}(\mu,u)$. Then the above arguments imply that $T(\tilde{\nu}) \leq_{\textnormal{SH}} \nu^*$, so that $v(x) := \Delta^{-1}(\nu^*-T(\tilde{\nu})) \geq 0$. Optimality and Lemma \ref{lem:integration by parts} then imply
\[ 0 = \int_{\R^d} u(x) \Delta v(x) dx = \int_{\R^d} \Delta u(x) \cdot v(x) dx = 0. \]  
Thus, since $\Delta u < 0$ and $v(x) \geq 0$, recalling that $v$ is continuous, we obtain $v \equiv 0$. This implies that $\nu^* = T(\tilde{\nu})$.  Hence, $T(\tilde{\nu})\in \{0,1\}$ and, since $\tilde{\nu}\leq1$, \eqref{eq:T(nu) defi} forces $\tilde{\nu}=T(\tilde{\nu})$, and thus $\tilde{\nu}=\nu^*$.\end{proof}

When $\delta>0$, the above result is enough to uniquely characterize the maximal solution, and to identify the transition zone $\Sigma$ a.e. We show next that, in fact, $\Sigma$ can be identified exactly, and the set of initial freezing can be shown to be empty.

We first prove that, \emph{conditional} on $\Sigma$ being a.e. equal to an outer shell, the set of initial freezing is empty. For a later application, we do not assume below that $u$ is radial.
\begin{lemma}[No initial freezing in a.e. annular transition zones]\label{lem:a.e. annular implies no freezing}  Let $\mu\in \cS_{\delta}(U)$, with $\delta>0$, $|\{\mu=1\}|=0$, and with $\mu$ radial if $d\geq 2$. Let $\Sigma$ be the transition zone for the maximal solution associated to some weight $u$, and assume that $\nu^*=\chi_{\Sigma}$ is a.e. given by one of the three cases in \eqref{eq:nu* radial}. Then the freezing time satisfies 
\begin{equation*}
   U \cap \{s=0\}=\emptyset.
\end{equation*}
\end{lemma}

\begin{proof} The problem is equivalent to showing that $v:=\Delta^{-1}({\chi_{\Sigma}-\mu})$ satisfies
\begin{equation*}
    v>0 \quad \text{in}\quad  U.
\end{equation*}    
Assume first that the second case in \eqref{eq:nu* radial} holds. Note that since $\chi_{\Sigma}-\mu$ is a.e. equal to a radial function, $v$ must be radial, so we may write $v=v(r)$, and we then have
\begin{equation*}
    (r^{d-1}v'(r))'=r^{d-1}(\chi_{\Sigma}-\mu).
\end{equation*}
Integrating twice from $r$ to $1$, we obtain
\begin{equation*}
    v(r)=\int_{r}^1 \frac{1}{s^{d-1}}\int_{s}^1t^{d-1}(\chi_{\Sigma}-\mu)(t)dt ds.
\end{equation*}
Letting $\sigma_{d-1}$ be the surface area of $\mathbb{S}^{d-1}$,  $N(|x|):=\mathcal{N}(x)$ (with  $\Delta \mathcal{N}=\delta_{0}$), and
$\phi_r(x):=(\cN(x)-N(r))_+$ in  $U$, integration by parts yields
\begin{multline} \label{eq:v rep formula radial}
    v(r)=\sigma_{d-1}\int_r^1 (N(s)-N(r))(\chi_{\Sigma}-\mu)s^{d-1}ds=\int_{B_1\setminus B_r}(N(|x|)-N(r))(\chi_{\Sigma}-\mu)dx\\=\int_{B_1}\phi_r(x)(\chi_{\Sigma}-\mu)dx.
\end{multline} 
Since $\mathcal{N}$ is subharmonic in $U=B_1$, so is $\phi_r$  for every $r\in(0,1)$, with $\phi_r \in L^1(U)$. Since $\mu \in \cS_{\delta}(U)$, letting $\nu_0$ be as in Definition \ref{def:Sdelta}, we have, by  Proposition \ref{prop:subharmonic equivalent},
\begin{equation*}
    \int_{U}\phi_r \mu \leq \int_U\phi_r \nu_0 \leq (1-\delta)\int_{U}\phi_r. 
\end{equation*}
Hence, if $r\in [\tilde{r},1)$, $\chi_{\Sigma} = 1$ a.e. in the support of $\phi_r$, and \eqref{eq:v rep formula radial} yields
\begin{equation*}
    v(r)=\int_{B_1}\phi_r(x)(1-\mu)dx\geq \delta \int_{B_1}\phi_r(x)dx>0.
\end{equation*}
On the other hand, we have $\Delta v=-\mu\leq0$ in $B_{\tilde{r}}$ with $v>0$ on $\partial B_{\tilde{r}}$, so the minimum principle yields $v>0$ in $B_{\tilde{r}}$, completing the proof.

The third case of \eqref{eq:nu* radial} can be argued similarly. The fact that $v>0$ in the outer shell $B_1 \setminus \overline{B}_{r_2}$ can be shown with identical arguments. As for the inner shell $B_{r_1}\setminus \overline{B_{\rho}}$, one obtains again through integration by parts that
\begin{equation*}
    v(r)=\int_{B_r \setminus B_{\rho}}(N(r)-N(|x|))(\chi_{\Sigma}-\mu)dx=\int_{U}\varphi_r(x)(\chi_{\Sigma}-\mu)dx,
\end{equation*}
where  $\varphi_r(x):=(N(r)-\cN(x))_+$ in $U$. The key observation here is  that, since in this case $0\notin U$, $\mathcal{N}$ is harmonic (instead of merely subharmonic) in $U$, which makes $-\mathcal{N}$ subharmonic. Therefore, $\varphi_r$ is subharmonic in $U$. The proof that $v>0$ in $\{\rho<|x|\leq r_1\}$ concludes as in the first case by testing $\varphi_r$ in \eqref{eq:integration by parts}. On  $\{r_1 < |x|< r_2\}$ we have $\Delta v = -\mu \le 0$, and $v>0$ on  $\partial B_{r_1}\cup\partial B_{r_2}$. By the minimum principle, it follows that $v>0$ in $\{r_1\leq |x|\leq r_2\}$, which concludes the proof.

Finally, for the first case in \eqref{eq:nu* radial}, after translation and rescaling we may assume that $U=(1/2,1)$, and after reflection, we may assume that $\mu$ is even and $U=(-1,-1/2)\cup (1/2,1)$. But this is just a particular instance of the third case in \eqref{eq:nu* radial}.
\end{proof}

We now establish the key structural result about radial and one-dimensional maximal solutions, proving that their transition zone is a pure outer shell of the domain (see Figures \ref{fig:non-universality} and \ref{fig:sigma} for examples of non-radial transition zones), and no freezing occurs at the initial time.

\begin{proposition}[Transition zone and no-initial-freezing of radial solutions]\label{prop:radial transition zone} Let $\mu \in \cS_{\delta}(U)$ be radial with $|\{\mu=1\}|=0$ and $\delta>0$. Let $\Sigma$ be the transition zone of the unique maximal solution to \eqref{eq:stefan intro}  associated with a weight $u$. Then, under the notation and assumptions of Proposition \ref{target_measure_universial},
\begin{equation}
 \Sigma =\begin{cases}{\{0<x<a}\}\cup \{b<x<1\} & \text{if} \quad U=(0,1),\\
 \{\tilde{r}<|x|<1\} &\text{if} \quad U=B_1,\\
 \{\rho<|x|<r_1\} \cup \{r_2<|x|<1\} & \text{if} \quad U=\{\rho<|x|<1\}.
 \end{cases}
\end{equation}
In all cases, the set of initial freezing $U\cap \{s=0\}$ is empty.
\end{proposition}
\begin{proof}By Proposition \ref{target_measure_universial}, $\Sigma$ is a.e. equal to the desired, uniquely determined open annular region, which we denote by $\Sigma'$. Thus, by Lemma \ref{lem:a.e. annular implies no freezing}, the set $U\cap\{s=0\}$ is empty. Finally, we may upgrade the a.e. equality to exact equality as follows: by Lemma \ref{lem:sigma saturation}, we have $\Sigma '\subset \Sigma \subset \overline{\Sigma'}$. Moreover, by Proposition \ref{prop:sigma structure}, no point of $\Sigma\cap \partial\Sigma$ may satisfy an exterior ball condition: thus $\Sigma \subset \Sigma'$.
\end{proof}
\begin{remark}\label{rem:counterexample subharmonic L1} One may now see that the three-way equivalence concluded in Proposition \ref{prop:subharmonic equivalent} is false without the assumption that $U=\overline{U}^{\circ}$. Indeed, let $d\geq 2$, let $(\mu,u)$ be any radial pair of initial data with $\{\mu>0\}=B_1$, and let $v=\Delta^{-1}(\chi_{\Sigma}-\mu)$, where $\Sigma$ is the unique annular transition zone given by Proposition \ref{prop:radial transition zone}. Since $B_1$ and $U:=B_1\setminus \{0\}$ have the same exit time $\tau^U$ a.s., we have $\mu\chi_{U} \leq_{\operatorname{SH}} \chi_{\Sigma}$ with respect to $U$. On the other hand, $v(0)\neq0$ (since the set of initial freezing $\{x\in B_1:v(x)=0\}$ is empty), so $v=\Delta^{-1}(\chi_{\Sigma}-\mu)=\Delta^{-1}(\chi_{\Sigma}-\mu\chi_{U})$ does not vanish outside of $U$. Thus \eqref{eq:pot vanish outside U equiv} and \eqref{eq:rel subharmonic equiv} are not equivalent for the pair $(\mu\chi_{U}, \chi_{\Sigma})$.
\end{remark}

\subsection{Non-universality of maximal solutions for dimensions $d \geq 2$} 
Having shown that maximal solutions are universal when the data is radial or when $d=1$, we prove below that universality is false without one of these assumptions.
\begin{figure}[h!]
    \centering
    \begin{minipage}{0.40\textwidth}
        \centering
        \includegraphics[width=\textwidth]{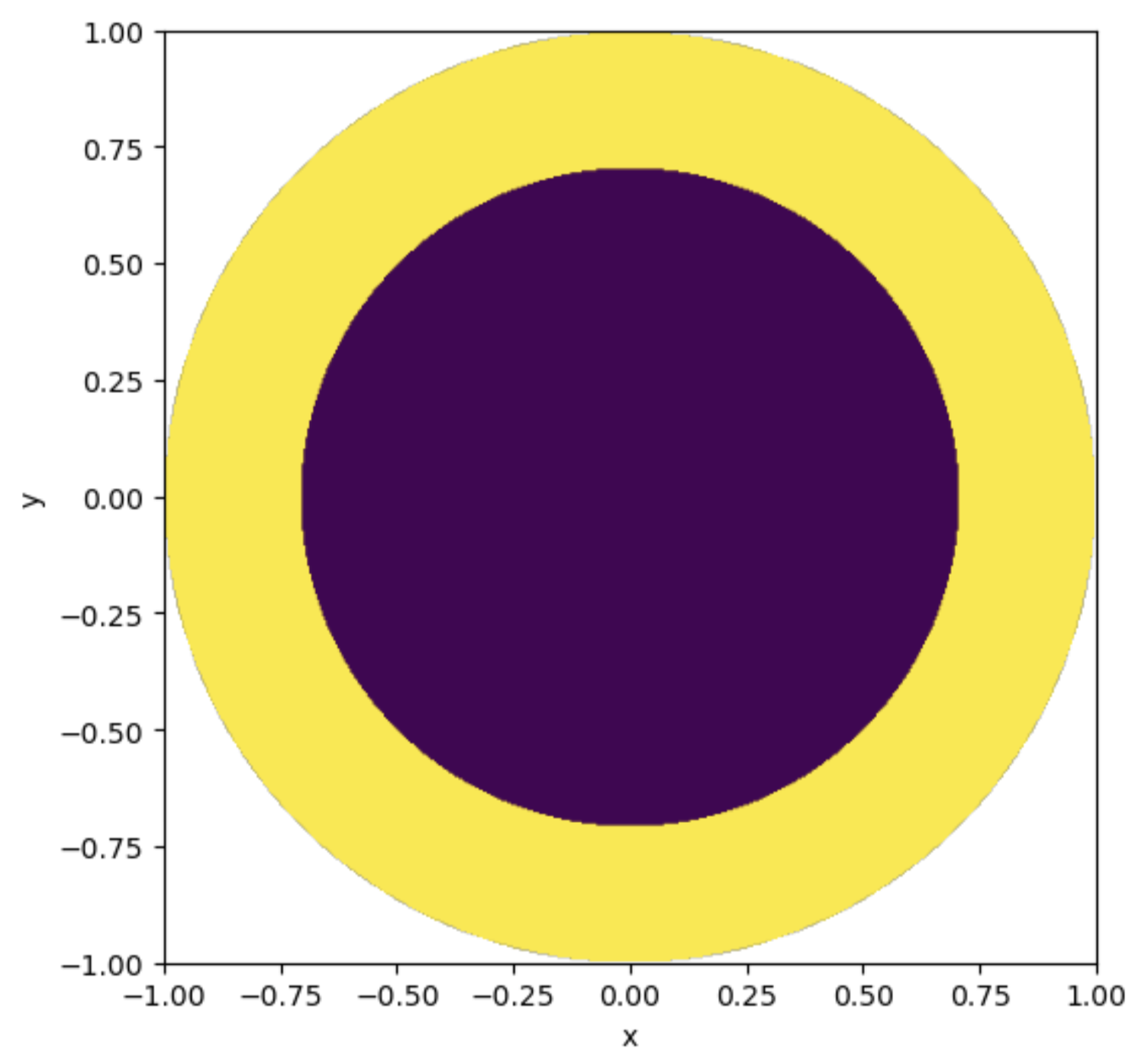}
    \end{minipage}\hfill
    \begin{minipage}{0.42\textwidth}
        \centering
        \includegraphics[width=\textwidth]{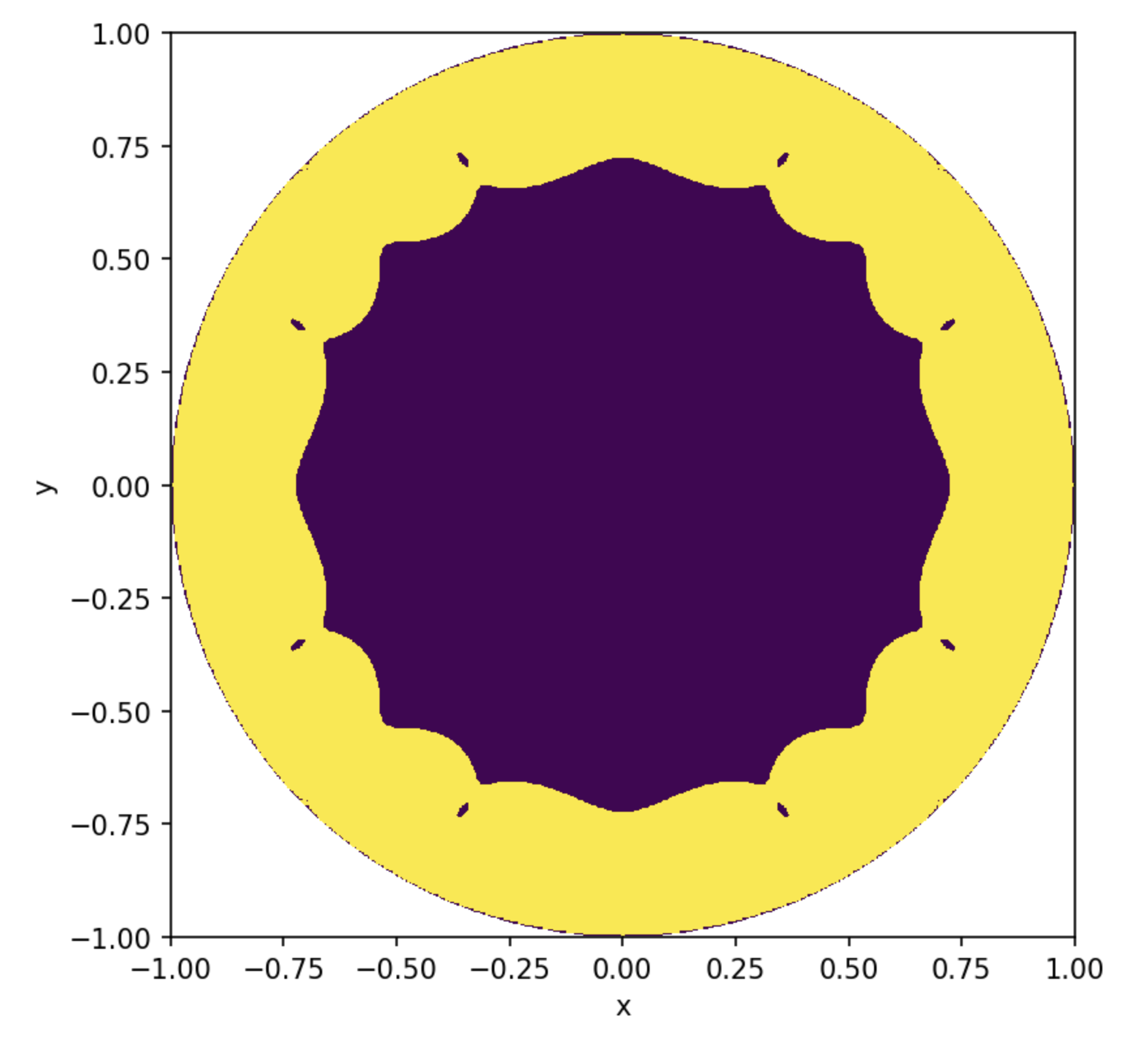} 
    \end{minipage}
    \caption{Transition zone (yellow) of maximal solutions illustrated for $\mu=\frac12 \chi_{B_1}$, and weight functions $u(x,y)=2-(x^2+y^2)$ and $u(x,y)=2-300x^{12}y^{12}-10^{-4}(x^2+y^2)$, respectively.}
    \label{fig:non-universality}
\end{figure}

\begin{proposition}[Non-universality of maximal solutions]\label{no_universality_prop}
    Let $d \ge 2$ and $U=B_1$. Consider a radial $\mu\in \cS_{\delta}(U)$, with $\delta>0$ and $|\{\mu=1\}|=0$. Let $\eta$ be the maximal solution to \eqref{eq:stefan intro} with weight $u$, and let $\Sigma$ be its transition zone. Then the following holds:
\begin{itemize}
\item[(i)] If $u$ is radial, then $\Sigma=B_1 \backslash \overline{B_{\tilde{r}}}$ for some $\tilde{r}\in (0,1).$
\item[(ii)] \emph{(Robust non-analytic counterexample)} If the restriction $u|_{\partial B_{\tilde{r}}}$ is not real analytic for some $0<\tilde{r}<1$, then $\Sigma \not = B_1 \backslash \overline{B_{\tilde{r}}}$. 
\item[(iii)] \emph{(Explicit analytic counterexample in $d=2$)} Let $d=2$, $\mu \equiv \tfrac12 \chi_U$, and define
\begin{equation} \label{eq:analytic u counter}
    u(x,y):=1-15x^4y^4-\varepsilon (x^2+y^2), \quad (x,y)\in U \quad\hbox{ for a sufficiently small }\varepsilon >0.
\end{equation}
Then $u$ is an admissible weight, and $\Sigma \not = B_1 \backslash \overline{B_{\tilde{r}}}$ for any $\tilde{r}\in (0,1)$.
\end{itemize}
In particular, in any dimension $d\geq 2$, maximal solutions depend on the choice of $u$.
\end{proposition}
\begin{proof}Part (i) follows from  Proposition \ref{prop:radial transition zone}.
Assume now that $\Sigma=B_1 \setminus \overline{B_{\tilde{r}}}$ for some $\tilde{r} \in (0,1)$. By Lemma \ref{lem:a.e. annular implies no freezing}, $U\cap \{s=0\}=\emptyset$. Thus, by Proposition \ref{prop:sigma structure}, any optimizer $\psi$ to the dual problem \eqref{eq:dualp simple} is harmonic in $B_1$, and 
\begin{equation*}
    \overline{\Sigma}=\overline{B_1 \setminus B_{\tilde{r}}}=\{\psi \geq u\}.
\end{equation*}
In particular, $\psi=u$ on $\partial B_{\tilde{r}}$, implying that $u$ is analytic on $\partial B_{\tilde{r}}$, which shows (ii) by contra-positive.

For case (iii), we note first that $u$ is an admissible weight: indeed the strict superharmonicity of $u$ is immediate for every $\varepsilon>0$, and the positivity for small $\varepsilon$ follows because $\sup_{(x,y)\in B_1}x^4y^4=\frac{1}{16}.$ Finally, assume by contradiction that $d=2$, $\mu\equiv \frac12\chi_{U}$, and $u$ is given by \eqref{eq:analytic u counter}. Note that, since $\mu$ and $\chi_{\Sigma}$ must have the same mass, we have $\tilde{r}=1/\sqrt{2}$.  Since $\psi$ is harmonic, by uniqueness of the Dirichlet problem on the ball, this determines $\psi$ uniquely on $B_{1/\sqrt{2}}$. Furthermore, by unique continuation of harmonic functions, this actually determines $\psi$ uniquely on all of $B_1$. A posteriori, since $u$ is a polynomial, $\psi$ is a harmonic polynomial that can be computed exactly in polar coordinates by solving the Dirichlet problem on $B_{1/\sqrt{2}}$. One then obtains
\begin{equation*}
(\psi-u)(1,0)=-(45/2048) +\varepsilon/2<0  
\end{equation*}
for $\varepsilon<\frac{90}{2048}$. However, this contradicts the fact that $(1,0)\in \overline{\Sigma}=\{\psi \geq u\}$. 

\end{proof}
 See Figure \ref{fig:non-universality} for a different polynomial example. This proposition illustrates a sort of converse phenomenon to what occurs in Theorem \ref{thm: regularity intro}: while analyticity of the weight improves the regularity of $\Sigma$, part (ii) shows that non-analyticity of $u$ already forces the transition zone to deviate from the simple annular geometry. In this sense, non-analyticity of $u$ \emph{by itself} can generate geometric irregularities and ill-posedness in the class of maximal solutions. The explicit example of part (iii) shows, however, that analyticity of $u$ is not enough to guarantee universality (this polynomial example is adapted from a similar construction in harmonic $L^1$ approximation theory \cite[Sec. 4]{GoHaRo}).

We may now show the main result of this section.
\begin{proof}[Proof of Theorem \ref{thm:universality}] By Proposition \ref{no_universality_prop}, universality fails in dimension $d\geq 2$.

In dimension $1$, any open bounded set $U\subset \mathbb{R}$ is a disjoint union of countably many bounded intervals. Hence, it suffices to establish universality on a single interval, and by translation and scaling we may reduce to the unit interval $(0,1)$. Similarly, in the radial case, it suffices to consider the case where $U$ is an annulus or a ball. 

By Proposition~\ref{target_measure_universial}, for any strictly superharmonic weight $u$, the optimal target measure $\nu^*$ of $\mathcal{P}(\mu,u)$ is independent of the choice of $u$ (radial $u$ if $d\geq 2$). Hence, all such weights induce the same active particle distribution, and therefore the corresponding maximal solutions coincide.  \end{proof}

\subsection{Waiting time and initial nucleation}\label{subsec:waiting}
Before discussing waiting times for maximal solutions, we first prove Proposition \ref{prop:t=0 nucleation}, which earlier provided an example of a non-maximal solution with waiting time and initial-time nucleation.

\begin{proof}[Proof of Proposition \ref{prop:t=0 nucleation}] For each $k\geq1$, let $A_k:={B_{1/k}} \backslash \overline B_{1/(k+1)}$. Let $\eta_k$ be the maximal solution to \eqref{eq:stefan intro} with initial data $\mu|_{A_k}$ associated to any radial weight (e.g. let $u(x):= 2-|x|^2$). By Proposition \ref{prop:radial transition zone} (applied to a rescaled solution), $\eta_k$  is a radial function, the transition zone $\Sigma_k$ of this solution is the union of two outer annuli in $A_k$, and the set of initial freezing for $\eta_k$ is empty.  By the arguments of Proposition \ref{prop: existence}, it follows that there exists a solution to \eqref{eq:stefan intro} with initial data $\mu$, given by the glued function
\begin{equation*}
    \eta(t,x)=\eta_k(t,x), \quad x\in A_k.
\end{equation*}
 By construction, using the fact that the set $\{s=0\}$ is closed, the solution $\eta$ satisfies the required properties.
\end{proof}

We now show by example that maximal solutions, too, can exhibit waiting times. In fact, our example will show that this pathology can occur for smooth measures $\mu$ that arise as arbitrarily small perturbations of the well-behaved radial solutions of Section \ref{subsec:radial}, displaying the inherent instability of the problem.

\begin{proposition}\label{prop:fourier example} Let $d=2$, $\delta>0$, and $U:=B_1 \backslash \overline{B_{\frac12}}$. Consider a radial $\mu_0\in \cS_{\delta}(U)\cap C^{\infty}(\overline{U})$, with $\mu_0\leq 1-2\delta$. For $\delta_0>0$ sufficiently small, we define in polar coordinates
\begin{equation} \label{eq:mu fourier defi}  \mu(r,\theta):=\mu_0(r)+\delta_0\sum_{k=1}^{\infty}e^{-\sqrt{k}}\cos(k\theta )\in (0,1-\delta), \quad (r,\theta)\in (1/2,1)\times  [0,2\pi).
\end{equation}
Then $\mu \in \cS_{\delta}(U)\cap C^{\infty}(\overline{U})$, and, for any weight $u$, the transition zone $\Sigma$ for the maximal solution associated to $u$ satisfies
\begin{equation}\label{eq:no outer shell in sigma}|U_{\varepsilon} \backslash \Sigma|>0 \;\;\;\text{ for every } \;\varepsilon>0.\end{equation}     
\end{proposition}

\begin{proof}
 Note first that $\mu$ is smooth by the Weierstrass M-test. Since $\mu \leq_{\operatorname{SH}} \chi_{\Sigma}$, Proposition \ref{prop:subharmonic equivalent} yields, for any harmonic function $h\in C(\overline{U})$,
\begin{equation*}
    \int_{U} h d\mu= \int_{\Sigma}h.
\end{equation*}
For $k \in \mathbb{N}$, let $h_k$ be the harmonic function defined in polar coordinates by
\begin{equation*}
    h_k(r,\theta):=\operatorname{Re}((re^{i\theta })^k)=r^k\cos(k\theta), \quad (r,\theta)\in (1/2,1)\times  [0,2\pi).
\end{equation*}
Then, by the $L^2$ orthogonality of the sequence $\{\cos(k\theta)\}_{k=1}^{\infty}$,
\begin{equation} \label{2einthkmu}
\int_{U}h_kd\mu = \delta_0 e^{-\sqrt{k}} \int_{\frac12}^1\int_{0}^{2\pi}r^{k+1} \cos^2(k\theta)d\theta dr=\frac{\delta_0 \pi}{k+2}e^{-\sqrt{k}}(1-2^{-(k+2)}).    
\end{equation}
Assume that, for some $\varepsilon\in (0,\frac14)$, $|U_{\varepsilon} \backslash \Sigma|=0$. Then $|(B_1\backslash B_{1-\varepsilon}) \backslash \Sigma|=0$, and
\begin{equation*}
\int_{\Sigma \cap (B_{1}\backslash B_{1-\varepsilon})}h_k=\int_{1-\varepsilon}^{1}\int_{0}^{2\pi}r^{k+1}\cos(k\theta)d\theta dr =0, 
\end{equation*}
which implies that
\begin{equation} \label{2einthksig}
\left|\int_{\Sigma} h_k\right| =    \left|\int_{\Sigma \cap B_{1-\varepsilon}} h_k\right|\leq \int_{0}^{1-\varepsilon}\int_{0}^{2\pi}r^{k+1}d\theta dr \leq \frac{2\pi}{k+2}(1-\varepsilon)^{k+2}. 
\end{equation}
We then conclude from \eqref{2einthkmu} and \eqref{2einthksig} that
\begin{equation*}
\delta_0 (1-2^{-(k+2)})e^{-\sqrt{k}}\leq 2(1-\varepsilon)^{k+2}=2e^{\ln(1-\varepsilon)(k+2)}.   
\end{equation*}
Since the right hand side decays exponentially, and the left hand side decays sub-exponentially, this yields a contradiction for sufficiently large $k$.
\end{proof}
Proposition \ref{prop:fourier example}, specialized to the case where the weight $u$ is radial, yields, in particular, the waiting time example stated in Proposition \ref{prop:waiting time}. 
\begin{remark}\label{rem:fourier} The proof of Proposition \ref{prop:fourier example} displays clearly how waiting time irregularity arises as a direct consequence of quantitative non-analyticity of $\mu$ in the tangential direction, that is, the insufficiently fast decay of the Fourier coefficients.

A simple variant of the same argument also shows that, for any fixed $\varepsilon>0$, even finite modes of the form $\mu(r,\theta)=\mu_0(r)+\delta_0 r^k\cos(k\theta)$ will satisfy $|U_{\varepsilon}\setminus \Sigma|>0$ if $k$ is sufficiently large (see Figure \ref{fig:sigma}).
\end{remark}
\begin{figure}[h!]
    \centering
    \includegraphics[width=0.40\textwidth]{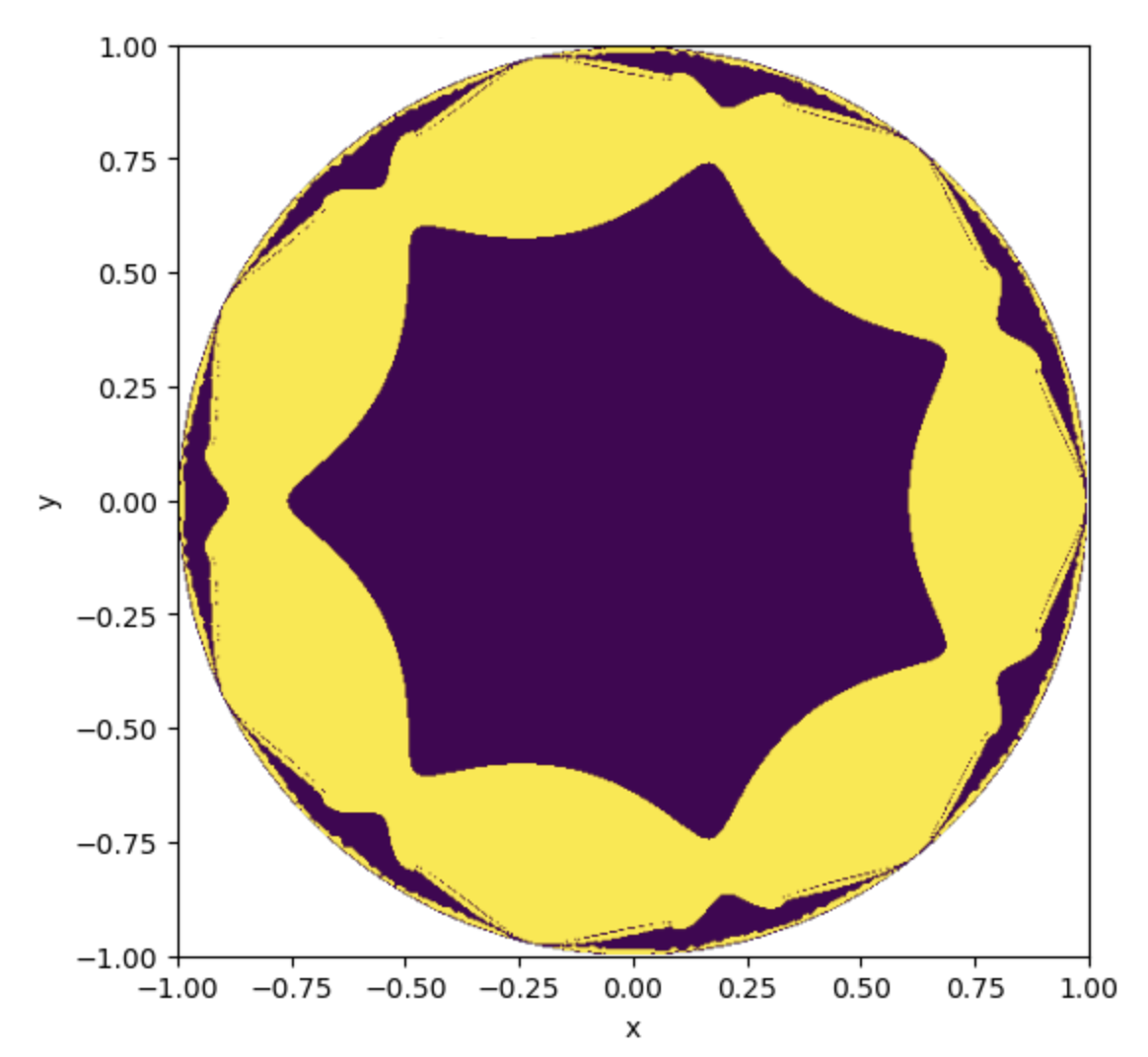}
    \caption{Illustration of the transition zone $\Sigma$ (in yellow) when $d=2$, $U=B_1$, $u(x,y)=2-x^2-y^2$, and $\mu(r,\theta)=\frac12+\frac{9}{20}r^7\cos (7\theta)$. The thin yellow layer at the boundary is not an artifact; see Lemma \ref{lem:sigma saturation}.}
    \label{fig:sigma}
\end{figure}

\section{Stability of maximal solutions}

\label{sec:stability}

In this section we prove Theorem~\ref{thm:stability intro}. 
The main difficulty is that the primal problem $\mathcal{P}(\mu, u)$  depends not only on $\mu$ but also on its positivity set $U$. As we will see, the set-convergence condition imposed in Definition \ref{def:set convergence} is sharp for stability (see Remark~\ref{rem:stability counterexample}). We first prove an upper semicontinuity result for $\mathcal{P}(\mu, u)$. Combined with the regularity of the dual optimizers from Section \ref{sec:regularity}, this yields stability of the transition zone. We then obtain strong $L^p$ stability of the solution $\eta$ as a consequence of the transition zone stability and the free boundary regularity of Corollary \ref{cor:regu}. Throughout this section, we assume that $\delta\in(0,1)$, and that $U$ is an open and bounded Lipschitz set.

\subsection{Stability of the transition zone}

We begin by stating an a priori $L^1(U)$ estimate for the dual optimizers, which follows directly from the proof of \cite[Theorem 5.3]{choi2024existence}:

\begin{lemma}\label{L1_control_optimizer} Let $\mu \in \mathcal{S}_{\delta}(U)$ and $u \in L^1(U)$. Then any optimizer $\psi^*$ of $\mathcal{D}(\mu,u)$ satisfies
\[  \int_{U} |\psi^*(x)| dx \leq  \delta^{-1}\|u\|_{L^1(U)}. \] 
\end{lemma}
Next, we show a boundary estimate for continuous $\mu$ that follows from the subharmonic order,
\begin{lemma} \label{lem:mu boundary estimate} Suppose that $\mu \in \cS_{\delta}(U)\cap C(\overline{U})$ for some $\delta>0$. Then 
\begin{equation}
    \sup_{x\in \partial U}\mu\leq 1-\delta.
\end{equation}    
\end{lemma}
\begin{proof}
Let $\nu$ be as in Definition \ref{def:Sdelta}, and let $v:=\Delta^{-1}(\nu-\mu)$. Since $\mu \leq_{\operatorname{SH}}\nu$ with respect to $U$, Proposition \ref{prop:subharmonic equivalent} implies that $v\geq 0$ and $v \equiv 0$ outside of $U$. Assume, for contradiction, that $\mu(x_0)>1-\delta$ for some $x_0 \in \partial U$. By continuity of $\mu$, there exists $r>0$ such that  $\Delta v=\nu-\mu\leq 1-\delta -\mu <0$ a.e. on $B_r(x_0) \cap U$. Since $v$ vanishes outside $U$, this means that $\Delta v\leq 0$ a.e. in $B_r(x_0)$. But $v(x_0)=0=\min_{\R^d} v$, which contradicts the strong minimum principle.
\end{proof} 

Next, we prove that the primal problem is upper semicontinuous with respect to $(\mu,u)$.
\begin{proposition}[Upper semicontinuity] \label{general_USC}    
For $\mu, \mu_n$, $u_n$, and $u$ as given in Theorem \ref{thm:stability intro},
\begin{equation}\label{usc}
\mathcal{P}(\mu,u) \geq \limsup _{n\to \infty}\mathcal{P}(\mu_n,u_n). 
\end{equation}
\end{proposition}
\begin{proof} We will show that $\mathcal{D}(\mu,u) \geq \limsup_{n \rightarrow \infty} \mathcal{D}(\mu_n,u_n)$, which, by duality, is equivalent to \eqref{usc}. Let $\psi_n$ be a dual optimizer for $\mathcal{D}(\mu_n,u_n)$. By Lemma \ref{L1_control_optimizer}, \eqref{conv:mode}, and \eqref{conv:mode2} we have
\begin{equation} \|\psi_n\|_{L^1(U_n)} \leq 2 \delta^{-1} |U| \cdot \sup_{n}\|u_n\|_{L^{\infty}}\label{L1_bound_uniform_psi} \end{equation}for $n$ sufficiently large. Because $\psi_n \geq 0$ and is subharmonic, we have the Caccioppoli estimate
\[ \|\nabla \psi_n\|_{L^2(V)} \leq C(\operatorname{dist}(V,\p U_n))  \|\psi_n\|_{L^1(U_n)} \hbox{ for any } V\subset \joinrel \subset U_n. \] On the other hand, due to \eqref{conv:mode}, for any open set $V \subset \joinrel \subset U$, we have $V \subset \joinrel \subset U_n$ for sufficiently large $n$.  The Caccioppoli estimate thus yields, up to a subsequence, $\psi_n \rightarrow \psi $ in $L^2_{\textnormal{loc}}(U)$, where $\psi \geq 0$ and is subharmonic. Note also that $\psi \in L^1(U)$ due to \eqref{L1_bound_uniform_psi}.

By Lemma \ref{lem:mu boundary estimate} and the continuity of $\mu$, there exists a compact set $K\subset U$ such that $\mu<1-\delta/2$ on $U \backslash K$. Due to the mode of convergence \eqref{conv:mode}, we have $K \subset U_n $ and $\mu_n<1$ on $U_n \backslash K$ for large $n$. Hence
\begin{multline}
\mathcal{D}(\mu_n,u_n) = \int_{U_n} -(\psi_n-u_n)^+ + \psi_n\mu_n \leq -\int_{K}(\psi_n-u_n)^+ + \psi_n\mu_n \\
+ \int_{U_n\backslash K}u_n+(\mu_n-1)\psi_n 
\leq \int_{K} -(\psi_n-u_n)^+ + \psi_n\mu_n + \int_{U_n \setminus K} u_n.
\end{multline}
Letting $n \to \infty$ and using \eqref{conv:mode}--\eqref{conv:mode2} we obtain
\begin{equation}
 \limsup_{n\to \infty} \mathcal{D}(\mu_n,u_n) \leq
 \int_{K} -(\psi-u)^+ + \psi \mu + |U \setminus K| \cdot \sup_{n}\|u_n\|_{L^{\infty}}.
\end{equation}
By inner regularity of measurable sets and dominated convergence we conclude
\begin{equation}
 \limsup_{n\to \infty} \mathcal{D}(\mu_n,u_n) \leq \int_{U}  -(\psi-u)^+ + \psi\mu  \leq \mathcal{D}(\mu,u) .
\end{equation}
\end{proof}
\begin{remark} \label{rem: optimizer lim}
    The proof of Proposition \ref{general_USC} shows that any subsequential limit $\psi$ of the optimizers $\psi_n$ for $\mathcal{D}(\mu_n,u_n)$ is an optimizer for $\mathcal{D}(\mu,u)$.
\end{remark}

Next we obtain the main stability result for the transition zone.
\begin{proposition}[Stability of transition zones] \label{general_stability} 
Under the notation and assumptions of Theorem \ref{thm:stability intro}, we have $\Sigma_n \to \Sigma$ in the sense of Definition \ref{def:set convergence}.
\end{proposition}

\begin{proof}  First observe that, letting $\nu_n=\chi_{\Sigma_n}$,
\[ \|\nu_n\|^{2}_{L^2(\R^d)} =\|\nu_n\|_{L^1(\R^d)} = \|\mu_n\|_{L^1(\R^d)} \leq (1-\delta) |U|. \] In the equalities above, we used the assumption that $\mu_n\in \cS_{\delta}(U_n)$ and the fact that $\nu_n \in \{0,1\}$ from Proposition \ref{prop:eulerian} (c). Hence, up to extracting a subsequence,  $\nu_n \rightharpoonup \tilde{\nu}$ in $L^2(U)$ for some $\tilde{\nu}$. From Proposition \ref{general_USC}, 
\begin{equation} \int_U u d\nu \geq \limsup_{n \rightarrow \infty} \int_{U_n} u_nd \nu_n = \int_U u d\tilde{\nu} \label{USC_implies123}  \end{equation} By Proposition \ref{prop:eulerian} (b), there exists a unique optimizer of $\mathcal{P}(\mu,u)$, so if $\tilde{\nu}$ were admissible, we would obtain $\tilde{\nu}=\nu$. To show admissibility, first observe that from the weak convergence we have $0 \leq \tilde{\nu} \leq \chi_{U}$, so it remains to check that $\mu \leq_{\textnormal{SH}} \tilde{\nu}$ over $U$. By the uniform $L^\infty$ bounds and compact supports of $\mu_n,\nu_n$, the weak convergence and the Calder\'on-Zygmund estimate imply that $\Delta^{-1}(\nu_n-\mu_n)$ converges uniformly to a limit $g(x)$, which solves
 \[ \Delta g = \tilde\nu - \mu \text{ on } \R^d, \quad \quad \lim_{|x| \rightarrow \infty} g(x) = 0.  \] Hence, by Liouville's theorem, this limit is $\Delta^{-1}(\tilde{\nu}-\mu)$. Therefore, we have the uniform limit
\[ \Delta^{-1} (\tilde{\nu} -  \mu)(x) = \lim_{n \rightarrow \infty} \Delta^{-1} (\nu_n -\mu_n)(x).  \] 
By Proposition \ref{prop:eulerian} (a) and the definition of weak solution,
$\Delta^{-1} (\nu_n - \mu_n) \geq 0$ and $\Delta^{-1} (\nu_n -  \mu_n) \equiv 0$ in $U_n^C$. Thus we have $ \Delta^{-1} (\tilde{\nu}-\mu) \geq 0$.
Moreover,  for any $x_0\in \R^d \backslash \overline{U}$ and any $r>0$ such that $B_r(x_0) \subset \R^d \backslash \overline{U}$,
\[ |\{x\in B_{r}(x_0):\Delta^{-1} \nu_n(x) - \Delta^{-1} \mu_n(x) \neq 0\}| \xrightarrow{n \to \infty} 0,  \] since $|U_n \Delta U|\xrightarrow{n\to \infty} 0$. By uniform convergence and continuity, it follows that 
\[ \Delta^{-1} (\tilde{\nu}- \mu) \equiv 0 \text{ in } U^C. \]
Since $v := \Delta^{-1}(\tilde{\nu}-\mu)$ is nonnegative, and vanishes outside $U$, we conclude from Lemma \ref{lem:integration by parts} (i) and Proposition \ref{prop:subharmonic equivalent} that $\mu \leq_{\textnormal{SH}} \tilde{\nu}$ with respect to $U$, which shows that $\tilde{\nu}$ is admissible for $\mathcal{P}(\mu,u)$. Hence, $\tilde{\nu} = \nu$. 

Now we show the symmetric difference convergence of $\Sigma_n$ to $\Sigma$, or equivalently the strong convergence of $\nu_n$ in $L^2(\R^d)$. To upgrade from weak to strong convergence, by the Hilbert space structure of $L^2$, it is enough to show that the $L^2$ norms converge. The latter holds because, from the subharmonic ordering and the fact that $\nu(x),\nu_n(x) \in \{0,1\}$, we have 
\[ \|\nu\|_{L^2(\R^d)} = \|\nu\|_{L^1(\R^d)}^{1/2} = \|\mu\|_{L^1(\R^d)}^{1/2} = \lim_{n \rightarrow \infty} \|\mu_n\|_{L^1(\R^d)}^{1/2} = \lim_{n \rightarrow \infty} \|\nu_n\|_{L^1(\R^d)}^{1/2} =  \lim_{n \rightarrow \infty} \|\nu_n\|_{L^2(\R^d)}. \] Hence, $\nu_n \rightarrow \nu$ in $L^2(\R^d)$.

Lastly, we show that if $\overline{\Sigma} \subset W$ for some open set $W$, then $\overline{\Sigma_n}\subset W$ for sufficiently large $n$. Let $K := \overline{U} \backslash W$, and let $\psi_n,\psi\in L^1(U)$ be, respectively, optimizers for \eqref{eq:dualp simple}. By the last statement in Lemma \ref{lem:sigma saturation}, we have $\partial U\subset \overline{\Sigma}$, and thus $K\subset U$. Observe also that $K \cap \overline{\Sigma} =\emptyset$ and therefore, by part (ii) of Lemma \ref{lem:sigma saturation},  $K \subset \{s=\infty\}^{\circ}$. In particular, noting that $K$ is compact, $\psi$ is harmonic in a neighborhood of $K$. 
By Proposition \ref{prop:sigma structure}, since $\psi-u$ is smooth and strictly subharmonic in $\{s>0\}$, we must have $\{s=\infty\}^{\circ} \cap \{\psi \geq u\}=\emptyset$, and thus
\begin{equation}
 K \subset \{\psi< u\}.
\end{equation}
Since $K \subset \{s>0\}=\{\Delta^{-1}(\nu-\mu)>0\}\subset U$,
 we have $K\subset \{\Delta^{-1}(\nu_n-\mu_n)>0\}\subset U_n$, and $\psi_n$ is harmonic in a fixed neighborhood of $K$ for sufficiently large $n$. Recalling Remark \ref{rem: optimizer lim}, and using the compactness of harmonic functions, up to a subsequence  we may assume that $\psi_n \to \psi$ uniformly in $K$  and thus, for sufficiently large $n$, 
\begin{equation}
    K \subset \{\psi_n <u_n\}. 
\end{equation}
 A second application of Proposition \ref{prop:sigma structure} then yields $K \subset \{s_n=\infty\}^{\circ}$, and thus $\overline{\Sigma_n} \subset W$, which concludes the proof.
\end{proof}

 \begin{remark} \label{rem:stability counterexample} 
By the notion of convergence in Definition \ref{def:set convergence}, we assumed in Proposition \ref{general_stability}  that for every compact set $K \subset U$, one has $K \subset U_n$ for sufficiently large $n$. The following counterexample states that this assumption is necessary for stability. Let $U:=B_1$, $u(x):=2-|x|^2$, and $\mu_n \equiv \frac12 \chi_{U_n}$ where $U_n := B_1 \backslash \overline{B_{2^{-n}}}$. Then all the other conditions of Proposition \ref{general_stability} are satisfied. One has $\Sigma= B_1\backslash \overline{B_{2^{-1/d}}}$ (see Proposition \ref{prop:radial transition zone}). Yet, from the results of Section \ref{subsec:radial},  $\Sigma_n \xrightarrow{n\to \infty} (B_1 \backslash \overline{B_{b}}) \cup B_a $, where $0<a<b<1$. 
 \end{remark}

\subsection{Stability of the temperature}
To complete the proof of Theorem~\ref{thm:stability intro}, it remains to show stability of the $\eta$ variable. The weak convergence of the $\eta$ variable is a direct consequence of the transition zone's stability. To upgrade the weak convergence to strong convergence, we use the results of Section \ref{sec:regularity}, namely the fact that the transition zone has measure zero boundary. It then follows that the free boundary $\partial\{\eta>0\}$ is a set of space-time measure zero, which enables upgrading the weak convergence of $\eta$.
 We begin by noting some basic stability properties of the potential variable $w$, which follow from energy estimates and the Calder\'on-Zygmund estimate.

\begin{lemma}[$L^2$ stability of potentials]\label{L2_stable_potential}
Let $(\mu_n,\nu_n) \subset (L^{\infty}(\R^d))^2$ be a sequence of uniformly compactly supported and bounded densities, with $\mu_n \leq_{\operatorname{SH}}\nu_n$.
Let $w_n$ solve \eqref{obstacle_problem_123} with $\mu=\mu_n$ and $\nu=\nu_n.$ Then 
\begin{equation}
\|w_n(t)-w(t)\|_{L^2(\R^d)}^2 
\lesssim 
\Big( 
\|\mu_n-\mu\|^2_{L^2(\R^d)} 
 + \|\nu_n-\nu\|^2_{L^2(\R^d)} 
\Big), \label{parabolic_stability_pot} \end{equation}
where $w$ is the unique solution of \eqref{obstacle_problem_123}. In addition, if $\mu_n \rightarrow \mu$ and $\nu_n \rightarrow \nu$ in $L^2(\R^d)$, then $w_n\to w$ in $C^{1,\alpha}_x C_t^{\beta}([0,\infty)\times\R^d)$  for any $0 < \alpha,\beta < 1$.
\end{lemma}
This strong convergence of $w_n$ implies weak convergence of $\eta_n:$

\begin{lemma} \label{thm:stable_eta} For $\mu_n, \nu_n, \mu, \nu$ given as in Lemma \ref{L2_stable_potential}, let $(\eta_n,\rho_n)$ and $(\eta,\rho)$ be respectively the Eulerian variables associated to $(\mu_n,\nu_n)$ and $(\mu, \nu)$.  Then
\[ \eta_n \rightharpoonup \eta \;\;\hbox{ weakly in} \;L^2( (0,\infty) \times \R^d).  \] 
\end{lemma}

\begin{proof}
    By Lemma \ref{L2_stable_potential},  for any $\psi \in L^2(\R^d)$ and any fixed time $t \geq 0$, 
    $$ \int_{\R^d} \psi(x) w_n(t,x) \,dx \rightarrow \int_{\R^d} \psi(x) w(t,x) \,dx. $$
    Recalling that $w_n(t,x) = \frac12 \int_t^{\infty} \eta_n(s,x) \,ds$ (Proposition \ref{prop:eulerian}), this implies  that
    $$
        \int_{\R^d} \int_0^{\infty} \psi(x) \chi_{[t,\infty)}(s) \eta_n(s,x) \,ds\,dx \rightarrow \int_{\R^d} \int_0^{\infty} \psi(x) \chi_{[t,\infty)}(s) \eta(s,x) \,ds\,dx. 
    $$
    This convergence holds for all test functions of the form $g(s,x) = \psi(x) \chi_{[t,\infty)}(s)$. By a density argument, it follows  that $\eta_n \rightharpoonup \eta$ in $L^2((0,\infty) \times \R^d)$.
\end{proof}

Now we upgrade the above to strong convergence:

\begin{proposition}[Stability of the temperature] \label{stable_Stefan_strong}
 For $\mu\in \mathcal{S}_{\delta}(U)\cap C(\bar{U})$ and $\mu_n\in \mathcal{S}_{\delta}(U_n)\cap C(\bar{U_n})$ satisfying \eqref{conv:mode}, and $u_n,u$ satisfying \eqref{conv:mode2}, let $\eta_n$ and $\eta$ be respectively the maximal solutions of \eqref{eq:stefan intro} with initial data $\mu_n$ and $\mu$, and with weight $u$. Then one has, for any $p\in [1,\infty)$,
\[
\eta_n \to \eta \quad \text{in } L^p((0,\infty)\times \mathbb{R}^d).
\]
\end{proposition}
\begin{proof}Since the $\eta_n$ are bounded, by interpolation we may take $p=1$. We recall from \eqref{eq:w limit t infty} that
\begin{equation}\int_{T}^{\infty} \int_{\R^d}|\eta-\eta_n| \leq \int_{U} w(T,\cdot)+\int_{U_n}w_n(T,\cdot) \xrightarrow[]{T\to \infty}0 \;\;\text{ uniformly in $n$}.\end{equation}
Note that $\eta$ and $\eta_n$ are uniformly bounded, and their supports are uniformly bounded in measure. Thus, to conclude it suffices to show $L^1$ convergence in any fixed finite time interval $(t_0,T)$, $0<t_0<T$. By Corollary \ref{cor:regu} (recalling that $|\{x \in U:s(x)=0\}|=0$ by Definition \ref{def:weak solution}) we have $|\partial\{(t,x):\eta(t,x)>0\}|=0$. Therefore,  by inner regularity of measurable sets, for any $\e>0$ there exists a compact set $K \subset \{(t,x): t\in[t_0,T],\; \eta(t,x)>0\}$ such that $|\{(t,x): t\in [t_0,T],\; \eta(t,x)>0\} \backslash K|<\e$. By regularity of $w$, we may also assume that $w\leq \e$ outside $K$. By Proposition \ref{general_stability} and Lemma \ref{L2_stable_potential}, we have $K \subset \{\eta_n>0\}$ and $w_n\leq 2\varepsilon$ outside $K$ for $n$ sufficiently large.  Since $\eta_n$ are bounded and caloric in $K$, and since $\eta_n \rightharpoonup \eta$ weakly (Lemma \ref{thm:stable_eta}), parabolic regularity implies that
 $\lim_{n\to \infty}\|\eta_n-\eta\|_{L^{\infty}(K)}=0$. Since
\begin{multline}
 \|\eta-\eta_n\|_{L^1((t_0,T)\times \mathbb{R}^d)} \leq \int_{K}|\eta-\eta_n| + \int_{((t_0,T)\times \R^d)\setminus K}|\eta-\eta_n|\\
 \leq \int_{K}|\eta-\eta_n|+ \|w+w_n\|_{L^{\infty}(((t_0,T)\times \R^d)\setminus K)}(|U|+|U_n|),
\end{multline}
letting $n\to \infty$ yields $\limsup_{n\to \infty}\|\eta-\eta_n\|_{L^1((t_0,T)\times \mathbb{R}^d)}  \leq 6 \e|U|$. To conclude, let $\e \to 0$.
\end{proof}

\section{Regularity and nucleation of radial and one-dimensional maximal solutions}
\label{maximal_solution_d_1_section}
In this section, we analyze the topology of the positivity set of maximal solutions in the radial and  one dimensional settings. Although one-dimensional solutions have been extensively investigated, the existing 
theory primarily concerns unbounded intervals, where the free boundary is constrained, by definition, to consist of a single point, avoiding any topological changes or nucleation phenomena. This theory does not apply to our setting of compactly supported initial data. In particular, 
previously proposed selection principles, such as the notion of physical solutions 
\cite{DNS22}, do not cover this regime.

Building on the universality result established in Proposition~\ref{target_measure_universial} for radial and one-dimensional solutions, 
we show that maximal solutions serve as a natural selection principle for compactly supported data. Indeed, the universal maximal solution minimizes nucleation: its support may only exhibit topological changes if every other weak solution does as well (Proposition~\ref{nucleation_most_unlikely}). We recall that 
nucleation has been interpreted as a continuation mechanism for the evolution of solutions 
\cite{gurtin1994thermodynamics}. 

Our analysis further shows that (1) nucleation does not occur when $\mu$ is below the critical supercooling threshold ($\sup \mu \leq 1$),
and (2) nucleation may occur in general otherwise, as demonstrated by an 
instructive example. 

Because of the universality established in Proposition~\ref{target_measure_universial}, given a radial initial data $\mu$, we simply refer to the maximal solution associated with any (radial if $d\geq 2$) weight $u$ as the maximal solution. 
 
\subsection{ The subcritical case: $0\leq \mu\leq 1$}
We show next that if $\mu$ does not surpass the critical threshold $1$, the evolution of the free boundary is smooth over time.  A crucial ingredient is the following monotonicity property, which will ensure that the free boundary does not nucleate or jump over time. Under the slack condition $\delta>0$ (which we drop in this section), we showed already in Corollary \ref{cor:regu} that even non-radial maximal solutions do not jump below the critical supercooling threshold, but it remains open whether such non-radial solutions may nucleate in dimensions $d\geq 2$.

\begin{proposition}[No nucleation or jumps]\label{prop: no_nuc_jumps}

Assume that 
\begin{equation*}U=B_1 \quad \text{or} \quad U=A_{\rho}:=\{\rho<|x|<1\}, \; \;\;\rho\in (0,1) \quad \text{or} \quad U=(0,1) \text{ with } d=1,    
\end{equation*} Let $0 \leq \mu \leq 1$ be supported on $U$, and assume $\mu$ is radial if $d\geq 2$. Assume further that $\mu$ is not a.e. equal to one in any neighborhood of $\p U$. Then the (radial if $d\geq 2$) maximal solution $\eta$ with initial data $\mu$ exists and the freezing time $s(x)$ satisfies
\begin{equation}
\begin{cases}
\textnormal{$s(x)$ is strictly increasing on $(0,a)$ and strictly decreasing on $(b,1)$} 
& \textnormal{if } U = (0,1), \\[6pt] 

\textnormal{$s(|x|)$ is strictly decreasing on $(\tilde r, 1)$} 
& \textnormal{if } U = B_1, \\[6pt]

\textnormal{$s(|x|)$ is strictly increasing on $(\rho, r_1)$ and strictly decreasing on $(r_2,1)$} & \textnormal{if } U = A_\rho.
\end{cases} \label{s_strictly_monotone}
\end{equation}
 where $(a,b)$ are from \eqref{one_dimension_mass_moments}, 
$\tilde r$ is from \eqref{same_mass_ball}, 
and $(r_1,r_2)$ are from \eqref{eq:annuli_constraints}.
In particular, there is no initial freezing, nucleation, or jump of the free boundary, and the transition zone $\Sigma$ is given by $\{\nu^*>0\}$ with $\nu^*$ from Proposition \ref{target_measure_universial}.
\end{proposition}

\begin{proof} For concreteness, we assume $U=A_{\rho}$, with the other two cases being similar. Consider the radial function $w(t,x) =\frac12 \int_t^{\infty} \eta(s,x) ds$, where $\eta$ is the active particle distribution associated to $(\mu,\nu^*)$. Then from Lemma \ref{LSC_eta} and \eqref{eq:s defi intro}, we have \[s(x) = \inf\{t>0:w(t,x)=0\}.\] 
Thus, as $\p_t w \leq 0$, it suffices to show that $\p_r w(t,r) >0$ on $(\rho,r_1)$ and $\p_r w(t,r) < 0$ on $(r_2,1)$ to obtain \eqref{s_strictly_monotone}.

We first show that $w(t,\cdot)$ is subharmonic on $\Sigma$ for every $t\geq 0$. By Proposition \ref{prop:eulerian}, we have $\Delta w(t,x) = \chi_{\Sigma}-\mu_t(x)$, where $\mu_t(x)$ is the law of $W_{t \wedge \tau^*}$ so it suffices to show $\mu_t \leq 1$. For this purpose, note that, since $\mu\leq 1$ and $\eta$ is subcaloric, $\eta\leq 1$ by the comparison principle. Thus, using \eqref{def_integral_rho} and \eqref{eq:intrho0t1sec42pf} we get from \eqref{mu_t_formula} that $$\mu_t(x)=\eta(t,x)+\chi_{\Sigma\cap\{s(x)\leq t\}} \leq\eta(t,x)+\chi_{\{\eta(t,x)=0\}}\leq1.$$ 
Hence, $w(t,\cdot)$ is subharmonic on $\Sigma$ for any $t \geq 0$. 

Since $w\geq 0$ and $w(t,\cdot)=0$ on $\p B_{\rho} \cup \p B_1$ due to Proposition \ref{prop:subharmonic equivalent}, we have $0=\partial_r w(t,\rho)=\partial_r w(t,1)$, and, $w(t,\cdot)$ being subharmonic then implies
\[ \frac{d}{dr} (r^{d-1} \p_r w(t,r)) \geq 0 \;\;\text{ on }\;\; (\rho,r_1) \cup (r_2,1). \] 
 This then implies that $\p_rw(t,r) \geq 0$ on $(\rho,r_1)$ and $\p_r w(t,r) \leq 0$ on $(r_2,1)$.  Since $\p_t w \leq 0$, we infer that $s(|x|)$ is nondecreasing on $(\rho,r_1)$ and nonincreasing on $(r_2,1)$.

To upgrade this to strict monotonicity, we argue by contradiction. If $s(|x|)$ was not strictly increasing on $(\rho,r_1)$, then there would be an interval $I = (c,d) \subset (\rho,r_1)$ such that $s(|x|)\equiv s(|c|)$ in $I$, which implies $w(s(c),r)\equiv 0$ for $r \in I$, and thus  $\p_r w(s(c),r)\equiv 0$ for $r \in I$. Differentiating the equation $w_t-\frac12 \Delta w=-\frac{1}{2}$, we see that
$p := \p_r w$ solves
\begin{equation} \p_t p = \p^2_{rr} p  + \frac{d-1}{r} \p_r p - \frac{d-1}{r^2} p \text{ on } (0,s(c)) \times I \label{radial_PDE} \end{equation}

Since $p \geq 0$ and attains its minimum value $0$ outside of the parabolic boundary, \eqref{radial_PDE} implies that $p \equiv 0$ on $[0,s(c)] \times I$. By spatial analyticity of caloric functions, $p \equiv 0$ on $[0,s(c)]\times [\rho,d]$.  
We infer from \eqref{obstacle_problem_123} that
$0=\Delta w(0,x) = 1-\mu(x)$ a.e. on $ \{ \rho < |x| < r_1\}$, contradicting our assumption on $\mu$. A similar argument shows that $s(|x|)$ is strictly decreasing on $
(r_2,1)$.
\end{proof}

We now show that the free boundary is smooth. 
\begin{proposition}[Smooth free boundary] \label{smooth_fb} Under the notation and assumptions of Proposition \ref{prop: no_nuc_jumps}, every free boundary point is a regular point. In particular,
\begin{equation}
  s\in C^{\infty}(\Sigma), \quad \nabla s \neq 0.  
\end{equation} 
\end{proposition}
\begin{proof}  By Proposition \ref{prop: no_nuc_jumps}, the transition zone $\Sigma$ is open, and thus, for any $x_0\in \Sigma$, $w$ satisfies, in a neighborhood of $(t_0,x_0):=(s(x_0),x_0)$, \begin{equation} w_t- \Delta w/2=-\chi_{\{w>0\}}/2.
\end{equation}
Letting
\begin{equation}
    w_r(t,x)=r^{-2}w(t_0+r^2t,x_0+rx), 
\end{equation}
by the classical theory for the parabolic obstacle problem \cite[Props. 1.1, 1.2]{Bla06}, there exists $p_2$ such that
\begin{equation} \label{eq:d=1 blow up convergence}
    \lim_{r\to 0^+}w_r(t,x)=p_2(t,x), \quad (t,x)\in (-\infty,0]\times \R^d
\end{equation}
locally uniformly, where
\begin{equation} \label{eq:alternative blow-up reg}
    p_2=(x\cdot e)_+^2/2, \quad e\in \mathbb{S}^{d-1}
\end{equation}
or \begin{equation} \label{eq:alternative blow-up sing}
    p_2=(-mt+Ax \cdot x)/2, \quad A\geq 0, \quad \operatorname{tr}(A)=1-m, \quad m\in [0,1].
\end{equation}
By Proposition \ref{prop: no_nuc_jumps}, $\{w(t_0,\cdot)>0\}$ is an annulus or a ball, so there exists $e_0\in \mathbb{S}^{d-1}$ such that \[\{w(t_0,\cdot)>0\}\subset \{e_0\cdot (x-x_0)>0\}.\] But this makes \eqref{eq:alternative blow-up sing} impossible unless $m=1$, since \eqref{eq:d=1 blow up convergence} at $t=0$ would then force \[\lim_{r\to 0}|B_r(x_0)|^{-1}|B_r(x_0) \setminus \{w(t_0,\cdot)>0\}| =0.\]
On the other hand, $m=1$ can only occur if
\begin{equation*}\limsup_{(t,x)\to (t_0,x_0),\; t\leq t_0} \eta(t,x)\geq 1,\end{equation*} 
which, since $\mu \leq1$, is ruled out by the strong maximum principle. Therefore, one must have \eqref{eq:alternative blow-up reg}. In the language of the obstacle problem, this means that $(t_0,x_0)$ is a regular free boundary point, and it is a classical fact that, in a neighborhood of $(t_0,x_0)$, the free boundary is a smooth hypersurface moving with finite positive speed (see \cite[Thm. II]{caffarelli2004regularity}), which means $s\in C^{\infty}$ near $x_0$, and $\nabla s(x_0) \neq 0$.\end{proof}

\subsection{Example of nucleation when $|\{\mu>1\}|>0$}  \label{nuc_heuristics} We now focus on nucleation in the case $d=1$, while noting that the heuristics below are also valid for the radial case. Proposition \ref{prop: no_nuc_jumps} implies that when $0 \leq \mu \leq 1$  there is no nucleation of the free boundary, namely the set $\{w(t,\cdot)>0\}$ is a single bounded open interval for all $t>0$. We will show that it is possible for solutions to nucleate at positive times when $|\{\mu>1\}|>0$. 
The key point is that if $|\{\mu>1\}|>0$, then the potential $w$ may no longer be monotone over the transition zone.

We first illustrate the heuristics behind our example. Fixing $k\in \mathbb{N}$, we consider a density $\mu$ whose maximum is slightly larger than $1$ and which is formally expected to produce $k$ nucleation. Let $ \mu = \chi_{ (0,1/2) } + \chi_{ (3/4,1) } - \varepsilon  \psi''$, where $\psi: \R\to [0,\infty)$ is smooth with $\{\psi>0\}=(0,1)$, is concave on $[1/2,3/4]$, and $\varepsilon>0$ is chosen small enough so that $\mu \geq 0$. 

 Then $\mu \leq_{\textnormal{SH}} \nu:=\chi_{(0,1/2)} + \chi_{(3/4,1)}$  by an integration by parts argument. In particular, by Proposition \ref{target_measure_universial}, $\nu$ is the optimal target measure of $\mathcal{P}(\mu,u)$. Hence $w(0,x) = \Delta^{-1}(\nu-\mu)(x) = \varepsilon \psi(x)$ and $\lim_{t \to 0} \partial_t w(t,x) = (1/2)(w_{xx}-\nu\chi_{\{w>0\}})(0,x)= -\mu(x)/2$. For small times $t>0$, using that $w \geq 0$ and $\p_t w \leq 0$, a first-order in time approximation for $w(t,x)$ suggests
$$w(t,x) \approx \max\{\varepsilon \psi(x) - (t/2) \mu(x),0\}.$$ Since $\e \psi(x) \ll 1$ and $\mu \approx 1$ on $(0,1/2)$, the linear profile collapses to zero so rapidly that we expect this first-order approximation to be accurate. 

Now if $\psi$ has $k$ intervals on which it decreases in $(0,1/2)$, then we expect that there should at least be $k$ times  where $w$ hits zero in interior points of its positivity set from the past times. Since $\{w>0\} \cap \{t>0\} = \{\eta>0\} \cap \{t>0\}$, we therefore expect at least $k$ times at which $\eta$ nucleates.

Based on the above heuristics, we rigorously show that there exists a polynomial $\mu$ for which the corresponding maximal solution with initial data $\mu \chi_{(0,1)}$ nucleates:
\begin{proposition}[Nucleation can occur] \label{nuc_example} There exists a polynomial initial density $\mu$ on $(0,1)$ such that $|\{\mu>1\} \cap (0,1)|>0$, and for the maximal solution with initial data $\mu \chi_{(0,1)}$ one has that
\[
s(x) \textnormal{ is not monotone on } (b,1),
\]
where $b$ is defined in \eqref{one_dimension_mass_moments}.
\end{proposition}

The proof of Proposition \ref{nuc_example} can be found in Appendix \ref{nuc_example_appendix}. 

Next we show that  maximal solutions nucleate only in cases where nucleation is unavoidable. Namely, if the maximal solution nucleates, then so does every other weak solution. 

\begin{proposition} \label{nucleation_most_unlikely}
Let $\eta$ be a weak solution of \eqref{eq:stefan intro} with initial data $\mu$, where $\mu$ satisfies the assumptions of Proposition \ref{target_measure_universial}. Assume that $\eta$ is \emph{not} the maximal solution associated with $\mu$. Let $s(x)$ denote the freezing time associated to $\eta$, and consider the set where the solution never freezes,
\[
    \Gamma := \{x : s(x) = \infty\}.
\]
If $U = (0,1)$, then $\Gamma$ is not an interval; if $U = B_1$, then $\Gamma$ is not a ball; and if $U = A_{\rho}$, then $\Gamma$ is not an annulus.
\end{proposition}

\begin{proof} We show the claim for the case when $U = (0,1)$ as the other two cases follow from a parallel argument. Since there exists a weak solution of \eqref{eq:stefan intro}, we obtain from Proposition \ref{prop:eulerian} (d) that $\mu \in \mathcal{S}_0((0,1))$, and 
\[ \mu \leq_{\textnormal{SH}} \nu := \chi_{ \{0<s(x)<\infty\} }. \]
By \cite[Thm. 5.6]{kim2024stefan}, $\eta$ is the active particle distribution associated to $(\mu,\nu)$. In particular, by \cite[Cor. 9.9]{kim2024stefan}, the fact that $\eta$ is not the maximal solution implies that $\nu \neq \nu^*$, where $\nu^*$ is given by Proposition \ref{target_measure_universial}.

As $\eta$ is supported on $(0,1)$, Proposition \ref{prop:eulerian} (a) implies that if $w(t,x)$ is the associated potential variable, then $\{x:w(0,x)>0\} \subset (0,1)$ and $\{x:0<s(x)<\infty\} \subset (0,1)$. Hence, from \eqref{optimal_stopping_time}
we obtain $\tau \leq \tau^{(0,1)}$, where $\tau$ is the optimal stopping time for $\mathcal{Q}(\mu,\nu).$ 

 Since $\mu \leq_{\operatorname{SH}}\nu$ and $\mu \leq_{\operatorname{SH}}\nu^*$, we have $\int_{0}^{1}\mu=\int_{0}^{1}\nu=\int_{0}^{1}\nu^*\leq 1$. If we had $\int\mu =1$, then $\nu=\nu^*=\chi_{(0,1)}$, a contradiction. Hence, we may assume that $\int_0^1 \mu(x)dx < 1$, and therefore $a<b$. Since $\nu^*$ is the unique target measure that is the indicator of two intervals containing $0$ and $1$,  $\nu$ cannot be the indicator of two intervals. By Lemma \ref{lem:sigma saturation}, $\nu$ also cannot be the indicator of a single interval. Thus, the set $\{s(x)=\infty\}$ cannot be a single interval. 
\end{proof}

In particular, Proposition~\ref{nucleation_most_unlikely} implies that $\Gamma$ is disconnected for any non-maximal solution, which forces these solutions to nucleate.

\subsection*{Acknowledgments}
The authors thank Young-Heon Kim for his helpful comment on lower semicontinuity of the optimal Skorokhod problem. The authors also thank Sergey Nadtochiy for insightful discussions that inspired Proposition \ref{fractal_freezing_maximal}.
 R.C. is partially supported by NSF grant DMS 2342349. I.K. is partially supported by NSF DMS 2452649. Part of this work is completed during I.K’s visit at KIAS, and she thanks KIAS’s hospitality.

\begin{myappendices} 
  
\section{Admissible target measures in radial and one-dimensional cases}
\label{app:opt-target}
In this appendix, we show that the target measure $\nu^*$ from
Proposition~\ref{target_measure_universial} is admissible for
$\mathcal{P}(\mu,u)$ under the assumption $\mu \le 1$. 
As shown in the proof of Proposition~\ref{target_measure_universial}, 
this admissibility already implies that $\nu^*$ is the optimal target measure  
for arbitrary radial $\mu \in \mathcal{S}_0(U)$, even without requiring $0\leq\mu \leq1$.

\begin{proposition}[Target measure over annuli]\label{prop:annuli_target}
Let $0 \le \mu(x) \le 1$ be radial and supported on the annulus $A_\rho := \{\rho < |x| < 1\}$ with $1 > \rho >0$.
Then choose radii $\rho \le r_1 \le r_2 \le 1$ such that for $\nu(x) := \chi_{ \{\rho <|x| < r_1\} } + \chi_{ \{r_2 <|x| <1\} }$, one has
\begin{equation}\label{eq:annuli_constraints}
    \int_{\rho}^1 r^{d-1} \mu(r)\,dr
    = \int_{\rho}^1 r^{d-1} \nu(r) dr, \quad \quad
   \int_{\rho}^1 \phi(r) \mu(r) dr = \int_{\rho}^1 \phi(r) \nu(r)dr,\end{equation} 
   where $\phi(r) = r$ if $d \neq 2$ and $\phi(r) = r\log(r)$ if $d=2$. Then $\mu \le_{\mathrm{SH}} \nu$ on $A_\rho$. In addition, $v(|x|) := \Delta^{-1}(\nu-\mu)(|x|)$ is nondecreasing on $(\rho,r_1)$ and nonincreasing on $(r_2,1)$.
\end{proposition}

\begin{proof} Using that $0 \leq \mu \leq 1$, we obtain the existence of such $0 \leq r_1 \leq r_2 \leq 1$ via an intermediate value theorem argument. 

Let us first show that \eqref{eq:annuli_constraints} implies $v(r)=0$ for $r \leq \rho$ and $r \geq 1$. Observe that the second equality in \eqref{eq:annuli_constraints} implies $v(0)=0$. Since $v$ is radial and harmonic on $B_{\rho}(0)$, it must vanish identically for $r \leq \rho$. 

Conversely, for $|x| \geq 1$, the map $y \mapsto \mathcal{N}(x-y)$ is harmonic on $B_1$. Using polar coordinates and the mean value property, we observe
\begin{align*}
    v(x) &= \int_{B_1} \mathcal{N}(x-y) \, d(\nu-\mu)(y) = C_d \int_0^1 r^{d-1} (\nu(r)-\mu(r)) \left( \int_{\partial B_1} \mathcal{N}(x-r \omega) \, d\omega \right) dr \\
    &= \tilde{C}_d \mathcal{N}(x) \int_0^1 r^{d-1}(\nu(r)-\mu(r)) \, dr = 0,
\end{align*}
where the last equality follows from the first equality in \eqref{eq:annuli_constraints}. Thus, $v(x)=0$ for $|x| \geq 1$.

Now because $v$ is radial and $\Delta v = \nu-\mu$, we have that
\begin{equation}\label{eq:radial_ODE}
\frac{d}{dr}\!\left(r^{d-1}v'(r)\right)=r^{d-1}(\nu(r)-\mu(r)).
\end{equation}

\medskip\noindent\textbf{Step 1: nondecreasing on $(\rho,r_1]$.}
Because $0\le \mu \le 1$ and $\nu=1$ on $(\rho,r_1]$, the right-hand side of \eqref{eq:radial_ODE} is non-negative there. As $v'(\rho)=0$ because $v(r)=0$ for $r \leq \rho$, we obtain
\[
v'(r)\ge 0 \qquad \text{on } (\rho,r_1].
\]

\medskip\noindent\textbf{Step 2: nonincreasing on $[r_2,1)$.}
Similarly, since $\nu=1$ on $[r_2,1)$, $v'(1)=0$, and \eqref{eq:radial_ODE} implies
\[
v'(r)\le 0 \qquad \text{on } [r_2,1).
\]

\medskip\noindent\textbf{Step 3: behavior on $(r_1,r_2)$.}
Define the radial quantile
\[
r^{*}:=\inf\left\{r>\rho:\int_{\rho}^{r}s^{d-1}\!\left(\nu(s)-\mu(s)\right)\,ds \le 0\right\}.
\]
The previous monotonicity conclusions now imply
\[
r_1 \le r^{*} \le r_2.
\]
Now for $ r \in (\rho,r^{*})$, $\int_{\rho}^r s^{d-1}(\nu(s)-\mu(s)) ds \geq 0$, so integrating \eqref{eq:radial_ODE} implies from $v'(\rho)=0$ that $v'(r)\ge 0$. Now if $r \in [r^*,r_2)$, $\int_{\rho}^r s^{d-1}(\nu(s)-\mu(s)) ds \leq 0$, so integrating \eqref{eq:radial_ODE} implies $v'(r)\le 0$. Thus $v$ attains its minimum at $r=\rho$ or $r=1$, where it is zero. Therefore $v\geq 0$ and vanishes outside $A_{\rho}$, which proves $\mu \le_{\mathrm{SH}} \nu$ on $A_{\rho}$ by Proposition \ref{prop:subharmonic equivalent}.
\end{proof}

Next we handle the case where $\mu$ is supported on the unit ball.

\begin{proposition}[Target measure over unit ball]\label{prop:radial_target}
Let $0 \le \mu(x) \le 1$ be radial and supported on $B_1$. Then choose $0 \leq r_1 \leq 1$ such that for $\nu(x) := \chi_{\{ r_1 \leq |x| \leq 1 \} }(x)$ we have
\begin{equation}
    \int_0^{\infty} r^{d-1} \mu(r) dr = \int_0^{\infty} r^{d-1} \nu(r) dr. \label{same_mass_ball}
\end{equation}
Then $\mu \le_{\mathrm{SH}} \nu$ on $B_1$. In addition, $v(|x|) := \Delta^{-1}(\nu-\mu)(|x|)$ is nonincreasing on $(0,1)$.
\end{proposition}

\begin{proof}
    By a similar argument as in Proposition \ref{prop:annuli_target}, we have $v(r)=0$ for $r \geq 1$. Similarly, $v$ satisfies \eqref{eq:radial_ODE}, which implies (using $v'(1)=0$) that
    \[ v'(r) = -\frac{1}{r^{d-1}} \int_r^1 s^{d-1}(\nu(s)-\mu(s)) \, ds. \]
    For $r \in (r_1, 1)$, we have $\nu(s) = 1$. Since $\mu \leq 1$, this implies $v'(r) \leq 0$ on $(r_1, 1)$. 
    For $r \in (0, r_1)$, we observe that $\nu(x) = 0$, so $\Delta v(x) = -\mu(x) \leq 0$. Since $v$ is radial and smooth at the origin ($v'(0)=0$), the minimum principle implies $v'(r) \leq 0$ on $(0, r_1)$. Combining these, we obtain $v'(r) \leq 0$ on $(0,1)$. Since $v(r)=0$ for $r \geq 1$, we see that $v(r) \geq 0$ on $B_1$ and vanishes outside. Thus Proposition \ref{prop:subharmonic equivalent} implies $\mu \leq_{\textnormal{SH}} \nu$ over $B_1$.
\end{proof}

Finally, a similar argument to Proposition \ref{prop:annuli_target} implies the following:

\begin{proposition}[Target measure on $(0,1)$] \label{prop:one_d_target} Consider $d = 1$ and a function $\mu$ with $0 \le \mu \le 1$ supported in $(0,1)$. Define $\nu(x) = \chi_{(0,a)}(x) + \chi_{(b,1)}(x)$ where $0 \leq a \leq b \leq 1$ are chosen such that
\begin{equation}
        \int_0^1 \mu(x) dx = \int_0^1 \nu(x) dx, \quad  \quad\int_0^1 x\mu(x) dx = \int_0^1 x \nu(x) dx,
 \label{one_dimension_mass_moments}
\end{equation} then $\mu \leq_{\textnormal{SH}} \nu$ with respect to $(0,1)$. In addition, $v(x) := \Delta^{-1}(\nu-\mu)(x)$ is nondecreasing on $(0,a)$ and nonincreasing on $(b,1)$.
\end{proposition}
\section{Example of nucleation in dimension $1$} \label{nuc_example_appendix}
In this appendix, we justify Proposition \ref{nuc_example}. We first establish some bounds on $w$.
\begin{lemma}[Lower bound on $w$]\label{lower_w}
Let $w(t,x)$ solve
\begin{equation}
\begin{cases}
\partial_t w - \Delta w/2 = -\nu \chi_{\{w>0\}}/2, & (t,x)\in (0,\infty)\times U, \\[2mm]
w(0,x) = \Delta^{-1}(\nu-\mu)(x), & x\in U, \\[2mm]
w(t,x) = 0, & (t,x)\in (0,\infty)\times \partial U,
\end{cases}
\end{equation}
where $\nu \ge 0$ and $U\subset \R^d$ is bounded and open.  
If $\mu \in C^2(\R^d)$ and $\|\Delta \mu\|_{L^{\infty}(U)} < \infty$, then
\[
w(t,x) \ge v(t,x)
:= w(0,x) - (t/2)\mu(x) - (t^2/8)\|\Delta \mu\|_{L^{\infty}(U)}.
\]
\end{lemma}

\begin{proof}
Observe that
\[
\partial_t v - \Delta v/2
= -\nu/2 + (t/4)\big(\Delta \mu(x) - \|\Delta \mu\|_{L^{\infty}(U)}\big)
\le -\nu \chi_{\{w>0\}}/2.
\]
Hence
\[
\partial_t(w-v) -\Delta(w-v)/2 \ge 0
\quad \text{on } (0,\infty)\times U,
\]
with
\[
(w-v)(0,x)=0, \qquad 
(w-v)(t,x) = (t/2)\mu(x) + (t^2/8)\|\Delta \mu\|_{L^{\infty}(U)} \ge 0
\quad \text{on } (0,\infty)\times\partial U.
\]
By the parabolic minimum principle, it follows that $w(t,x)\ge v(t,x)$.
\end{proof}

Next, we show that when there is no nucleation, we can obtain an upper bound on $w(t,x)$ on $\{w>0\}$ in $d=1$.

\begin{lemma}[Linearization upper bounds under no nucleation] \label{no_nuc_bounds} Let $\mu \in \mathcal{S}_0((0,1)) \cap C^2([0,1])$ and let $\nu^*$ be from Proposition \ref{target_measure_universial}. Let $w$ be the unique solution of
\[ 
\p_t w -  \p^2_{x} w/2 = -\frac{\nu^*}{2} \chi_{ \{w>0\} }  \quad \text{ in } \;(0,\infty) \times \R,  \quad\quad
w(0,x) = \Delta^{-1}(\nu^*-\mu)(x) \;\;\;\;  x\in\R.  
\]
Assume that the freezing time $s(x)$ is nonincreasing on some interval $I := (b,c) \subset (b,1)$, where $b$ is given by \eqref{one_dimension_mass_moments}, and let $T := s(c)$. Then, for any smooth function $\zeta$ satisfying $0 \leq \zeta \leq 1$ with $\zeta(b)=\zeta(c)=0$,
\[ 
w(t,x) \leq w(0,x) - (t/2) (\zeta(x)\mu(x)) + (t^2/8) \| (\zeta \mu)''\|_{L^{\infty}(I)} \quad (t,x)\in (0,T)\times I.
\]
\end{lemma}

\begin{proof}
Define $v(t,x) := w(0,x) - (t/2) (\zeta(x)\mu(x)) + (t^2/8) \| (\zeta \mu)''\|_{L^{\infty}(I)}$. Then one has
\[ 
\p_t v -\p^2_{x} v/2 = -\nu^*/2 + ((1-\zeta)/2)\mu + (t/4) \left( \|(\zeta \mu)''\|_{L^{\infty}(I)} + (\zeta \mu)'' \right) 
\]
\[ 
\geq -\nu^*/2 = \p_t w -\p^2_{x}w/2 \text{ on } (0,T) \times I. 
\]
Here we used the fact that $(0,T) \times I \subset \{w>0\}$, which follows from our assumptions. 
Also, we have that $v(0,x)=w(0,x)$ and $v(t,c)=w(0,c) + (t^2/8) \|(\zeta \mu)''\|_{L^{\infty}(I)} \geq w(0,c) \geq w(t,c)$ (recall that $\p_tw \leq 0$) and similarly $v(t,b) \geq w(t,b)$. 
Hence, we conclude by the maximum principle.
\end{proof}

We now use these two inequalities to prove Proposition \ref{nuc_example}. 
\begin{proof}[Proof of Proposition \ref{nuc_example}]

Consider the polynomial
\[ \mu(x) := 154( -x^4/4+ 7x^3/15 -  7x^2/25 + 8x/125)\; \text{ for } x \in (0,1). \] We first show that $\mu \chi_{(0,1)} \in \mathcal{S}_0((0,1))$. Motivated by Proposition \ref{target_measure_universial} we define the constants
\[ a = (12397/100500) \approx 0.12335 \;\;\;\text{ and } \;\;\;b = (30353/100500) \approx 0.30202. \] Then we claim that $\mu \chi_{(0,1)} \leq_{\text{SH}} \nu: = \chi_{(0,a)} + \chi_{(b,1)}$ over $(0,1)$, which implies $\mu \in \mathcal{S}_0((0,1))$. Note that $a$ and $b$ are chosen so that $\mu \chi_{(0,1)} $ and $\nu$ have the same mass and first moment.

 First, by direct computation, one has
\begin{equation} \int_0^x \int_0^y \mu(s) \chi_{(0,1)} (s) dsdy = (77/1500) x^3\left( 32 -70x +70x^2 - 25x^3 \right). \label{double_int_mu}   \end{equation} and
\begin{equation} \int_0^x \int_0^y \nu(s) ds = \begin{cases} x^2/2 &  x \in (0,a) \\ 
ax-a^2/2 & x \in (a,b) \\ x^2/2+ (a-b)x + (b^2-a^2)/2 &x\in (b,1)\end{cases}. \label{double_int_nu} \end{equation} Hence, we have from \eqref{double_int_mu} and \eqref{double_int_nu} that 
\[ v(x) := \int_0^x\int_0^y (\nu(s)-\mu(s) \chi_{(0,1)}(s) ) ds \geq 0  \text{ for } x \in (0,1). \] In particular, one has $v(0)=v(1)=0$ because $\mu \chi_{(0,1)} $ and $\nu$ have the same mass and first moment. In addition, we have that $\lim_{x \rightarrow 0^+}v'(x)=\lim_{x \rightarrow 1^-}v'(x) = 0$ because $\nu$ and $\mu$ have the same mass. Thus, we can continuously extend both $v$ and $v'$ to be zero outside of $(0,1)$. In particular, since $\nu = \mu + v''$, the conditions $v(0)=v(1)=v'(0)=v'(1)$ and $v \geq 0$ on $(0,1)$ imply, by integration by parts, that $\mu \leq_{\textnormal{SH}} \nu$ over $(0,1)$ (see Lemma \ref{lem:integration by parts}). In particular, by Proposition \ref{target_measure_universial}, $\nu$ is the optimal target measure for $\mathcal{P}(\mu,u)$. 

We now claim that, for this choice of $\mu$, that $s(x)$ fails to be nonincreasing over $(b,0.8)$. To see this, we argue by contradiction by using Lemmas \ref{no_nuc_bounds} and \ref{lower_w}. First, we compute
\begin{equation} \|\mu''\|_{L^{\infty}((0,1))} = 2926/25, \label{mu_double_bound} \end{equation} and we choose the cutoff function
\[ \zeta(x) := C(x-b)(0.8-x), \] where $1/C \approx 0.062$ is chosen so that $\|\zeta\|_{L^{\infty}( (b,.8))} = 1$. Then this implies that
\begin{equation} \| (\zeta\mu)''\|_{L^{\infty} ( (b,0.8))} \leq 71. \label{cut_off_2nd} \end{equation} We now show that  Lemmas \ref{no_nuc_bounds} and \ref{lower_w}, together with the assumption that $s(x)$ is nonincreasing over $(b,0.8)$ imply  $s(0.5) < s(0.8)$, a contradiction. 

If $s(x)$ is nonincreasing on $(b,.8)$, then we obtain by Lemma \ref{no_nuc_bounds} and Lemma \ref{lower_w} combined with the bounds \eqref{mu_double_bound} and \eqref{cut_off_2nd}, recalling $v(x)=w(0,x)$, that for $t \leq s(0.8)$ and $x \in (b,0.8)$,
\begin{equation} v(x) - (t/2)(\zeta(x)\mu(x)) +  (71/8)t^2  \geq w(t,x) \geq v(x) - (t/2) \mu(x) - (1463/100)t^2. \label{linearization_bounds_w123} \end{equation}  Here the lower bound is always valid, while the upper bound is only valid until $t=s(0.8)$.

From the explicit formulae for $v$ and $\mu$, the lower bound in \eqref{linearization_bounds_w123} at $x=0.8$ implies
\[ s(0.8) \geq \inf_{t \geq 0} \left\{ t:v(0.8) - (t/2) \mu(0.8) - (1463/100)t^2 = 0\right\} \geq 0.006.  \] Consequently, the upper bound in \eqref{linearization_bounds_w123} is valid until at least $T=0.006$. But observe that the upper bound in \eqref{linearization_bounds_w123} at $x=0.5$ implies from the explicit formula of $v$ and $\mu$
\[ \min\{0.006,s(0.5)\} \leq \inf_{t \geq 0} \{ t:v(0.5) - (t/2)\zeta(0.5)\mu(0.5) + (71/8)t^2 =0  \} \leq0.0025. \] This contradicts the assumption that $s(x)$ is nonincreasing on $(b,0.8)$. 
\end{proof}

\end{myappendices}

\bibliographystyle{alpha}
\bibliography{bib}
\end{document}